\documentclass{amsart}
\usepackage{graphicx}

\usepackage{amsmath}
\usepackage{amssymb}
\usepackage{amsthm}
\usepackage{tikz-cd}
\usepackage{float}
\usepackage{hyperref}
\usepackage{fullpage}
\usepackage{url}
\usepackage{mathrsfs}
\usepackage{newtxtext,newtxmath} 
\newcommand{\R}{\mathbb{R}}

\newcommand{\Z}{\mathbb{Z}}

\newcommand{\h}{\mathbb{H}}
\newcommand{\E}{\mathbb{E}}

\newtheorem{theorem}{Theorem}[subsection]
\renewcommand{\thetheorem}{
	\ifnum\value{subsection}>0
	\thesubsection
	\else
	\thesection
	\fi
	.\arabic{theorem}
}
\newtheorem{lemma}[theorem]{Lemma}
\newtheorem{corollary}[theorem]{Corollary}
\newtheorem{proposition}[theorem]{Proposition}

\theoremstyle{definition}
\newtheorem{definition}[theorem]{Definition}

\theoremstyle{remark}
\newtheorem{remark}[theorem]{Remark}

\begin{document}
    
	\title{Finite-State Machines for Horospheres in Hyperbolic Right-Angled Coxeter Groups
	}

	\author{Noah Jillson \and Daniel Levitin \and Pramana Saldin \and Katerina Stuopis \and Qianruixi Wang \and Kaicheng Xue}

	
	\maketitle
	\thanks{Daniel Levitin, Noah Jillson, Pramana Saldin, and Katerina Stuopis were supported by National Sciences Foundation grant 2230900 during this project. Daniel Levitin was further supported by National Sciences Foundation grant 1552234.}
	
	\begin{abstract}
		Relatively little is known about the discrete horospheres in hyperbolic groups, even in simple settings. In this paper we work with hyperbolic one-ended right-angled Coxeter groups and describe two graph structures that mimic the intrinsic metric on a classical horosphere: the Rips graph and the divergence graph (the latter due to Cohen, Goodman-Strauss, and Rieck (Ergodic Theory and Dynamical Systems 42(9):2740-2783, 2022)). We develop, analyze, and implement algorithms based on finite-state machines that draw large finite portions of these graphs, and deduce various geometric corollaries about the path metrics induced by these graph structures.
	\end{abstract}
	
	\section{Introduction}
	\label{sec:intro}
	
	\subsection*{Horospheres}
    
    If $(X,d)$ is an unbounded geodesic metric space, and $\gamma$ is a geodesic ray, (i.e. $\gamma:[0,\infty)\to X$ is an isometric embedding), the \textit{Busemann Function} associated to $\gamma$, $b_\gamma:X\to \R$ is defined to be $$\lim_{t\to\infty} d(x, \gamma(t))-d(\gamma(t), x_0),$$ where $x_0$ is a fixed base point in $X$. While this function depends on $x_0$, the level sets, called  \textit{horospheres}, do not. In the case where $(X,d)$ is the hyperbolic space $\h^n$, then regardless of the choice of ray $\gamma$ or value for the Busemann function, the space $b_\gamma^{-1}(r)$ with its intrinsic metric is isometric to the Euclidean space $\E^{n-1}$. These horospheres are exponentially distorted relative to their extrinsic metric. One views the ray $\gamma$ as encoding a point at infinity on the \textit{boundary sphere} $S^{n-1}=\partial \h^n$. The points on this sphere are directions in $\h^n$, and the Euclidean space $\E^{n-1}$ arises naturally as the stereographic projection of the sphere.
	
	In this paper, we study horospheres in the Cayley graphs of certain \textit{Right-Angled Coxeter Groups} (RACGs) with respect to their standard generating sets. We will be interested in graph structures on these horospheres that provide some notion of an intrinsic metric on a non-smooth space. It is a straightforward calculation that unequal words on a horosphere are at distance at least $2$ from one another, and therefore do not span an edge (see Lemma \ref{CalculatingBusemannFunctions}). Therefore, we will treat two notions of an intrinsic graph structure, neither of which are the induced subgraph. First we discuss the \textit{k-Rips graph}, which connects two vertices exactly when they are at distance at most $k$ in the ambient graph metric. One imagines this graph allowing them to traverse a short distance along the horosphere as a coarse analogue of requiring a path to be everywhere tangent to the horosphere in the smooth case. Second, we treat (under an additional assumption) the \textit{divergence graph} due to Cohen, Goodman-Strauss, and Rieck \cite{CGSR}. Edges in this graph have a more complicated description, and encode the existence of certain parallel lines. See Definition \ref{DivergenceGraphDef}. In both cases, we are interested in the induced path metric in the resulting graph. These graphs are infinite and difficult to study in general, but under minor additional assumptions, we give algorithms to draw large finite pieces of both of these graphs in asymptotically optimal time. We restrict to considering \textit{simple alternating rays}, i.e. $\gamma=a_ia_ja_ia_j ...$, where $a_i$ and $a_j$ are generators. As we will discuss further on, this assumption is more for convenience than any major limitation on our work.
	
	\begin{theorem} \label{IntroMainThm}
		
		Let $\Gamma$ be a graph satisfying the assumptions in Definition \ref{StandingAssumptions}, and let $W_\Gamma$ be the associated RACG. Let $\gamma$ be a simple alternating ray, and consider the set $H_k=b_\gamma^{-1}(k)$ for any integer $k$. A subgraph of either the $2$-Rips graph or the divergence graph on $H_k$ with $n$ vertices can be generated in runtime $O(n\log(n))$. This time complexity is optimal.
		
	\end{theorem}
	
	See Theorems \ref{RipsGraphGenerationSpeed} and \ref{DivergenceGraphGenerationSpeed} for further details, and Proposition \ref{EveryStateLarge} for the additional assumption in the divergence graph case.
	
	Optimality is unsurprising: we have $n$ vertices, and a positive proportion of their labels have length that is linear in $\log(n)$. So just listing each vertex already requires $\Omega(n\log(n))$ steps.
	
	For the Rips graph case, our consideration of just the $2$-Rips graph is justified by the following result, which is proved in Proposition \ref{RipsGraphBiLipschitzEquivalence}. Recall that two metric spaces $(X, d_X)$ and $(Y, d_Y)$ are \textit{bi-Lipschitz equivalent} if there is a bijection $f:X\to Y$ between them and a constant $L>1$ so that $\frac{1}{L} d_X\le d_Y\circ f\le Ld_X$.
	
	\begin{proposition}
		
		For any $k_1, k_2\ge 2$, the induced path metrics on the $k_i$-Rips graphs on any horosphere are bi-Lipschitz equivalent.
		
	\end{proposition}

        This says that the large-scale geometry of the $k_1$ and $k_2$ Rips graphs should be comparable. So we may as well study the simplest one.
    
    \subsection*{Hyperbolic groups and their boundaries}
    
        The most important of the assumptions on the graph $\Gamma$ for Theorem 1.1 is that there are no induced subgraphs that are $4$-cycles. This is equivalent to the Cayley graph of $W_\Gamma$ being $\delta$-\textit{hyperbolic}, i.e. displaying large-scale features similar to that of the hyperbolic space \cite{GromovHyperbolic,Moussong}. This assumption underlies all the algorithms we write, allowing computations in RACGs to be carried out quickly and with minimal memory usage in relatively simple finite-state machines (FSMs). See Definition \ref{FSMDef}. All hyperbolic groups (i.e. those with $\delta$-hyperbolic Cayley graphs) admit FSMs that perform computations, but they are very difficult to describe or construct in practice \cite{WordProcessing}. Therefore, we use machines due to Brink-Howlett and Antolin-Ciobanu \cite{BrinkHowlett, AntolinCiobanu}, which are much easier to construct. These latter FSMs exist for RACGs in general, but it is only for the hyperbolic RACGs that they are useful for performing computations in horospheres. The bulk of our technical task consists of reducing geometric questions to properties of (pairs of) words, and then checking these properties algorithmically using modifications of these two FSMs.
	
	While Theorem \ref{IntroMainThm} is phrased for a fixed generating graph $\Gamma$ of the group $W_\Gamma$, a caveat is in order. The complexity of these algorithms is exponential in terms of the largest clique in the graph $\Gamma$. This fact, though inconvenient, is expected. It is a straightforward consequence of Papasoglu's theorem on thin bigons that the groups $W_\Gamma$ become less hyperbolic as the clique size of $\Gamma$ increases \cite{Papasoglu1995}. Since all of our algorithms require hyperbolicity to work, it is unsurprising that they take longer as $W_\Gamma$ becomes less hyperbolic.
	
	Besides their algorithmic properties, it is a matter of mathematical folklore that some properties of the horospheres in $\h^n$ should carry over to hyperbolic groups. The Cayley graphs of these groups have boundaries $\partial W_\Gamma$ consisting of (equivalence classes of) rays, equipped with a suitable metric, but these boundaries can be topologically and metrically much more complicated than the round spheres at the boundary of $\h^n$. Nevertheless, it is expected that there should be a topological and geometric analogy between these discrete horospheres and the stereographic projections of the boundary.
	
	The geometry of the boundary of a $\delta$-hyperbolic group is of particular interest, for instance, because of its value to Gromov's program of classifying finitely-generated groups up to quasi-isometry \cite{GromovAsymptotic}. See chapter III.H.3 of \cite{BridsonHaefliger} for more details. Nevertheless, to our knowledge there has been little work to understand what an appropriate intrinsic metric on these discrete horospheres might be, and what it might encode about the metric structure of the boundary.
    
    \subsection*{Geometry of horospheres in RACGs}
    
    For the groups $G$ with the strongest analogy to $\h^n$, the fundamental groups of closed hyperbolic $n$-manifolds (and more generally cocompact lattices in semisimple Lie Groups), there is a bi-Lipschitz model of $G$ in which the horospheres (with respect to the Lie Group or symmetric space metric) receive a cocompact action by a group $\Z^{n-1}$ (or more generally a discrete nilpotent group), with exponentially distorted orbits \cite{PengitoreMcReynolds}. This action is \textit{translation-like}, i.e. geometrically similar to right multiplication by the elements of a subgroup. However, it is unclear how much of this translates to horospheres with respect to the metric on the Cayley graph.
	
	On this topic, we show connectivity and exponential distortion for the k-Rips graph with $k\ge 2$. We note also that the connectivity of the Divergence Graph was already known \cite{CGSR}.
	
	\begin{proposition}
		
		Let $W_\Gamma$ be a hyperbolic right-angled Coxeter group whose boundary, $\partial W_\Gamma$, is connected. The $2$-Rips graph on the horosphere $b_\gamma^{-1}(k$) is connected, for $\gamma$ a repeating ray.
		
	\end{proposition}
	
	See Proposition \ref{RipsGraphOriginConnectivity}. Following the analogy between horospheres and stereographic projections of the boundary, this proof breaks into cases based on the connectivity of \newline $\partial W_\Gamma\setminus \gamma(\pm\infty)$, where $\gamma(\pm\infty)$ refers to the endpoints of the defining line $(...a_ia_ja_ia_j...)$ on the boundary. See Proposition \ref{HorosphericalDistortionUpperBound} and its proof for further details. Along the way, in Remark 4.5.3, we give a description of the set of path components of $\partial W_\Gamma\setminus\{\gamma(\pm\infty)\}$, which to our knowledge has not yet appeared in the literature without additional assumptions (for RACGs with no cliques larger than edges,\cite{DaniThomas,ThetaGraphs,Edletzberger}).
	
	Putting the connectivity statement together with some general facts about the divergence functions of a hyperbolic group then yields the conclusion that the Rips graph is exponentially distorted.
	
	\begin{theorem}
		
		Let $W_\Gamma$ be a hyperbolic right-angled Coxeter group with connected boundary. There exist exponential functions $f_1$ and $f_2$ such that the following holds. If $\gamma=(a_ia_j)^\infty$ is any periodic ray, $w$ and $v$ are in $b_\gamma^{-1}(k)$, and if $d$ denotes the metric on $W_\Gamma$ while $d_H$ denotes the path metric on the $2$-Rips graph, then 
		
		$$ f_1(d(w, v))\le d_H(w, v)\le f_2(d(w, v)).$$
	\end{theorem} 
	
	See Theorem \ref{RipsGraphDistortion} for further details. It is an immediate corollary that the divergence graph inherits an exponential lower bound on distortion, due to a result in \cite{CGSR} that says that the divergence graph is a subgraph of a $k$-Rips graph for a sufficiently large $k$, though we obtain a better bound in the setting of RACGs. It is not obvious how to prove an upper bound for divergence graphs in general. However if $\Gamma$ satisfies a technical condition that holds, e.g., for any triangulation of a closed manifold, then an upper bound can be deduced without trouble. See Proposition \ref{DivergenceDistortionUpperBound}.
	
	As a consequence of the lower bound, we further conclude that each Rips and divergence graph is a graph of polynomial growth.
	
	\begin{corollary}
		
		There is a polynomial $P$ depending on $\Gamma$ and $\{a_i, a_j\}$ so that a certain ball in the $2$-Rips graph grows with rate $|B(w_0, r)|\le P(r)$. The same holds (with a possibly-different polynomial) in the divergence graph.
		
	\end{corollary}
	
	See Corollaries \ref{RipsGraphGrowth} and \ref{DivergenceGrowth} for more information. Note that by Proposition 1.2, this then holds with a different polynomial for each other Rips graph. Also, the triangle inequality shows that $|B(w, r)|\le P(r+d(w,w_0))$. So every ball grows at a potentially-different polynomial rate.
	
	While we cannot presently give any nilpotent structure to the horospheres as in \cite{PengitoreMcReynolds}, polynomial growth can be seen a first step toward such a structure. More precisely, if we could find a bi-Lipschitz equivalent graph structure that is coarsely vertex-transitive, then a theorem of Trofimov would tell us that the horospheres are approximately Cayley graphs of nilpotent groups \cite{Trofimov}. This would be a major step toward understanding the geometry of hyperbolic RACGs. A discrete space foliated by a $\Z$-indexed family of exponentially-distorted (approximate) nilpotent Cayley graphs would also match with Heintze's description of homogeneous negatively-curved Riemannian manifolds as semidirect products $N\rtimes \R$ where $N$ is a simply-connected nilpotent Lie group and $\R$ acts on $N$ by an expansion, so that each coset of $N$ is exponentially distorted \cite{Heinze}.

    \subsection*{Divergence graphs}

    The divergence graph was first defined by \cite{CGSR} as part of the authors' construction of strongly aperiodic subshifts of finite type on any one-ended hyperbolic group. Informally, these are finite collections of local rules so that, if a coloring of the Cayley graph satisfies these rules everywhere, then it has no global symmetries. The divergence graph plays a key role in their construction by showing that, with a finite collection of rules, one can place two incompatible structures on the horospheres of a group, which the authors use to rule out infinite-cyclic symmetry groups.

    The divergence graph is of interest for several reasons beyond dynamics. First of all, it is a finer object than the Rips graph. It encodes information about the extrinsic geometry of the horosphere in question as well as each subsequent horosphere, as well as information about the combinatorial structure of the hyperbolic group and a chosen order on the generators. The divergence graph is also more clearly analogous to the boundary at infinity because of their shared connection to geodesic rays. If one imagines a point in the Cayley graph as having some diameter, and placing a light source at one end of a ray, then the divergence graph encodes the pattern that the shadows of points make on the boundary. As we shall see in Section \ref{sec:Pictures}, the divergence often provides a clearer picture of the topology of the boundary than the Rips graph. The cost to pay for this extra information is in added technical difficulties. While our main results show that the two graphs are geometrically comparable, the proofs for the divergence graph are considerably harder. One may think of the added difficulty as coming in part from the fact that the divergence depends on an ordering on the generating set, but the geometric properties we wish to deduce do not.
	
    \subsection*{Further remarks}
    
    One difficulty with studying horospheres in the Cayley graph is that they are, even as sets, not particularly invariant objects. A choice of a different generating set for the group $G$ may render the ray $\gamma$ no longer geodesic. Even if this does not happen, changes of generating set may scramble horospheres. For instance, consider the group $F_2=\langle a, b\rangle$ with generating sets $S_1=\{a,b\}$, $S_2=\{a, b, b^3\}$. It is a straightforward exercise to compute that if we take $\gamma(n)=a^n$ and $x_0$ to be the identity, then with respect to $S_1$, $b_\gamma(a^mb^m)=0$, while with respect to $S_2$, $b_{\gamma'}(a^mb^m)\approx \frac{-2m}{3}$. Another advantage of studying RACGs, then, is that they come equipped with a preferred generating set, and thus a preferred Cayley graph and metric.
	
	With minor modifications, the algorithms should all work for Rips and divergence graphs defined with respect to any repeating ray, and yield similar geometric properties. Without the assumption that the ray is an infinite power, there is no hope of obtaining FSM-based algorithms to draw horospheres. The geometric statements should still be true without the repeating assumption, because they are mostly consequences of the algorithms in abstract rather than the FSM implementations. However, all of the statements would become considerably more technical and less enlightening. Therefore we have chosen to focus on the simple case for this paper.
	
	In the case of the 2-Rips graph and the divergence graph, the code is available on Daniel Levitin's github, at \url{https://github.com/dnlevitin/horospheres}. The Main branch has the code used to generate the figures in this paper.
	
    \subsection*{Outline}

    The outline of the paper is as follows. In Section \ref{sec:preliminaries}, we define the basic objects of study. Section \ref{sec:FSMs} lays out the finite-state machines that we will use and combine repeatedly throughout the remainder of the paper. Sections \ref{sec:RipsGraph} and \ref{sec:DivergenceGraph} treat the $2$-Rips graph and the divergence graph, and follow a similar structure. First we introduce a normal form on the vertices that we will use for our computations. We then describe an algorithm to find edges between words whose normal forms have the same lengths, and then an algorithm to find edges between words whose normal forms have different lengths. Both sections conclude by proving the geometric corollaries of our description of the edge set, including distortion estimates and polynomial growth. In Section \ref{sec:Pictures}, we give graphical examples that demonstrate the outputs of the algorithms we have described.
	
	\section{Preliminaries}
	\label{sec:preliminaries}

	In this section, we will lay out background definitions from geometric group theory.
	
	\begin{definition}
		
		Let $G$ be a group with generating set $S$. They \textit{Cayley Graph} of the $G$ with generating set $S$, denoted $Cay(G,S)$, is the graph whose vertex set is $G$ and with edges $(g,gs)$ for each $g\in G$ and $s\in S$. We endow $Cay(G,S)$ with a metric $d_S$ by giving each edge length $1$ and taking the path metric. We refer to the restriction of this metric to the vertex set (i.e. to $G$) again as $d_S$. Note that on $G$ this metric is given by the formula $d_S(g_1,g_2)=|g_1^{-1}g_2|_S$, where $|\cdot|_S$ denotes the word length in terms of $S$. We will call $d_S$ the \textit{word metric (with respect to S)}.
		
	\end{definition}
	
	This allows us to associate a metric space to a group, and indeed a family of metric spaces. Without going into great detail, there is a natural notion of equivalence among the various metrics $d_S$ arising from different choices of generating sets. The large-scale geometric properties of the Cayley graph will usually turn out to be independent of $S$
    
	When describing segments and rays in Cayley graphs, we will consider either edge paths in the Cayley graph or sequences in the group. It will usually not matter which we consider. 
	
	\begin{definition}
		
		We will say that a word in the generating set is \textit{geodesic} if the associated edge path starting at the identity in the Cayley graph is distance-minimizing. If the generating set $S$ carries an alphabetical order, then a word is a \textit{shortlex geodesic} if it is a geodesic and is alphabetically first among all geodesics with the same start and endpoint.
		
	\end{definition}
	
	One example of a crucial geometric property of a group that does not depend on choice of generating set is that of negative curvature, which is formalized in the following definition.	
	
	\begin{definition}
		
		Let $G$ be a group and $S$ a generating set. $G$ is said to be $\delta$-hyperbolic for a non-negative real number $\delta$ if for any 3 points $x_1$, $x_2$, and $x_3$ in the associated Cayley graph of $G$, and any geodesic segments $\gamma_1=[x_1,x_2]$, $\gamma_2=[x_2, x_3]$, and $\gamma_3=[x_3,x_1]$, then any $\gamma_i$ is in the $\delta$-neighborhood of the union of the other two.
		
	\end{definition}
	
	We will say that $G$ is \textit{hyperbolic} if it is $\delta$-hyperbolic for some $\delta$ and some choice of generating set. It is well-known that a change of generating set can change the value of $\delta$, but that if $G$ is $\delta$-hyperbolic with respect to some finite generating set $S$, it will be hyperbolic with respect to any other finite generating set \cite{BridsonHaefliger}.

	In this paper, we will study right-angled Coxeter Groups and their Cayley graphs with respect to a special generating set. 
	
	\begin{definition} \label{RACGDef}
		
		Let $\Gamma$ be a graph, with vertex set $V=\{a_i\}$ and edge set $E$. The \textit{right-angled Coxeter group} (RACG) determined by $\Gamma$, denoted $W_\Gamma$, is the group with presentation $\langle V| R\rangle$ where $R$ is the set of words $\{a_i^2, [a_i,a_j]: (a_i,a_j)\in E\}$.
		
	\end{definition}
	
	We think about these relations as allowing us to cancel any two adjacent copies of the same letter, and otherwise to reverse the order of an adjacent pair of letters exactly when the associated vertices span an edge. When we perform a sequence of such swaps, we will say that this is a \textit{rearrangement} of the word.
	
	We will be interested in the Cayley graphs for the given generating set, and we will order the generators by subscripts, i.e. $a_1<a_2<...$. When we use the symbol $d$ for distance in a RACG, we always mean with respect to the generating set $V$. Similarly, the expression $|w|$ means the length of any geodesic word equivalent $w$ in the generating set $V$.
	
	We fix some notation in the graph $\Gamma$ We will denote $Star(a_i)$ to be $a_i$ together with the vertices adjacent to $a_i$. The sets $Star_<(a_i)$ and $Star_>(a_i)$ are those adjacent vertices earlier and later than $a_i$ respectively, and $Star_\le(a_i)$ and $Star_\ge(a_i)$ are $Star_<(a_i)$ and $Star_>(a_i)$ together with $a_i$. $Link(a_i)$ means the set of vertices adjacent to $a_i$, and may be decorated with a subscript as before. Finally, $Clique(\Gamma)$ is the size of the largest clique in $\Gamma$.
	
	For a RACG, a great deal is known about its properties just in terms of the defining graph. For example hyperbolicity can be read off from the defining graph.
	
	\begin{theorem} \cite{Moussong}
		
		Let $\Gamma$ be a graph without induced square subgraphs. Then $W_\Gamma$ is hyperbolic.
		
	\end{theorem}

	\begin{definition}[Standing Assumptions] \label{StandingAssumptions}
		
		We make the following standing assumptions of all of our defining graphs $\Gamma$.		
		\begin{itemize}
			\item $\Gamma$ has no induced square subgraphs, i.e. $W_\Gamma$ is hyperbolic.
			\item $\Gamma$ is not a complete graph.
			\item $\Gamma$ has no separating cliques, i.e., no (possibly empty) complete subgraphs $K$ so that $\Gamma\setminus K$ is disconnected. In particular $\Gamma$ is connected.
		\end{itemize}	
	\end{definition}

	These assumptions are necessary and sufficient for the group $W_\Gamma$ to be one-ended and hyperbolic \cite{Davis}.
	
	\begin{definition}
		
		Let $\gamma$ be a geodesic ray in a hyperbolic RACG. The \textit{Busemann function} determined by $\gamma$ is the function $b_\gamma:W_\Gamma\to \Z$ defined by $b_\gamma(w)=\lim_{n\to\infty} d(\gamma(n), w)-d(\gamma(n), e)$ where $e$ denotes the identity.
		
		A \textit{horosphere} for $\gamma$ is a level set of $b_\gamma$. We will refer to the $k$-\textit{horosphere} to mean $b_\gamma^{-1}(k)$.
		
	\end{definition}
	
	We will restrict our consideration for convenience to horospheres in hyperbolic RACGs about the rays that we will colloquially refer to as $\gamma = (a_ia_j)^\infty$. That is, for any two vertices $a_i$ and $a_j$ not spanning an edge in the defining graph, this geodesic starts at the origin in the Cayley Graph, then traverses the edge to $a_i$, then to $a_ia_j$, and so on forever. It is straightforward to verify that this really is a geodesic.
	
	We will be interested in putting graph structures on these horospheres. The first such graph structure we define here. The second one is more complicated and will be defined later.
	
	\begin{definition}
		
		Let $(X,d)$ be a metric space. The $k$-\textit{Rips graph} on $X$ is the graph whose vertex set is $X$ and where edges connect points at $d$-distance at most $k$.
		
	\end{definition}
	
	For us, $X$ will be a horosphere, and $d$ will be the restriction of the word metric.

        \subsection*{Finite-state machines and formal language theory}
	
	For our computations, we will use a formalism called a Finite-State Machine (FSM).
	
	\begin{definition} \label{FSMDef}
		
		A \textit{Finite-State Machine} $M$ consists of the following
		
		\begin{itemize}
			\item A finite, labeled, directed graph $G=(V,E)$ (whose vertices are called \textit{states}), whose edges are labeled in the alphabet $\mathscr{A}$ such that every vertex has exactly one directed edge exiting it labeled by each letter in $\mathscr{A}$.
			\item A special vertex $v\in V$ called the \textit{starting state}	
			\item A subset $A$ of the vertices of $V$ called the \textit{accepted states}.
		\end{itemize}
		
		We write $v_1\to_{a} v_2$ if there is an edge labeled $a$ from $v_1$ to $v_2$.
		
		Given such a machine $M$, the associated \textit{regular language} $\mathscr{L}(M)$ consists of all strings of letters in $\mathscr{A}$ labeling paths from the starting state to an accepted state.
		
	\end{definition}
	
	We will usually not describe FSMs in this level of detail because it is somewhat cumbersome. The following remarks provide some simplifications we will usually use tacitly.
	
	\begin{remark}\label{PrefixClosedLanguages}
		
		Many of the languages $\mathscr{L}(M)$ we consider will be \textit{prefix-closed}, i.e. if $w=a_{i_1}a_{i_2}...a_{i_n}$ is an accepted word, then so are all words $a_{i_1}a_{i_2}...a_{i_k}$ for $k<n$. In such a case, every non-accepted state leads only to non-accepted states. Therefore, we can remove all non-accepted states and the edges to and from them, and get a graph where the words labeling directed paths are again $\mathscr{L}(M)$. When we describe an FSM without specifying accepted states or where not every vertex has an outgoing edge with every label, we mean that we are describing the accepted subgraph of a machine $M$ associated to a prefix-closed language. The full graph $G$ can be recovered by adding a single rejected vertex into the graph, and directing each missing edge to go to this vertex.
		
	\end{remark}
	
	Sometimes, we will describe a larger set of of vertices and edges than strictly necessary.
	
	\begin{remark}\label{DescribingFSMs}
		
		When we describe a finite-state machine, especially if it is derived from another finite-state machine, we will specify a vertex set, labeled edge set, and starting state. It will often be the case that from the starting state, many vertices cannot be reached by a directed path. For instance, when combining machines $M_1$ and $M_2$ that perform similar functions, we may take a vertex set to be $M_1\times M_2$, even though there may be pairs of vertices which contain mutually incompatible data. When this happens, we implicitly intend to consider only the subgraph consisting of accessible vertices and the edges between them. 
		
	\end{remark}
	
	A note on notation: the states in an FSM $M$ will often be defined in terms of data that we want to use later. By a slight abuse of notation, if $w$ is in $\mathscr{L}(M)$, then $M(w)$ will mean the final state that the word $w$ ends on in the FSM $M$. 
	
	The motivation for defining regular languages is that we will often prefer to think of a machine in terms of the language it accepts. For instance, the following proposition is convenient to phrase as a combination theorem for regular languages rather than one about finite-state machines.
	
	\begin{proposition}\label{CombiningFSMs}
		
		Suppose $\mathscr{L}_1$ and $\mathscr{L}_2$ are regular languages. Then the languages \newline $\mathscr{L}_1\cap \mathscr{L}_2$, $\mathscr{L}_1\cup \mathscr{L}_2$, and $\mathscr{L}_1\mathscr{L}_2$ are all regular, where the last expression refers to the set of concatenations of words in $\mathscr{L}_1$ with those in $\mathscr{L}_2$.
		
		Furthermore, if $\mathscr{L}$ is a language with alphabet $\mathscr{A}$ and $\mathscr{B}$ is a subset $\mathscr{A}$, the language $\mathscr{L}|_{\mathscr{B}}$ consisting of words in $\mathscr{L}$ containing only letters in $\mathscr{B}$, is regular.
		
	\end{proposition}
	
	For a proof, see any standard text on formal language theory, such as \cite{WordProcessing}.
	
	Here is one standard language we will use in these combinations.
	
	\begin{lemma}
		
		Let $\mathscr{B}$ be a subset of $\mathscr{A}$. Then the language $F_\mathscr{B}$, consisting of words in $\mathscr{A}$ that cannot be rearranged to begin with a letter of $\mathscr{B}$, is regular\footnote{More precisely, this is the set of words in $\mathscr{A}$ that cannot be made to begin with a letter of $\mathscr{B}$ without cancellation. We will always combine these languages with others that guarantee no cancellation occurs, so that the resulting words cannot begin with letters of $\mathscr{B}$.}.
		
	\end{lemma}
	
	\begin{proof}
		
		We define the relevant machine $M_\mathscr{B}$ whose vertices are subsets of $\mathscr{B}$ and whose starting state is $\mathscr{B}$. From each state $S$, we define edges leaving $S$ labeled by each letter not in $S$ by $S\to_{a_i} S\cap Star(a_i)$. As in Remark \ref{PrefixClosedLanguages} these states are all accepted, and every other edge leads to an inescapable rejected state.
		
		One sees by induction that for a word $w$, $M_{\mathscr{B}}(w)$ is the set of letters that, if written after $w$, could commute to the beginning. Therefore, the fact that every state has edges exiting it labeled by its complement means that the associated language is $F_{\mathscr{B}}$ as desired.
	\end{proof}
	
	From time to time we will need to restrict to words of odd or even length.
	
	\begin{lemma}
		
		The languages $Odd(\mathscr{A})$ of odd-length words and $Even(\mathscr{A})$ of even-length words are regular\footnote{Again, these are without canceling letters, and we will always combine this language with others to guarantee no cancellation. As it happens, in a RACG, letters cancel one pair at a time and therefore preserve the parity, see \cite{Davis}}.
		
	\end{lemma}
	
	\begin{proof}
		The associated machines have two vertices and every edge goes from the other. In one, the start state is accepted and the other is rejected, and in the other it's the other way around.
	\end{proof}
	
	\section{Finite State Machines for Hyperbolic Right-Angled Coxeter Groups}
	\label{sec:FSMs}
	
	As mentioned in the introduction, one advantage to working with RACGs is that they admit convenient computations in FSMs. For instance, RACGs, hyperbolic or otherwise, have a finite-state machine that accepts the set of all geodesics, and a separate finite-state machine that accepts the set of all shortlex geodesics. We make the following convention: when applied to words, the symbol $``="$ means that the expressions on both sides of the symbol are identical words, while the symbol $``=_{Geo}"$ means that the expression on the right of the symbol is a geodesic word equivalent to the expression on the left.

    The geodesic machine is due to Antolin-Ciobanu \cite{AntolinCiobanu}, though it is also an almost immediate generalization of Brink-Howlett in the RACG case. We include the proof for completeness.
    
	\begin{proposition} [\cite{AntolinCiobanu} Proposition 4.1]
		
		Let $\Gamma=(V,E)$ be a graph. Then $Geo(W_\Gamma)$, the set of geodesics in the alphabet $V$ subject to the commutation and cancellations in $W_\Gamma$, is a regular language. The associated FSM $M_{Geo}$ can be taken to have states given by sets $S\subset V$ where \newline $M_{Geo}(w)=\{a_i: |wa_i|=|w|-1\}=\{a_i: w=va_i \text{ and } |v|=|w|-1\}$, and edges leaving state $S$ labeled by all the elements of $S^c$. The starting state is the empty set. As the geodesic language is prefix-closed, all vertices are accepted states and unmentioned edges lead to a single rejected state as in Remark \ref{PrefixClosedLanguages}.
		
	\end{proposition}
	
	In more down-to-earth terms, the finite-state machine takes states which keep track of, for a word $w$, which letters $a_i$ can commute to the end of $w$.
	
	\begin{proof}
		
		First of all, we show that such a finite-state machine is well-defined, that is, that we can determine $M_{Geo}(wa_i)$ entirely from $M_{Geo}(w)$. This is clear for the trivial word. If $M_{Geo}(w)$ is the collection of letters that $w$ could end with after rearrangement for all words of length at most $n$, and if $wa_i$ is geodesic, then $wa_i$ can end with $a_i$, or any last letter of $w$ that also commutes with $a_i$. In particular, $M_{Geo}(wa_i)=\bigl(M_{Geo}(w)\cap star(a_i)\bigr)\cup\{a_i\}$. This shows that a machine exists whose states are subsets of $V$ as required.
		
		We now must show that the language of this machine is $Geo(W_\Gamma)$. Suppose that $w$ is a geodesic word, but that $wa_i$ is not geodesic. By the triangle inequality, $|wa_i|=|w|$ or $|wa_i|=|w|-1$. By a variant of the Exchange Condition described in \cite{Davis} we determine that $|wa_i|=|w|-1$. In particular, $wa_i=_{Geo}v$ where $|v|=|w|-1$. Then by right multiplication, $w=_{Geo}va_i$, so that $a_i$ is a last letter of $w$. Hence if $a_i\not\in M_{Geo}(w)$, then $wa_i$ is geodesic. Therefore, if a word $w=a_{i_1}a_{i_2}...a_{i_n}$ is accepted by $M_{Geo}$ and is not geodesic, it must be because $v=a_{i_1}a_{i_2}...a_{i_{n-1}}$ was not geodesic. Iterating, one sees a contradiction, showing that that all accepted words are geodesics. 
		
		One shows by induction on length that  $\mathscr{L}(M_{Geo})$ contains all the geodesics in $W_\Gamma$. Certainly $\mathscr{L}(M_{Geo})$ contains the empty word and all words of length $1$. If $w=a_{i_1}a_{i_2}...a_{i_n}$ is a geodesic word of length $n$, and $\mathscr{L}(M_{Geo})$ contains all geodesic words of length at most $n-1$, then $v=a_{i_1}a_{i_2}...a_{i_{n-1}}$ is an accepted word for $M_{Geo}$, and since $w$ is geodesic, $a_{i_n}$ is not in $M_{Geo}(v)$, so that there is an edge labeled by $a_{i_n}$ exiting $M_{Geo}(v)$. Therefore, $w$ is in $\mathscr{L}(M_{Geo})$. 
	 \end{proof}
	
	The statement for the shortlex language is very similar. This is a specialization of Brink and Howlett's shortlex machine to the RACG case \cite{BrinkHowlett}. Again, we include the proof for completeness.
	
	\begin{proposition} [\cite{BrinkHowlett} Proposition 3.3]
		
		Let $\Gamma$ be a graph, and fix an order $<$ on the $V$. The language $ShortLex(W_\Gamma)$, consisting of all shortlex geodesics, is regular. An FSM $M_{lex}$ for this language can be taken to have states labeled by subsets $S\subset V$ where $S(w)=\{l_i: wl_i \text{is not shortlex}\}$, and edges leaving $S(w)$ corresponding to $S(w)^c$.
		
	\end{proposition}
	
	\begin{proof}
		
		Again we show that we can determine $M_{lex}(wa)$ from $M_{lex}(w)$. Suppose $wa$ is shortlex and $waa''$ is not. Then either $|waa'|\le |wa|$, which again by \cite{Davis} implies that $|waa'|=|w|$, or $waa'$ can be reordered to an alphabetically earlier word.
		
		In the first case, we claim it suffices to consider when $a'=a$. Suppose instead that $a'$ cancels with a different letter in $wa$. Then since $wa$ is assumed to be shortlex, we can write $wa=v_1a'v_2a$, where $a'$ commutes with each letter in $v_2a$, and $a'$ is alphabetically earlier than the first letter of $v_2$. But then the substring $v_2aa'$ can already be reordered to an alphabetically earlier word.
		
		So by induction, we assume that $M_{lex}(w)$ consists of the last letter of $w$ together with any letters $a_i$ so that $wa_i$ can be rearranged to be alphabetically earlier. Then $M_{lex}(wa)$ can be taken to be $\bigl(M_{lex}(w)\cap star(a)\bigr)\cup \{a\}\cup star_<(a)$. That is, $M_{lex}(wa)$ consists of $a$, letters $a'$ so that $aa'$ should be rearranged to $a'a$, as well as letters $a'$ so that $waa'=_{Geo}wa'a$ and $wa'$ itself can be rearranged to a lexicographically earlier word. This shows that $M_{lex}(wa)$ can be computed entirely from $M_{lex}(w)$ and $a$.
		
		An accepted word of this FSM is by definition a shortlex word.
	\end{proof}
	
	\begin{remark}
		
		Based on these two propositions, we can generate these two FSMs by a breadth-first search. More precisely, we describe the algorithm as follows:
		
		Start at the empty string. Since any letter should be allowed as a shortlex geodesic on its own, compute $M_{Geo}(a_i)$ or $M_{lex}(a_i)$ for each $i$ and connect the start vertex to each such state by the edge $a_i$. Add the start vertex into the set $C$ of ``completed vertices", i.e. those vertices whose outgoing edges have been added.
		
		For each vertex $S$ in the FSM but not in $C$, add outgoing edges labeled by $S^c$, using the above calculation to determine which state such an edge should reach. Add in new states as necessary if such a state has not yet been reached in the algorithm. When this process is completed on a given vertex, add that vertex to $C$.
		
		Repeat until every vertex is in $C$.
		
	\end{remark}

	\section{The Rips Graph on a Horosphere}
	\label{sec:RipsGraph}
	
	\subsection{A normal form for vertices}
	\label{subsec:PrefixSuffix}
	In this subsection, we will describe a process to generate a large set of points on the same horosphere. This will serve as the vertex set for the graphs we construct later on.
	
	From now on, we assume that $a_i$ and $a_j$ are non-commuting letters, and study the level sets $b_\gamma$ where $\gamma$ is the geodesic ray which we will colloquially refer to as $(a_ia_j)^\infty$, by which is meant the ray starting at the origin and alternating $a_i$ and $a_j$ indefinitely. We first prove a formula for the Busemann Functions $b_\gamma$. To obtain this formula, some terminology is required.
	
	\begin{definition}
		
		Let $w$ be a shortlex word. Then a \textit{prefix-suffix decomposition of} $w$ is an equality $w=_{Geo}w_{pref}w_{suff}$ so that $w_{pref}$ (the \textit{prefix}) consists entirely of the letters $a_i$ and $a_j$, and $w_{suff}$ (the \textit{suffix}) is shortlex and cannot be rearranged to begin with either $a_i$ or $a_j$. Note that either $w_{pref}$ or $w_{suff}$ may be empty.
		
	\end{definition}
	
	\begin{lemma}
		
		Every shortlex word has a unique prefix-suffix decomposition.
		
	\end{lemma}
	
	\begin{proof}
		
		To obtain the prefix-suffix decomposition, we read $w$ from left to right. Each time we read a letter other than $a_i$ or $a_j$, we append it to $w_{suff}'$. When we read a letter $a_i$ or $a_j$, if every letter so far in $w_{suff}'$ commutes with the $a_i$ or $a_j$, we append it to $w_{pref}$, and otherwise we append it to $w_{suff}'$.
		
		It is a simple induction on the length of $w$ to show that $w=_{Geo}w_{pref}w_{suff}'$. For $a_i\ne a_k\ne a_j$, $wa_k=_{Geo}w_{pref}(w_{suff}'a_k)$ and for $a_k=a_i$ or $a_j$, $wa_k=_{Geo}(w_{pref}a_k)w_{suff}'$ exactly when the new copy of $a_i$ or $a_j$ commutes with all of $w_{suff}'$, and otherwise \newline $wa_k=_{Geo}w_{pref}(w_{suff}'a_k)$.
		
		The word $w_{pref}$ is now shortlex because $a_i$ and $a_j$ do not commute, while any cancellation could already have happened in $w$. Similarly, any cancellation in $w_{suff}'$ could already have happened in $w$, so that $w_{suff}'$ is geodesic but not necessarily shortlex. We therefore replace $w_{suff}'$ with the shortlex rearrangement of equal length, and term this word $w_{suff}$.
		
		Uniqueness is then shown by a straightforward induction on length.		
	\end{proof}
	
	It is immediate that a shortlex word is a suffix (e.g. for itself) if and only if it cannot be rearranged to begin with either $a_i$ or $a_j$. When we say that a word is a suffix, we mean it in this sense, because this notion is independent of any attached prefix.
	
	The existence and uniqueness of the prefix-suffix decomposition can also be shown as follows: re-order the alphabet so that $a_i$ and $a_j$ are the first two letters (the order between the two does not matter since they do not commute). Then $w$ remains geodesic but is not necessarily shortlex. Rearranging $w$ to be shortlex in this new letter order requires moving each copy of $a_i$ and $a_j$ as early in the word as possible, so that $w$ begins with a word consisting of an alternation of $a_i$ and $a_j$. The maximal such word is the prefix, and the remainder of the word is the suffix. In particular, if $\{i,j\}=\{1,2\}$, the prefix and suffix can be read off directly from the shortlex word $w$. 
	
	\begin{lemma} \label{CalculatingBusemannFunctions}
		
		Let $w$ be shortlex and $w_{pref}w_{suff}$ be its prefix-suffix decomposition. Then $b_\gamma(w)=|w_{pref}|+|w_{suff}|$ if the first letter of $w$ is $a_j$, or $b_\gamma(w)=-|w_{pref}|+|w_{suff}|$ if the first letter of $w$ is $a_i$.
	\end{lemma}
	
	\begin{proof}
		
		To compute the Busemann Function, we only need to compute word lengths. Therefore, $w$ need not be in shortlex order, so we may as well work with $w_{pref}w_{suff}$.	Since $a_i$ and $a_j$ do not commute, $w_{pref}$ can be of the form $a_ia_ja_ia_j...$ or $a_ja_ia_ja_i...$.
		
		Now, $(a_ia_j)^{-1}=a_ja_i$, so that in the first case, for $n$ sufficiently large, the entirety of $w_{pref}$ cancels out of the expression $(a_ia_j)^{-n}w_{pref}w_{suff}$. The expression therefore has length \newline $2n-|w_{pref}|+|w_{suff}|$ at most. It cannot have shorter length since any further cancellation would come between a letter $a_i$ or $a_j$ in $(a_ia_j)^{-n}$ and a letter at the start of $w_{suff}$, which is impossible by assumption. Therefore, $b_\gamma(w)=-|w_{pref}|+|w_{suff}|$.
				
		In the case where $w_{pref}$ begins with $a_j$, 
		
		$$b_\gamma(w)=\lim_{n\to\infty} |(a_ia_j)^{-n}w_{pref}w_{suff}|_{\phantom{ }}-2n$$
		
		and no cancellation is possible at all in this expression. It is immediate that $b_\gamma(w)$ is the limit of a sequence whose value is constantly $|w|=|w_{pref}|+|w_{suff}|$. 		
	\end{proof}
	
	Following this lemma, a prefix beginning with $a_i$ is termed a $\textit{negative}$ prefix while a prefix beginning with $a_j$ is a positive prefix. Given a positive prefix, we can make it $\textit{more positive}$ by adding letters on the end or $\textit{more negative}$ by deleting its last letters, and vice versa for negative prefixes. Every negative prefix is more negative than the empty prefix, and every positive prefix is more positive than the empty prefix.
	
	As a consequence of this description, it is usually better to keep track of suffixes. If we know a value for $b_\gamma$, then every word on the horosphere of this value is determined uniquely by its suffix. It will follow from the next lemma that we can generate suffixes as the accepted language of an FSM derived from the shortlex machine. When possible, we will try to perform the computations in this paper by taking only the suffix as an input because it is faster and requires less memory to generate them.
	
	\begin{lemma}\label{ShortlexSuffixMachine}
		
		There is a finite-state machine $M_{Suff}$ whose accepted language consists of shortlex words $w$ not equal to any geodesic word beginning with $a_i$ or $a_j$.
		
	\end{lemma}
	
	\begin{proof}
		
		This language is the intersection of the shortlex language with the language $F_{\{a_i, a_j\}}$. As it is the intersection of two regular languages, it is regular by Proposition \ref{CombiningFSMs}. There is therefore a finite-state machine that recognizes it. 
               
		
		
		
	\end{proof}
	
	It is an inconvenient fact that, in general, it takes quadratic time to convert between a shortlex word and its prefix-suffix form. Therefore, $M_{suff}$ rather than $M_{lex}$ is the correct machine to generate the vertex set of Rips graph.
	
	As a result of this lemma, we can generate a large set of points on a horosphere of a fixed $b_\gamma$ value by taking the set of all suffixes of length at most $k$, and then adding appropriate positive or negative prefixes to achieve the desired $b_\gamma$ value. Next, we must compute the edge set in the Rips graph for each point. While it is straightforward to check the distance between each pair of vertices, this will be extremely time-consuming. It will be more efficient to generate the edges from the ground up.	
	
	Recall that the $k$-Rips graph on a metric space $(X,d)$ consists of the space $X$ with edges between points at distance at most $k$. We think of this graph as describing what it means to move along a subspace of a metric space. Our space $X$ will be a horosphere, which is a discrete set of points. Therefore, a continuous nonconstant path cannot stay in a horosphere. However, we can demand that a path make only short jumps. One might imagine zooming out from a path that makes short jumps along a horosphere, and seeing it limit to a continuous path tangent to the horosphere. Paths in the $k$-Rips graph encode this notion of approximate tangency.
	
	We will focus primarily on the $2$-Rips graph. Note that it is not at all obvious that this graph should be connected.
	
	Given a word $w$, the set of words $v$ at distance at most $2$ from $w$ is $\{wa_ka_l\}$, but most of these words are not on the same horosphere as $w$ because typically multiplying on the right by two letters lengthens the suffix by $2$ while preserving the prefix. To generate these edges more efficiently, it will be helpful to consider two separate cases: edges between words $w$ and $v$ with the same prefix (and therefore the same length), edges between words $w$ and $v$ with suffixes of different length.  
	
	\subsection{Edges between words of the same suffix length} 
	\label{subsec:RipsEdgesSameLength}
        
	Let $w$ and $v$ be words on a horosphere about a simple alternating ray $(a_ia_j)^\infty$. If $w$ and $v$ have the same length of suffix, then they necessarily have the same prefix. So in order for $(w,v)$ to span an edge in the $2$-Rips graph, $w^{-1}v=w_{suff}^{-1}v_{suff}$ must be a word of length 2, and $|w_{suff}|=|v_{suff}|$. That is, given $w_{suff}$, to find all the edges in the $2$-Rips graph to words of the same suffix length, we need to delete one letter from $w_{suff}$ and then add one on. We first address the deletion step.
	
	\begin{lemma} \label{DeletionPreservesShortlexSuffixes}
		
		Let $w$ be a shortlex word, and let $a_k$ be a last letter in a geodesic representative of $w$. Then the word $w_k$ obtained from $w$ by deleting the last copy of $a_k$ is shortlex. Moreover, if $w$ is a suffix, then so is $w_k$.
		
	\end{lemma}
	
	\begin{proof}
		
		Suppose $w_k$ is not shortlex. Then there is some $a_l$ so that, in $w_k$, this copy of $a_l$ either cancels or commutes earlier in a favorable way. But then since $a_k$ is a last letter, either $a_k$ and $a_l$ commute, or this copy of $a_l$ precedes the deleted copy of $a_k$. In both cases, this cancellation or rearrangement is possible in $w$ so that $w$ was not shortlex.
		
		Suppose $w_k$ is shortlex but not a suffix. The argument is the same: there is a copy of $a_i$ or $a_j$ that commutes to the beginning of the word. Since $a_k$ is a last letter, either $a_k$ commutes with this copy of $a_i$ or $a_j$, or the copy of $a_i$ or $a_j$ appears before $a_k$, and in either case this rearrangement is possible in $w$.
	\end{proof}
	
	Next, we address the step of adding a letter. A priori, it suffices to write a letter $a_l$ after $w_k$ permitted by $M_{GeoSuff}(w_k)$, and then perform a shortlex reordering. However, shortlex sorting a string of length $n$ takes, a priori, $O(n^2)$ steps. Here we give an operation that uses the existing order on $w_k$ to alphabetize $w_ka_l$ in linear time.
	
	\begin{lemma}\label{ShortlexInsertion}
		
		Suppose $w$ is a shortlex suffix and $a_l$ is permitted by $M_{GeoSuff}(w)$. The shortlex suffix equivalent to $wa_l$ can be computed in $O(|w|)$ steps.
		
	\end{lemma}
	
	\begin{proof}
		
		Write $w=a_{i_1}...a_{i_n}$. We read the letters of $w$ last to first. When we read $a_{i_k}$, if \newline $a_l\in Star_<(a_{i_k})$, we remember the value $k$ as the current preferred insertion point of $a_l$. If $a_{i_k}$ does not commute with $a_l$, the algorithm terminates and outputs $a_{i_1}...a_{i_k}...a_{i_{m-1}}a_la_{i_m}...a_{i_n}$ where $a_{i_m}$ is the preferred insertion point of $a_l$.
		
		There are no favorable rearrangements of $a_l$ by assumption. Since the remaining letters $a_{i_1},...a_{i_n}$ appear in the same order as in $w$, which is shortlex, any rearrangements of these letters are necessarily unfavorable. Therefore, this algorithm outputs a shortlex word.
	\end{proof}
	
	As a result, we can generate each word of the same suffix length as $w$ that is within distance $2$ of $w$ quickly and with minimal duplication.
	
	\begin{lemma} \label{RipsEdgesSameSuffixLength}
		
		Let $w=_{Geo}w_{pref}w_{suff}$ be a word on the horosphere $b_\gamma=m$. The set of words $\{v: v_{pref}=w_{pref}, b_\gamma (w)=b_\gamma(v), \text{and } d(w,v)=2\}$ can be computed in linear time with respect to the length of $w$.
		
	\end{lemma}

	\begin{proof}
		
		We wish to delete a letter at the end of $w_{suff}$, for which it suffices to multiply by a letter that $w_{suff}$ could be rearranged to end with. We therefore compute $M_{Geo}(w_{suff})$. For each $a_k$ in $M_{Geo}(w_{suff})$, we obtain $w_{k_{suff}}$ with $w_k$ as in Lemma \ref{DeletionPreservesShortlexSuffixes}. These are again shortlex suffixes. We wish to enumerate every shortlex suffix $w_{k_{suff}}a_l$. We therefore take, for each $w_{k_{suff}}$, the set of successors to $w_{k_{suff}}$ in $M_{GeoSuff}$ and reorder according to Lemma \ref{ShortlexInsertion}.
		
		There are at most $Clique(\Gamma)$ words $w_k$ each with at most $|V|-1$ successors. So we have a bounded number of strings of length $|w_{suff}|-1$ to reorder, which takes $O(|w_{suff}|)=O(|w|)$ steps. 
	\end{proof}
	
	Note that for each $w_{k_{suff}}$, $w_{suff}$ is the shortlex form of $w_{k_{suff}}a_k$ so that this algorithm also returns $|M_{Geo}(w_{suff})|$ copies of the degenerate edge $(w_{suff}, w_{suff})$. For convenience, we will delete these edges, which does not change the asymptotic time complexity.
	
	\subsection{Edges between words of different suffix length}
	\label{subsec:RipsEdgesDifferentLength}
	
	We will suppose that $w$ and $v$ are words on the same horosphere, and WLOG that $v$ has the shorter suffix. If the prefix of $w$ is positive or empty, then it gets one letter longer when the suffix shortens in order to preserve the Busemann function. If the prefix of $w$ is negative, then it shortens in order to preserve the Busemann function of $w$.
	
	We describe these cases in the following lemma.
	
	\begin{lemma} \label{RipsEdgesDifferentSuffixLength}
		
		Let $w=_{Geo}w_{pref}w_{suff}$. Then the set of edges between $w$ and words $v$ such that $|v_{suff}|=|w_{suff}|-1$ can be determined in linear time with respect to the length of $w$ based only on $w_{suff}$ and $b_{\gamma}(w)$.
		
	\end{lemma} 
	
	Of course, if one knows $w_{suff}$ and $b_{\gamma}(w)$, then one can reconstruct $w$. The point is to emphasize that we can write such an algorithm where we take as inputs an integer and a desired set of suffixes to work with. We do not need to store the redundant information of the prefix of each word.
	
	\begin{proof}
		
		If $b_{\gamma(w)}\ge|w_{suff}|$, then the prefix of $w$ is non-negative (beginning with $a_j$ or empty), and a letter must be added. If $b_{\gamma(w)}-|w_{suff}|$ is even, then $w_{pref}$ ends with a copy of $a_i$, so a copy of $a_j$ must be added, while if $b_{\gamma(w)}-|w_{suff}|$ is odd, then $w_{pref}$ ends with a copy of $a_j$ so $a_i$ that must be added.
		
		So, letting $k=i$ in the first case or $j$ in the second case, $v$ so that $wa_ka_l=_{Geo}v$ or $w a_l a_k=_{Geo}v$ and so that $w_{suff}a_l=_{Geo}v_{suff}$ where $a_l$ cancels with a letter in $w_{suff}$. This means that there is an $a_l$ present in $w_{suff}$, so that if $a_k$ does not commute with $a_l$, then in the expression $w a_ka_l$, $a_k$ cannot reach $w_{pref}$. So in fact we only need to consider the case $w a_l a_k=_{Geo}v$.
		
		In order for $a_l$ to cancel with a letter of $w_{suff}$, $a_l\in M_{Geo}(w_{suff})$. For each such $a_l$, we obtain a candidate suffix $v_{suff}$ by deleting the last copy of $a_l$ in $w_{suff}$. If $a_k$ commutes with each remaining letter, i.e. if $a_k\in M_{\{a_i, a_j\}}$, then $wa_la_k=_{Geo}v$ is a word of the desired form, and we keep the suffix $v_{suff}$. Otherwise we discard it. One sees immediately that this takes only linear time.
		
		The case where $b_{\gamma(w)}<|w_{suff}|$ is equivalent.
	\end{proof}
	
	Combining all of the above, we have a complete algorithm for generating the $2$-Rips graph on a subset of a horosphere.
	
	\begin{theorem} \label{RipsGraphGenerationSpeed}
		
		Let $\Gamma$ be a graph satisfying the standing assumptions. Let $a_i$ and $a_j$ two non-adjacent letters of $\Gamma$, and let $\gamma=(a_ia_j)^\infty$. Take $k$ to be an integer. Then $n$ vertices of the $2$-Rips graph on horosphere $b_\gamma^{-1}(k)$ and the edges between them can be generated by an algorithm whose runtime is $O(n\log(n))$.
		
	\end{theorem}
	
	\begin{proof}
		
		Using $M_{suff}$, we generate the set of suffixes of length at most $m$. The number $n$ of these is exponential in $m$, so that this takes time asymptotic to $O(n\log(n))$. For each such vertex, it takes $O(\log(n))$ steps to generate all outgoing edges to suffixes of the same length by Lemma \ref{RipsEdgesSameSuffixLength}, and $O(\log(n))$ steps to generate the outgoing edges to suffixes with different length by Lemma \ref{RipsEdgesDifferentSuffixLength}.
	\end{proof}

	\subsection{Connectivity of the Rips graph}
	\label{subsec:RipsConnectivity}
	
	We are almost prepared to show that the 2-Rips graph is connected. First we will first need a technical lemma about the defining graph $\Gamma$. 
	
	\begin{lemma} \label{RipsGraphTraversability}
		
		Let $\Gamma$ be a graph satisfying the standing assumptions, and let $\Gamma'\subset\Gamma$ be a path component of $\Gamma\setminus\{a_i, a_j\}$. Then
		
		\begin{itemize}
			\item the subgraph $K$ of vertices in $\Gamma'$ adjacent to both $a_i$ and $a_j$ is complete.
			
			\item If $\Gamma'\ne K$ then $\Gamma'$ contains both vertices not adjacent to $a_i$ and vertices not adjacent to $a_j$, and every vertex adjacent to $a_i$ but not $a_j$ is connected to a vertex adjacent to $a_j$ but not $a_i$ (and vice versa) by a path in $\Gamma'\setminus K$.
			
		\end{itemize}
		
	\end{lemma}
	
	\begin{proof}
		
		For the first part, recall that $a_i$ and $a_j$ are nonadjacent. So if $a_k$ and $a_l$ are adjacent to $a_i$ and $a_j$, the four vertices will form an induced square subgraph unless $a_k$ and $a_l$ are adjacent.
		
		For the second part, suppose $\Gamma'$ contains a vertex $a_k$ adjacent to $a_i$ but not $a_j$. By assumption, $\Gamma\setminus (K\cup\{a_i\})$ is connected, so that there is a path between $a_k$ and $a_j$ which we may take not to contain $a_j$ so that it stays in $\Gamma'$. But then the second-last vertex in this path is both connected to $a_k$ by a path in $\Gamma'\setminus K$ adjacent to $a_j$, and not adjacent to $a_i$.
	\end{proof}
	
	In order prove the connectedness of the Rips graph, it will be convenient to have another FSM prepared. When proving connectedness, we will write a geodesic suffix  $a_{i_1}a_{i_2}...$ and not care about a shortlex reordering. This is because in order to talk about paths in the Rips graph, we will need to study edge paths along which the $k^{th}$ letter changes in a prescribed way. It would be equivalent but much more painful notationally if the $k^{th}$ letter were forced to wander around the suffix according to the shortlex order.

	\begin{lemma}\label{GeoSuffMachine}
		
		There is a finite-state machine $M_{GeoSuff}$ whose accepted language consists of geodesic words $w$ not equal to any geodesic word beginning with $a_i$ or $a_j$.
		
		Moreover, this finite-state machine has no dead-end states (i.e. states in which no letter can be written).
		
	\end{lemma}
	
	Notice the slight difference from the statement of Lemma \ref{ShortlexSuffixMachine}.
	
	\begin{proof}
		
		We mimic the construction of $M_{Suff}$, but use $M_{Geo}$ rather than $M_{lex}$ as our starting point. The states are $M_{GeoSuff}(w)=(M_{Geo}(w),P_{Suff}(w))$, and the starting point $(\emptyset, \{a_i, a_j\})$. The proof that the machine accepts the desired language is the same as the proof of Lemma \ref{ShortlexSuffixMachine}, replacing the word ``shortlex" with ``geodesic".
		
		The set of edges exiting a given state $M_{GeoSuff}(w)$ is $\Gamma\setminus (M_{Geo}(w)\cup P_{Suff}(w))$. So there is an outgoing edge as long as $\Gamma$ is not the union of a clique and a pair of letters that are not adjacent. This is guaranteed by the standing assumption that there is no separating clique in $\Gamma$.
	\end{proof}
	
	The following tool will be used repeatedly in proving connectivity of the Rips graph.
	
	\begin{lemma} \label{RipsGraphPartialConnectivity}
		
		Suppose $w=a_{j_1}...a_{j_k}$ is a geodesic suffix so that $\Gamma\setminus M_{GeoSuff}(v)$ is connected whenever $v$ is a successor to $w$. Then there is a connected subgraph of the $2$-Rips graph on a horosphere consisting of the set of words of length $m$ with geodesic suffixes beginning $a_{j_1}...a_{j_k}$.
		
	\end{lemma}
	
	The slightly artificial assumption that $\Gamma\setminus M_{GeoSuff}(v)$ always remains connected is chosen to cover two cases: one where $M_{GeoSuff}(w)$ is a clique, and one where $\Gamma\setminus (\{a_i, a_j\})\cup K'$ is connected for every $K'\subset K$. In the former case, either $a_i$ or $a_j$ cannot commute to the start of the word and will not be able to regardless of which letters continue to be written. In the latter case, $\Gamma\setminus M_{GeoSuff}(v)$ will always be connected regardless of whether it is a successor to $w$.
	
	\begin{proof}
		
		We prove this by induction on $k$. If $k=m-1$, an edge is guaranteed by Lemma \ref{RipsEdgesSameSuffixLength}. So all we need to show is that if the result is true for $k=n+1$ then it is true for $k=n$.
		
		So suppose the graph whose vertices are the words whose geodesic suffixes begin \newline $a_{j_1}...a_{j_{n+1}}$ is connected. Then from each word $a_{j_1}....a_{j_m}$ there is a path to each geodesic suffix $a_{j_1}...a_{j_{n+1}}a_{j_{n+2}'}...a_{j_{m'}}$. Choose a path from $a_{j_{n+1}}=a_{j_{n+1_0}},a_{j_{n+1_1}}...a_{j_{n+1_l}}=a_{j_{n+1'}}$ in \newline $\Gamma\setminus M_{GeoSuff}(a_{j_1}...a_{j_n})$. By assumption, we can find a path to a geodesic suffix \newline $a_{j_1}...a_{j_{n+1}}a_{j_{n+1_1}}...a_{j_m'}$ as long as there is a geodesic suffix of length $m$ beginning \newline $a_{j_1}...a_{j_{n+1}}a_{j_{n+1_1}}$. This is guaranteed by the assertion in Lemma \ref{GeoSuffMachine} that $M_{GeoSuff}$ does not have dead ends. But then since $a_{j_{n+1}}$ and $a_{j_{n+1_1}}$ commute, we can write this as the geodesic suffix $a_{j_1}...a_{j_{n+1_1}}a_{j_{n+1}}...a_{j_{m}'}$. 
		
		If we do this $l$ times, we obtain a word that with suffix begin $a_{j_1}...a_{j_n}a_{j_{n+1}'}$. Doing this one more time then gives any word whose suffix begins $a_{j_1}...a_{j_n}a_{j_{n+1}'}$. Since $a_{j_{n+1}'}$ was arbitrary, it follows that in $l+1$ steps we can get to every word beginning $a_{j_1}...a_{j_n}$ is connected.
	\end{proof}
	
	\begin{corollary} \label{RipsSubgraphDiameterBound}
		
		Consider all the graphs $\Gamma\setminus K'$ where $K'$ ranges over all complete subgraphs of $\Gamma$, and all the path components of $\Gamma\setminus (\{a_i, a_j\}\cup K')$ where $K'$ ranges over subgraphs of the clique $K$ adjacent to both $a_i$ and $a_j$. Let $D$ denote the maximum diameter of any such graph considered with its path metric, not with the induced metric of $\Gamma$. Then if $M_{GeoSuff}(a_{i_1}....a_{i_k})$ is a clique, the diameter of the set of words whose (geodesic, not necessarily shortlex) suffixes have length $m$ and begin with $a_{i_1}....a_{i_k}$ is at most $(D+1)^{m-k-1}$ (again where this set is given its intrinsic path metric).
		
	\end{corollary}
	
	The need for $D$ to dominate the diameter of any component of $\Gamma\setminus (\{a_i, a_j\}\cup K')$ will not be used in this proof. It is included only to have a single constant in certain upper bounds in the next subsection.
	
	\begin{proof}
		
		If $k=m-1$, then the subgraph is a clique by Lemma \ref{RipsEdgesSameSuffixLength}. So suppose diameter of the subgraph of words beginning $a_{i_1}...a_{i_n}b$ is at most $D^{m-n-2}$ for each choice of $b$. Then the previous proof gives a path between any two points whose suffixes begin $a_{i_1}...a_{i_n}$ with $l+1$ subpaths where $l$ is a path in a graph $\Gamma\setminus K'$. We can therefore take $l$ to be at most $D$. Each of these subpaths lies in a subgraph of words whose suffixes begin with $a_{i_1}...a_{i_{n+1_{l-1}}}$, and therefore can be taken to have at length at most $(D+1)^{m-n-2}$. Therefore, the total path length is at most $(D+1)^{m-n-1}$.		
	 \end{proof}
	
	As a consequence, we see that all the $k$-Rips graphs on the same horosphere, for any $k$, are bi-Lipschitz equivalent.
	
	\begin{proposition} \label{RipsGraphBiLipschitzEquivalence}
		
		Let $k\ge 2$. The path metrics on any horosphere induced by the $k$-Rips graph and the $2$-Rips graph are bi-Lipschitz equivalent.
		
	\end{proposition}
	
	\begin{proof}
		
		One direction is clear, because points at distance at most $2$ from one another are therefore at most $k$ from one another. So the $k$-Rips graph contains every edge of the $2$-Rips graph, and therefore the distance induced from the $k$-Rips graph is bounded above by that of the $2$-Rips graph. So it remains to show that the path distance in the $2$-Rips graph is at most $L_k$ times the distance in the $k$-Rips graph for some constant $L_k$. It suffices to do this edgewise, i.e., that every edge in the $k$-Rips graph has at most distance $L_k$ in the $2$-Rips graph.
		
		We will need to consider the following cases of edges in the $k$-Rips graph: edges between words with the same suffix length, edges between words whose suffix length differs by $1$, and edges between words whose suffix length differs by at least $2$.
		
		\textbf{Case I: Same suffix length}
		
		Suppose $|w_{suff}|=|v_{suff}|$, so that $w^{-1}v=_{Geo}w_{suff}^{-1}v_{suff}$, which reduces to a word of length at most $k$. Then $w_{suff}w'=_{Geo}v_{suff}$ for some word $w'$ of length at most $k$. $w'$ then has even length, and half of its letters cancel with letters in $w_{suff}$. Write \newline $w_{suff}=_{Geo}a_{i_1}...a_{i_l}a_{i_{l+1}}...a_{i_m}$ where the letters past $a_{i_l}$ cancel with $w'$. If $l\ge clique(\Gamma)$, then $M_{GeoSuff}(a_{i_1}...a_{i_l})$ is necessarily a clique. Since $v_{suff}$ begins with $a_{i_1}...a_{i_l}$, then Corollary \ref{RipsSubgraphDiameterBound} shows that the two at distance at most $(D+1)^{m-l-1}$ in the $2$-Rips graph. Since $m-l=\frac{k}{2}$, this provides the desired bound.
		
		There are only finitely many words for which $l=m-\frac{k}{2}<clique(\Gamma)$, i.e. words whose suffix lengths are at most $clique(\Gamma)+\frac{k}{2}$. The maximum $2$-Rips graph distance between two such points whose Cayley graph distance is at most $k$ is therefore bounded.
		
		\textbf{Case II: Suffix length differing by 1}
		
		This is the hardest case. Suppose $|v_{suff}|=|w_{suff}|+1$, and WLOG let $w_{pref}a_i=v_{pref}$. Once again there is a word $w'$ of length at most $k$ so that $ww'=_{Geo}v$. There is an uncanceled prefix copy of $a_i$ in $v$, as well as at most $\frac{k}{2}$ uncanceled suffix letters in $v$ and at most $\frac{k}{2}-1$ uncanceled suffix letters of $w$. So as before we write $w_{suff}=_{Geo}a_{i_1}...a_{i_l}a_{i_{l+1}}...a_{i_m}$ where the letters after $a_{i_l}$ cancel with letters of $w'$, and the first $l$ letters match letters in $v_{suff}$ and commute with $a_i$. 
		
		As before, as long as $l>clique(\Gamma)$, $M_{GeoSuff}(a_{i_1}...a_{i_l})$ is a clique. We then apply Corollary \ref{RipsEdgesDifferentSuffixLength}  to the set of words beginning with these letters of length $m$ to find that the set has diameter at most $(D+1)^{\frac{k}{2}-2}$. In particular, after distance at most $(D+1)^{\frac{k}{2}-2}$ we can reach a new word where every suffix letter commutes with $a_i$. We can then take a single edge to a word whose suffix length matches that of $v$ by Lemma \ref{RipsEdgesDifferentSuffixLength}. Then a further application of Corollary \ref{RipsEdgesDifferentSuffixLength} to the set of suffixes of length $m+1$ beginning $a_{i_1}...a_{i_l}$ allows us to reach $v$ in another $(D+1)^{\frac{k}{2}-1}$ steps. So the total distance is bounded by $(D+1)^{\frac{k}{2}-2}+1+(D+1)^{\frac{k}{2}-1}$.
		
		Once again, the set of words $w$ for which $l=|w_{suff}|-(\frac{k}{2}-1)<clique(\Gamma)$ is finite, so we have an immediate bound on the lengths of $k$-Rips graph edges incident to these vertices.
		
		\textbf{Case III: Suffix length differing by at least 2}
		
		Suppose $w$ and $v$ are at distance at most $k$, and $|v_{suff}|-|w_{suff}| \ge 2$. Then there is an uncanceled prefix copy of both $a_i$ and $a_j$ in $w^{-1}v$. By Lemma \ref{RipsGraphTraversability}, there are at most $clique(\Gamma)$ letters from $w_{suff}$ that can cancel with letters of $v_{suff}$. Therefore, $$|v_{suff}|-|w_{suff}|+|w_{suff}|+|v_{suff}|-2clique(\Gamma)\le k.$$ Canceling, $$2|v_{suff}|\le k+2clique(\Gamma).$$ Since $v$ had the longer suffix than $w$, we conclude that there are again only finitely many such edges, and attain a maximum length of such an edge in the $2$-Rips graph. \end{proof}

	We now come to the main result of the section. We will present an alternate, more direct proof under an additional assumption after this one.
	
	\begin{proposition} \label{RipsGraphOriginConnectivity}
		
		Let $\Gamma$ be a graph satisfying the standing assumptions, and let $a_i$ and $a_j$ be non-adjacent vertices. Let $\gamma$ be the geodesic ray $(a_ia_j)^\infty$, and let  $w=_{Geo}w_{pref}w_{suff}$ be an element of $W_\Gamma$. Denote $G$ as the $2$-Rips graph on $b_{\gamma}^{-1}(b_{\gamma}(w))$, the horosphere containing $w$. Let $v$ be the word with empty suffix in $G$. Then there is a path in $G$ from $v$ to $w$.
		
	\end{proposition}
	
	\begin{proof}
		
		Let $|w|=m$. Then there is an edge between $w$ and $v$ in the $m$-Rips graph, so they are at distance $1$ in that graph. Since the $m$-Rips graph and the $2$-Rips graph are bi-Lipschitz, the distance between $w$ and $v$ in the $2$-Rips graph is bounded, and therefore a path exists.
	\end{proof}
	
	Since there is a point in the horosphere to which every other point is connected by a path, we see immediately that the horosphere is connected.
	
	\begin{theorem} \label{RipsGraphConnectivity}
		
		The $2$-Rips graph on a horosphere is connected.
		
	\end{theorem}
	
	Since the word on the horosphere of empty suffix is connected to every other word, the horosphere is connected. However, depending on the defining graph $\Gamma$ it is sometimes possible to connect vertices without going through this word. The following proposition presents a simpler proof of connectivity under an additional assumption.
	
	\begin{proposition}\label{RipsGraphNoCutPtConnectivity}
		
		Suppose that $\Gamma\setminus (\{a_i, a_j\}\cup K')$ is connected for each $K'$ a (potentially empty) subclique of $K$ (where $K$ is again the set of letters commuting with both $a_i$ and $a_j$). Then the words on a given horosphere with suffix $a_{i_1}...a_{i_n}$ and $a_{j_1}...a_{j_m}$ are connected through a path of words that of suffix length at least $\min\{n, m\}$.
		
	\end{proposition}
	
	\begin{proof}
		
		By assumption, $\Gamma\setminus M_{GeoSuff}(w)$ is connected for each geodesic suffix $w$. Therefore, we will again use Lemma \ref{RipsGraphPartialConnectivity}.
		
		For the words of suffix length $m$, the empty string satisfies the assumptions of Lemma \ref{RipsGraphPartialConnectivity}, so they are all connected. WLOG suppose $m<n$. Then $n-m$ times we use Lemma \ref{RipsGraphPartialConnectivity} to reach words satisfying the assumptions of Lemma \ref{RipsEdgesDifferentSuffixLength}, followed by one additional time to reach the word $a_{i_1}...a_{i_n}$.
	\end{proof}
	
	\subsection{Distortion of the Rips Graph}
	\label{subsec:RipsDistortion}

	In this subsection, we will give an upper and lower bound on the exponential distortion of geodesic paths in the Rips graph with respect to geodesic paths in the Cayley Graph.

	\begin{theorem}\label{RipsGraphDistortion}
		
		Let $\Gamma$ be a graph satisfying the standing assumptions, and let $a_i$ and $a_j$ be non-adjacent vertices. There are constants $C_i$ for $i=1, 2, 3, 4$ depending only on $\Gamma$ so that the following holds. Let $d_H$ denote the combinatorial distance in the $2$-Rips graph on $b_\gamma^{-1}(k)$, and $d$ denote the distance in the Cayley graph. Then for any $w$, and $v$ on the $k$-horosphere, $$C_1 C_2^{d(w,v)}\le d_H(w,v)\le C_3 C_4^{d(w,v)}.$$
		
	\end{theorem}	
	
	The upper bound follows from the methods of the previous section.
	
	\begin{proposition} \label{HorosphericalDistortionUpperBound}
		
		Let $w$ and $v$ be shortlex words with $b_\gamma(w)=b_\gamma(v) = k$. Denote $d$ the distance in the Cayley graph and $d_H$ the path distance in the $2$-Rips graph on the horosphere $b_\gamma^{-1}(k)$. Then $d_H(w,v)\le O_\Gamma(1)(D+1)^{d(w,v)}+d(w,v)+O_\Gamma(1)$.

	\end{proposition}
	
	Of course, the addition of the linear term $d(w,v)$ makes no difference, since both $D+1$ and the distance between points is at least $1$. So we could delete this term, but will keep it for convenience. In fact we could also subsume the constant term into the exponential if we desired.
	
	The proof of this proposition is somewhat lengthy and involves counting the steps in the proofs of Propositions \ref{RipsGraphOriginConnectivity} and \ref{RipsGraphNoCutPtConnectivity}. However, when exactly we can use which argument is somewhat subtle, and requires breaking into cases that will likely be hard to follow for the reader unfamiliar with the JSJ theory of RACGs. Therefore, some explanation is in order. We will not prove the content of this remark, since it is not used in the sequel and is meant only to aid the intuition.

    \begin{remark}
		
		In the boundary $\partial W_\Gamma$, the pair $\{\gamma(\pm\infty)\}$ is a cut pair exactly if it is the boundary of a 2-ended group $H$ so that $W_\Gamma \cong G_1*_HG_2$ for groups $G_1$ and $G_2$ properly containing $H$. $H$ can be taken to be the RACG generated by some subgraph of $\Gamma$, necessarily containing $a_i$ and $a_j$, as well as some subclique $K'$ of the clique $K$ of vertices adjacent to $a_i$ and $a_j$, such that the graph $\Gamma\setminus( \{a_i, a_j\}\cup K')$ is disconnected.
		
		The group $H_{max}=\langle \{a_i, a_j\}\cup K\rangle$ is the maximal 2-ended subgroup whose endpoints are $\{\gamma(\pm\infty)\}$, and thus the stabilizer of $\{\gamma(\pm\infty)\}$ in $W_\Gamma$. If $\Gamma_1$ is a path component of \newline $\Gamma\setminus (\{a_i, a_j\}\cup K)$, and $K_1$ is the clique of vertices in $K$ adjacent to vertices in $\Gamma_1$, then $\Gamma_1$ is also a path component of $\Gamma\setminus (\{a_i, a_j\}\cup K_1)$. By Lemma \ref{RipsGraphTraversability}, $\Gamma_1$ contains vertices adjacent to $a_i$ and $a_j$.  Denote $H=\langle \{a_i, a_j\}\cup K_1\rangle$, and $G_1=\langle \Gamma_1\cup K_1 \cup \{a_i, a_j\}\rangle$. The connectivity assumptions guarantee that $G_1$ is a 1-ended group in which $\gamma(\pm \infty)$ is not a cut pair. The vertices in $K\setminus K'$ generate a finite subgroup of $H_{max}$. In the amalgamation $H_{max}*_H G_1$, left multiplying by a product of elements in $K\setminus K_1$ fixes $\gamma(\pm\infty)$ (since these elements are in $H_{max}$, a 2-ended subgroup with $\gamma(\pm\infty)$ as its boundary), but sends $G_1$ to a coset that is not at finite distance. As a result, $\partial(H_{max}*_H G_1)$ has $2^{|K\setminus K_1|}$ copies of $\partial G_1$ glued along $\gamma(\pm\infty)$. Moreover, $\partial(H_{max}*_H G_1)$ embeds into $\partial(W_\Gamma)$. Applying the same logic to each path component $\Gamma_l$ of $\Gamma\setminus (\{a_i, a_j\}\cup K)$ shows that $\gamma(\pm\infty)$ disconnect the boundary into $\sum _l 2^{|K\setminus K_l|}$ path components. These components of $\partial W_\Gamma\setminus \gamma(\pm\infty)$ are therefore parameterized by a subgraph $\Gamma_l$ and a subset of $K_l$.	
    \end{remark}

    The point of the following proof is that Rips graph on a horosphere shows a coarse version of these same properties. The vertex $w_0$ with empty suffix is a member of the subgroup $\langle a_i, a_j\rangle$, and the collection $S$ of vertices with suffixes spelled in the letters of $K$ make up a ``coarse cut point"  in the horosphere (i.e. a connected set of bounded diameter disconnecting the horosphere into unbounded components) if and only if $\gamma(\pm\infty)$ are a cut pair in $\partial W_\Gamma$. There are again $\sum_l 2^{|K\setminus K_l|}$ different complementary components to the coarse cut point. We will show that each suffix outside the coarse cut point $S$ determines a subgraph $\Gamma_l$ and a subset of $K\setminus K_l$. Two points on the horosphere are connected by a path that does not go near $w_0$ exactly when their suffixes determine the same subgraph $\Gamma_l$ and subset of $K\setminus K_l$. Such pairs of points will admit a path like the one described in the proof of Proposition \ref{RipsGraphNoCutPtConnectivity}, while any other pair essentially requires a concatenation of two of the paths described in the proof of Proposition \ref{RipsGraphOriginConnectivity}.

    Crucially to this proof, two points on a horosphere that are close in the Cayley graph metric will necessarily admit one of the shorter paths described in Proposition \ref{RipsGraphNoCutPtConnectivity}. This will allow for the desired upper bound; it will not be possible for two points nearby in the Cayley graph and far from the word $w_0$ to be connected only by a long path that goes near $w_0$.
	
	\begin{proof} [Proof of Proposition \ref{HorosphericalDistortionUpperBound}]
		
		In this proof, all paths will go from $w$ to $v$. The words $w_1, w_2, ...$ represent specific intermediate words that the path goes through.
		
		Take $w_{suff} =_{Geo} a_{i_1} ... a_{i_k}$ and $v_{suff} =_{Geo} a_{j_1} ... a_{j_l}$. Write these letters in order such that $a_{i_1} ... a_{i_m}$ are adjacent to both $a_i$ and $a_j$ (hence all commute with one another) and no other letter commuting with both can be rearranged to the $m+1^{st}$ position of $w_{suff}$, and the same for $v$ and the letters $a_{j_1}, ... a_{j_n}$.
		
		Before getting to the main cases described in the remark above, two edge cases must be handled.
		
		If both $m=k$ and $n=l$, then $w$ and $v$ are at most distance $2clique(\Gamma)$ apart in the Cayley graph, and are connected by a path of length no more than $clique(\Gamma)$ edges long. So $d_H(w, v)\le O(1)$ and the result is immediate in this case.
		
		If $m\ne k$ but $n=l$, or vice versa, then we can essentially use a single path as described in Proposition \ref{RipsGraphOriginConnectivity}. We will make repeated use of Lemma \ref{RipsGraphTraversability} and Corollary \ref{RipsSubgraphDiameterBound}. It takes no more than $(D+1)^{k-m-2}$ steps to reach a any word beginning $a_{i_1} ... a_{i_n}a_{i_{n+1}}a'$. Taking $a'$ to commute with $a_{i_{n+1}}$, it follows that we can change the $m+1^{st}$ entry to an adjacent vertex in $(D+1)^{k-m-2}$ steps. We need to reach a vertex commuting with the last prefix letter of $v$, so that a prefix letter will commute to the end of the word and thus we will find a Rips graph edge to a shorter word. The path from $a_{i_{m+1}}$ to such an edge is a path in $\Gamma\setminus  \{a_i, a_j\}$, and therefore has length at most $D$. After one more application of Corollary \ref{RipsSubgraphDiameterBound}, we are able to delete a letter from the suffix in $(D+1)^{k-m-1}+1$ steps. Call the resulting word $w_1$.
		
		Deleting the next letter to reach $w_2$ therefore takes $(D+1)^{k-m-2}+1$ steps, and so on until we have deleted all the letters after $a_{i_m}$. In total, this takes $$k-m+\sum_{p=0}^{k-m-2} (D+1)^{p+1} \le k-m+\frac{1}{D}(D+1)^{k-m}\le d(w,v)+O_\Gamma(1)(D+1)^{d(w,v)}$$ steps to get to this word $w_{j-m}$, since $d(w,v)\ge k-m$. Finally, it takes at most \newline $clique(\Gamma) = O_\Gamma(1)$ steps to change $a_{i_1}, ... a_{i_m}$ into $a_{j_1}, ... a_{j_n}$. 
		
		We come now to the main cases, where both $w$ and $v$ contain letters outside of the clique $K$. To the word $w$ we associate the subgraph $\Gamma_w$ of $\Gamma$, which is the path component of $\Gamma\setminus(\{a_i, a_j\}\cup K)$ containing $a_{i_{m+1}}$. Similarly $\Gamma_v$ is the path component containing $a_{j_{n+1}}$. Take $K_w$ to be the letters in $a_{i_1} ... a_{i_m}$ that are not adjacent to $\Gamma_w$, and $K_v$ to be the letters in $a_{j_1}, ... a_{j_n}$ not adjacent to $\Gamma_v$.
		
		\textbf{Case I}: Either $\Gamma_w\ne \Gamma_v$ or $K_w\ne K_v$.
		
		In this case we will assume WLOG that $k-m\ge l-n$.
		
		In the first case, $a_{i_{m+1}}$ and $a_{j_{n+1}}$ do not commute, and the only letters in $w_{suff}$ and $v_{suff}$ that could commute past them to cancel with one another necessarily belong to $\{a_i, a_j\}\cup K$. There are none of these besides $a_{i_1}, ... a_{i_m}$ and $a_{j_1}, ... a_{j_n}$ by assumption. In the second case, similarly, a letter of $K_w\Delta K_v$ appears in the difference $w^{-1}v$ and prevents any other letters outside of $a_{i_1}, ... a_{i_m}$ and $a_{j_1}, ... a_{j_n}$ from canceling. As a result, in either of these cases,
		
		\begin{align*}
		d(w,v)&=|w_{suff}^{-1}w_{pref}^{-1}v_{pref}v_{suff}|\\
		&\ge|k-l|+|\{a_{i_1}, ... a_{i_m}\}\Delta \{a_{j_1}, ... a_{k_n}\}|+k-m+l-n\\
		&\ge |k-l|+|m-n| +k-m+l-n.
		\end{align*}
		
		We proceed as in the previous case. First we will describe a path from $w$ to the word $w_{k-m}=a_{i_1},... a_{i_m}$. If $m=k$, then we can skip this first step. It takes no more than \newline $(D+1)^{k-m-1}$ steps to delete the first letter, then $(D+1)^{k-m-2}$, and so on. Thus, deleting $k-m$ letters takes at most $\frac{1}{D}(D+1)^{k-m}+k-m$ steps. It them takes no more than $clique(\Gamma)$ steps to change $a_{i_1}, ... a_{i_m}$ into $a_{j_1}, ... a_{j_n}$. Reversing the roles of $v$ and $w$, it takes at most $\frac{1}{D}(D+1)^{l-n}+l-n$ more steps to reach $v$.
		
		So 
		\begin{align}
			|k-l|+|m-n| +k-m+l-n&\le d(w,v)\\
			d_H(w,v)&\le \frac{1}{D}(D+1)^{k-m}+k-m+\frac{1}{D}(D+1)^{l-n}+l-n + clique(\Gamma)\\
			&\le \frac{2}{D}(D+1)^{k-m}+d(w,v)+ clique(\Gamma),
		\end{align}
		
		where we used the facts that $l-m\ge j-n$ and $d(w,v)\ge k-m+l-n$ in the equation (3). It follows that $$d_H(w,v)\le O_\Gamma(1)(D+1)^{d(w,v)}+d(w,v)+O_\Gamma(1).$$
		
		\textbf{Case II}: $\Gamma_w=\Gamma_v$ and $K_w=K_v$
		
		In this case, we suppose WLOG that $k\ge l$, and also that the letters of $K_w=K_v$ are rearranged to be the first $|K_w|$ letters $a_{i_1}, ... a_{i_{|K_w|}}$ and $a_{j_1}, ... a_{j_{|K_v|}}$ of $w_{suff}$ and $v_{suff}$.
		
		Here the distance $d(w,v) \ge |k-l|+m-n+|k-m-(l-n)|$. However, in order for letters past $a_{i_m}$ to cancel with letters past $a_{j_n}$, they must commute with each prefix letter that does not cancel in $w_{pref}^{-1}v_{pref}$. In particular, either $k-l\le 1$, or $d(w,v) \ge k-l |m-n|+k-m+l-n$ (i.e. cancellation is only possible between prefix letters and between the $a_{i_1}, ... a_{i_m}$ and the $a_{j_1}, ... a_{j_n}$).
		
		\textbf{Case II(a)}: The only letters that cancel between $w_{suff}$ and $v_{suff}$ lie in the clique $K$.
		
		Note that, in particular, this happens when $k-l\ge 2$.
		
		Here again the first step is to change the length of suffix of $w$. As before, it takes at most \newline $(D+1)^{k-m-1}+1$ steps to reduce the length of the suffix of $w$ by $1$ and reach $w_1$, and then $(D+1)^{k-m-2}$ to reduce the length the second time to reach $w_2$, and so on. We therefore require at most $\frac{1}{D}(D+1)^{k-m}+k-l$ steps to go from $w$ to $w_{k-l}$, which has the desired suffix length.
		
		The next step is to change $a_{i_1}, ... a_{i_m}$ into $a_{j_1}, ... a_{j_n}$. Each $a_{i_*}$ that is not an $a_{j_*}$ is necessarily adjacent to vertices in $\Gamma_w=\Gamma_v$, and vice versa. Therefore, for $a_{i_m}$ through $a_{i_{|K_w|+1}}$, we repeatedly move to a vertex where the $m+1^{st}$ letter is either $a_{j_{|K_v|+1}}$ if $|K_v|+1\le n$, and then commute this letter to position $|K_v|+1$, to reach the word $w_{m-n+1}$. This takes $D(D+1)^{l-m-2}$ steps, and we do it $n-|K_v|$ times. If $n>m+1$, then eventually we move to a vertex where the $m+2^{nd}$ letter is $a_{j_{m+2}}$, then where the $m+3^{rd}$ letter is $a_{j_{m+3}}$ and so on. Each of these take less than $D(D+1)^{l-m-2}$ steps. So in at most $(n-|K_v|)D(D+1)^{l-m-2}$ steps, we reach a word $w_{k-l+n-|K_v|}$ which begins $a_{j_1}, ... a_{j_n}$.
		
		If $m>n$, it now takes $D(D+1)^{l-m-2}$ steps to replace $a_{i_{n+1}}$ with a letter in $\Gamma_w$ and reach the word $w_{k-l+n-|K_v|+1}$, followed by $D(D+1)^{l-m-1}$ to replace $a_{i_{n+2}}$ steps to replace $a_{i_{n+2}}$ with a letter in $\Gamma_w$, and so on $m-n$ times. All told, this takes at most $(D+1)^{l-n}$ steps. If $m\le n$ then we can skip this step.
		
		Finally, it takes $(D+1)^{l-n}$ more steps to reach $v$ as before.
		
		In total, it takes $\frac{1}{D}(D+1)^{k-m}+k-l$ steps to go from $w$ to $w_{k-l}$, followed by \newline $(n-|K_v|)(D)(D+1)^{l-m-2}$ to reach $w_{k-l+n-|K_v|}$, $(D+1)^{l-n}$ to reach $w_{k-l+n-|K_v|+1+m-n}$ (if applicable), and $(D+1)^{l-n}$ steps to reach $v$.
		
		Using the facts that $k\ge l$ and $|m-n|\le clique(\Gamma)$, we see that
		
		\begin{align*}
			d_H(w,v) & \le \frac{1}{D}(D+1)^{k-m} + k-l + (n-|K_v|)(D)(D+1)^{l-m-2} + (D+1)^{l-n} + (D+1)^{l-n-2}\\
			&\le \frac{1}{D}(D+1)^{k-m} + k-l + (clique(\Gamma))(D)(D+1)^{l-m-2} + (D+1)^{l-n} + (D+1)^{l-n-2}\\
			&\le k-1+ \frac{1}{D} \bigl[ (D+1)^{k-m}+clique(\Gamma)(D+1)^{l-m}+(D+1)^{l-n+1}+(D+1)^{l-n-1}\bigr]\\
			&\le k-1+ \frac{1}{D}\bigl[ (D+1)^{k-m}+clique(\Gamma)(D+1)^{k-m}+(D+1)^{k-n+1}+(D+1)^{k-n-1}\bigr]\\
			&\le k-1 + \frac{1+clique(\Gamma)+(D+1)^{clique(\Gamma)+1}+(D+1)^{clique(\Gamma)-1}}{D} (D+1)^{k-m}.
		\end{align*}
		
		Since $d(w,v)\ge k-l$ and also $d(w,v)\ge k-m$, it follows that $$d(w,v) \le O_\Gamma(1)(D+1)^{d(w,v)}+d(w,v)$$ and no constant term is needed in this case.
		
		\textbf{Case II(b)}: Letters outside of $K$ cancel.
		
		Note that this implies that $0\le k-l\le 1$. The thrust of the proof in this case is that any such cancellation shortens the path between $w$ and $v$. 
		
		Reorder $w_{suff}$ and $v_{suff}$ so that $w_{suff}=_{Geo}b_{i_1}, ... b_{i_m}, ... b_{i_k}$ and $v_{suff}=_{Geo}b_{j_1}, ... b_{j_m}, ... b_{j_l}$, where the letters $b_{i_1}, ... b_{i_m} = b_{j_1}, ... b_{j_m}$ are the suffix letters that cancel in the expression $w^{-1}v^{-1}=_{Geo}w_{suff}^{-1}w_{pref}^{-1}v_{pref}v_{suff}$ (by a slight abuse of notation, the value of $m$ may no longer be the same as in the previous cases). At least one of the $b_{i_1}, ... b_{i_m}$ is not in $K$ by assumption, so $M_{GeoSuff}(b_{i_1}, ... b_{i_m})$ is a clique. Also since each such letter cancels with the equivalent letter in $v$, if $w_{pref}\ne v_{pref}$, then each such letter commutes with the single letter in $w_{pref}^{-1}v_{pref}$. 
		
		Therefore, if $k-l=1$, then there is a word $w_1$ of length $k$ beginning $b_{i_1}, ... b_{i_m}$ adjacent to a word of length $k-1=l$ again beginning $b_{i_1}, ... b_{i_m}= b_{j_1}, ... b_{j_m}$. Therefore, applying Lemma \ref{RipsGraphPartialConnectivity}, there is a path from $w$ to $w_1$ of length at most $(D+1)^{k-m-1}$, 1 more edge is required to shorten the suffix to the same length as $v_{suff}$. Finally, another application of Lemma \ref{RipsGraphPartialConnectivity} shows that it takes no more than $(D+1)^{l-m-1}$ edges to reach $v$. In total, this path has length at most $2(D+1)^{k-m-1}+1=O(1)(D+1)^{d(w,v)}+O(1)$ since $d(w,v)\ge k-m$.
		
		If instead $k=l$, the result is a direct application of Lemma \ref{RipsGraphPartialConnectivity}, with length at most $(D+1)^{k-m-1}\le (D+1)^{d(w,v)}$.		
	\end{proof}
	
	We next give the lower bound. Somewhat irritatingly, it involves the value of $\delta$, which is often hard to compute precisely.
	
	\begin{proposition}[\cite{BridsonHaefliger}, Proposition III.H.1.6]
		
		Let $X$ be a geodesic $\delta$-hyperbolic metric space, $w$ and $v$ two points connected by a path $\sigma$, and a geodesic $[w,v]$. If $x\in[w,v]$, then 
        $$d(x, im(\sigma))\le \delta |log_{2}(\ell(\sigma))|+1,$$
        where $\ell(\sigma)$ is the length of $\sigma$.
		
	\end{proposition}
	
	Paraphrasing, if there is a geodesic between $w$ and $v$ with a point $x$ far from the path $\sigma$, then $\sigma$ must be very long.
	
	\begin{corollary} \label{HorosphericalDistortionLowerBound}
		
		Let $w$ and $v$ be on the horosphere $b_{\gamma}^{-1}(k)$ and let $d_H$ denote the distance in the $2$-Rips graph on the horosphere. Then $d_H(w, v)\ge 2^{\frac{d(w, v)-4-2\delta}{2\delta}}$.
		
	\end{corollary}
	
	\begin{proof}
		
		A geodesic from $w$ to $v$ is given by the expression $w^{-1}v =_{Geo} w_{suff}^{-1}w_{pref}^{-1}v_{pref}v_{suff}$. Let $m$ letters of $w_{suff}$ and $n$ letters of $v_{suff}$ survive cancellation. By Lemma \ref{CalculatingBusemannFunctions}, $w_{pref}^{-1}v_{pref}$ consists of $m-n$ letters, where a negative number denotes deletions that decrease the Busemann function, while a positive number denotes deletions that increase the Busemann function.
		
		This geodesic therefore has length $2\max\{m, n\}$. Deleting a prefix letter decreases the Busemann function value. So this path begins by decreasing the Busemann function value $m$ times, followed by either decreasing it $n-m$ more times if $n>m$ or else increasing it $m-n$ times if $m>n$, and concludes by increasing the Busemann function $n$ times. So the geodesic contains a point $x$ whose Busemann function value is $k-\frac{d(w, v)}{2}$.
		
		Now let $\sigma$ be any path between $w$ and $v$ in the 2-Rips graph on $b_{\gamma}^{-1}(k)$. We can realize $\sigma$ in the Cayley graph as a sequence of edge-pairs, each of which either first increases $b_\gamma$ and then decreases it, or first decreases $b_\gamma$ and then increases it. In particular, $\sigma$ follows a curve $\sigma'$ in the Cayley graph whose maximum Busemann function value is $k-1$. The geodesic $[w,v]$ therefore avoids $\sigma'$ by at least $\frac{d(w, v)}{2}-1$ since each Busemann function is 1-Lipschitz.
		
		Applying the proposition to $\sigma'$ shows that $\ell(\sigma')\ge 2^{\frac{d(w, v)-4}{2\delta}}$ as a path in the Cayley Graph. Then as an edge path, $\ell(\sigma)\ge 2^{\frac{d(w, v)-4-2\delta}{2\delta}}$.
	\end{proof}
	
	Notice that the rate of exponential distortion is slower the larger $\delta$ is. That is, the slowest distortion corresponds to the thickest triangle. This is to be expected. See \cite{HeinzeImHof} for another circumstance where the minimal exponential distortion rate of a horosphere corresponds to the least negative curvature present.
	
	This proves Theorem \ref{RipsGraphDistortion}. As a free corollary, we obtain polynomial growth of the $2$-Rips graph. Of course, since each $k$-Rips graph is bi-Lipschitz to the $2$-Rips graph, this proves it for all of them.
	
	\begin{corollary} \label{RipsGraphGrowth}
		
		There is a polynomial $P$ depending only on $\Gamma$ so that if $B(w_0, r)$ is a ball of radius $r$ about the point with empty suffix in the $d_H$ metric on the $2$-Rips graph on $b_\gamma^{-1}(k)$, then 
        $$ |B(w_0, r)|\le P(r).$$
		
	\end{corollary}
	
	Since all the Rips graphs are bi-Lipschitz, this will also hold for any other Rips graph with a possibly-different polynomial.
	
	\begin{proof}
		
		By a theorem of Coornaert, the number of shortlex words of length at most $n$ is asymptotic to $\lambda^n$, where $\lambda$ depends only on $\Gamma$ \cite{Coornaert}. Since shortlex suffixes are a subset of shortlex words, the number of shortlex suffixes is bounded above by $O_\Gamma(1)\lambda^n$. 
        
        The set of words whose suffix has length at most $n\delta$ is precisely the $2n\delta$ ball around $w_0$ in the Cayley graph, intersected with $b_\gamma^{-1}(k)$. By Corollary \ref{HorosphericalDistortionLowerBound}, this contains a $d_H$-ball of radius at least $2^{n-\frac{2}{\delta} -1}$. Therefore, $|B(w_0, 2^{n-\frac{2}{\delta}-1})|\le O_\Gamma(1)\lambda^{2n\delta}$. So if $P(x)$ is a polynomial of order at least $2\delta\log_2(\lambda)$, then so is $P(x+\frac{2}{\delta}+1)=P_1(x)$. In particular, for sufficiently large $n$, $$|B(w_0, 2^{n-\frac{2}{\delta}-1})|\le O_\Gamma(1)P_1(2^{n-\frac{2}{\delta}-1}).$$
		
		This shows the result for an unbounded set of radii. We must now modify $P_1$ to some polynomial $P_2$ so that the result holds for every radius. To do this, note that $|B(w_0, r)|$ is monotone in $r$. For any $r$, between $r$ and $2r$ there is a number of the form $2^{n-\frac{2}{\delta}-1}$ for some $n$. For this value of $n$, $$O_\Gamma(1)P_1(2r)\ge O_\Gamma(1)P_1(2^{n-\frac{2}{\delta}-1}) \ge |B(w_0, 2^{n-\frac{2}{\delta}-1})|\ge |B(w_0, r)|.$$ So the result holds for $P_2(x)=P_1(2x)$ for sufficiently large $r$. Then we can simply add a constant large enough to account for the small values of $r$.
		\end{proof}
	
	\section{The divergence graph on a horosphere}
	\label{sec:DivergenceGraph}
	\subsection{Basic definitions}
	\label{subsec:DivergenceDefinitions}
	We next define the other graph structure we wish to study for a horosphere, which is due to \cite{CGSR}. We present a specialization of their definition to the case of horospheres in RACGs.
	
	\begin{definition} \label{DivergenceGraphDef}
		
		Let $A$ be the adjacency matrix of the shortlex FSM. This matrix is not necessarily irreducible. That is, the graph need not be connected (allowing only directed paths). The graph can however be decomposed into connected subgraphs $\Gamma_i$ so that there are edges from $\Gamma_i$ to $\Gamma_j$ only if $i<j$. 
		
		Correspondingly, the adjacency matrix can be permuted into block upper triangular form, with $\lambda_i$ the Perron-Frobenius eigenvalue of the $i^{th}$ diagonal block $A_i$. We then divide the states into three classes as follows. \textit{Large states} are those in a subgraph $\Gamma_i$ whose adjacency matrix has $|\lambda_i|$ maximal. \textit{pre-large states} are those in a subgraph $\Gamma_i$ for which $|\lambda_i|$ is not maximal, but which is connected by an oriented path to a subgraph $\Gamma_j$ for which $|\lambda_j|$ is maximal. \textit{Small states} are states that are neither large nor pre-large.
		
	\end{definition}
	
	Since the Perron-Frobenius eigenvalues of the shortlex machine correspond to the exponential growth rate of the group $G$, one should think that large and pre-large states are those whose shortlex successors grow at rates comparable to the growth of the group as a whole.
	
	\begin{definition}
		
		For any geodesic ray $\gamma$, we define the state function $S$ by $$S(w)=\lim_{n\to\infty} M_{lex}( \gamma(n)^{-1} w).$$ One may a priori need to take a subsequential limit in order for this limit to exist. If so, choose a subsequence $n_i$ so that $M_{lex}( \gamma(n_i)^{-1} w)$ terminates for every $w$ (this can always be done by a standard diagonalization trick from analysis). We will address later on why this is unnecessary in our setting.
		
		The \textit{Divergence Graph} for the horosphere $H_\gamma(n)$ is the graph whose vertex sets consists of the points on $H_\gamma(n)$ whose $S$-state is either large or pre-large. Its edge set is described as follows.
		
		We define a function $P_0$ on shortlex words so that $P_0$ applied to the empty word returns the empty word, and given any nonempty shortlex word $w$, $P_0(w)$ is $wa_i$ where $a_i$ is the last letter of $w$. That is, $P_0$ deletes the last letter of shortlex words other than the identity. We will abuse notation slightly by applying $P_0$ to non-shortlex words. By writing this, we always mean that $P_0$ is to be applied to the shortlex representative of the word.
		
		With this notation, the \textit{horocyclic predecessor function} $P$ is defined as $$\lim_{n\to\infty} (a_ia_j)^nP_0\bigl( (a_ia_j)^{-n}w\bigr).$$ This stabilizes when $n$ is sufficiently large. We write the set $P^{-m}(w)$ to be the $m$-fold pre-image $P^{-1}(P^{-1}(...(w)...))$. 
		
		Then $w$ and $v$ are connected by an edge exactly when $b_\gamma(w)=b_\gamma(v)$ and the infimum distance between $P^{-m}(w)$ and $P^{-m}(v)$ is bounded above independent of $m$. If this last condition holds (for any pair of words, regardless of whether they are on the same horosphere), we will say the words \textit{have close horocyclic successors}. Equivalently, $P^{-m}(w)$ consists of the words $ww'$ where $w'$ may be written from the state $S(w)$. So two words have close horocyclic successors when their $S$ states allow for close successors.
		
	\end{definition}
	
	By \cite{CGSR} Lemma 7.4, the divergence graph is connected for $W_\Gamma$ where $\Gamma$ satisfies the standing assumptions.
	
	This definition presents three difficulties. First of all, it is defined only on some of the points in a horosphere. Second of all, the operation $P^{-m}$ may be hard to compute for a shortlex word $w$. Third of all, it would be convenient to consider the a priori more restrictive condition that $(w,v)$ is an edge when $b_\gamma(w)=b_\gamma(v)$ and there are geodesic rays $\eta_w$ and $\eta_v$ starting at $w$ and $v$ respectively so that $\eta_w$ can be read as an infinite edge sequence from the state $S(w)$ and the same for $\eta_v$, and so that $d(\eta_w(n),\eta_v(n))$ is bounded independent of $n$. (i.e. instead of having a sequence of pairs of segments stay close for all time, we would like to have a single pair of rays that stay close for all time). We will address these difficulties in Proposition \ref{EveryStateLarge}, Remark \ref{DescribingPredecessors}, and Proposition \ref{CloseSuccessorRays} respectively.
	
	For the first difficulty, the next Proposition provides sufficient conditions on the defining graph of a RACG for there to be no small states.
	
	\begin{proposition} \label{EveryStateLarge}
		
		Let $\Gamma$ be a graph satisfying the standing assumptions, and where every vertex $v$ has a vertex $v'$ at distance at least $3$ from $v$. Then in $M_{lex}$ is the starting vertex is the only pre-large vertex, and the remainder of the graph is connected.
		
	\end{proposition}
	
	We note that, e.g. $\Gamma$ having diameter at least $5$ is sufficient for this last condition. 
	
	\begin{proof}
		
		We show, from any state, we can reach $S(a_k)$ for each $k$. Suppose $w$ is a shortlex word with last letter $a_l$, and fix $a_l=a_{j_0}, a_{j_2}, ...a_{j_m}=a_k$ a path in $\Gamma$ from $a_l$ to $a_k$. We denote by $a_{j_n}'$ a vertex at distance at least $3$ from $a_{j_n}$.
		
		Since $d_\Gamma(a_{j_n}, a_{j_n}')\ge 3$, the stars of $a_{j_n}$ and $a_{j_n}'$ do not intersect. Then by the definition of the shortlex machine, any shortlex word ending with $a_{j_n}'a_{j_n}$ is in shortlex state $S(a_{j_k})$. Therefore, the word $wa_{j_1}'a_{j_1}...a_{j_m}'a_{j_m}$ will end in state $S(a_{j_m})=S(a_k)$ as desired, as long as this word is shortlex.
		
		By the definition of the shortlex machine, $S(v)$ is always a subset of the star of the last letter of $v$. That is, we can always follow a shortlex word $v$ with any letter not commuting with the last letter of $v$. By the triangle inequality, the distance between $a_{j_k}$ and $a_{j_{k+1}}'$ is at least $2$, so that the two do not commute. Therefore $wa_{j_1}'a_{j_1}...a_{j_m}'a_{j_m}$ is shortlex.
		\end{proof}
	
	For the rest of this section, it is assumed that there are no small states in $M_{lex}$.
	
	The second difficulty we mentioned with the divergence graph is determining the operation $P$. For this purpose, we prove the following lemma.
	
	\begin{lemma}
		
		Let $w=_{Geo}w_{pref}w_{suff}$ be the prefix-suffix decomposition of $w$. If $w_{pref}$ is empty or begins with $a_j$, then after multiplying by $(a_ia_j)^{-2}$ the form of $(a_ia_j)^{-n}w$ stabilizes. If instead $w_{pref}$ begins with $a_i$, then after multiplying by $(a_ia_j)^{-\left(\left\lceil\frac{|w_{pref}|}{2}\right\rceil+2\right)}$, the form of $(a_ia_j)^{-n}w$ stabilizes.
		
		Once stabilized, if $|w_{pref}|$ was odd, then $(a_ia_j)^{-m}w=_{Geo} w_1a_jw_2(a_ia_j)^ka_iw_3a_jw_4$, where $k\ge 1$, $w_1$ consists of letters commuting with $a_j$ and $a_i$, and preceding $a_j$, $w_2$ consists of letters commuting with $a_j$ and $a_i$ and preceding $a_i$, $w_3$ consists of letters commuting with and preceding $a_j$, and $w_1w_2w_3w_4=_{Geo}w_{suff}$.
		
		If instead $|w_{pref}|$ was even, then $(a_ia_j)^{-m}w=_{Geo} w_1a_jw_2(a_ia_j)^kw_5a_iw_6$ where $k$, $w_1$, and $w_2$ are as before, $w_5$ consists of letters commuting with and preceding $a_i$, \newline $w_1w_2w_5w_6=_{Geo}w_{suff}$.
		
	\end{lemma}
	
	Note that this only includes half of the expressions $\gamma(n)^{-1}w$, but does show directly that the values of $\lim_{n\to \infty} M_{lex}(\gamma(n)^{-1}w)$ converge for all $w$ along the subsequence $n_i=2i$. See remark \ref{FullConvergence} for an argument that not even this subsequence is necessary.
	
	\begin{proof}
		Start by canceling as many letters as possible from $(a_ia_j)^{-n}w$. Since $w$ was shortlex, the only possible cancellations are between $(a_ia_j)^{-n}$ and letters of $w$. As $w_{suff}$ is not equal to any geodesic word beginning with $a_i$ or $a_j$, cancellation only happens with negative suffixes. For powers at least those given, we cancel any negative prefix and leave a positive prefix of size at least $4$. So we are left to find a shortlex form of the geodesic words $a_ja_i...a_iw_{suff}$ and $a_ja_i...a_j w_{suff}$, and show that these will not depend on the anything more than the parity of the positive prefix.
		
		We wish to show that, in shortlex form, these words are of the form $w_1a_jw_2\mathbf{(a_ia_j)^ka_i}w_3a_jw_4$ or $w_1a_jw_2\mathbf{(a_ia_j)^k}w_5a_iw_6$. That is, it suffices to show that no letters of $w_{suff}$ end up inside the bolded subwords. In order for a letter $a_l$ in $w_{suff}$ to commute into these subwords, $a_l$ must commute with both $a_i$ and $a_j$. If $a_l$ follows both $a_i$ and $a_j$, then $a_l$ will not end up in the bolded subword (if it did it should instead shift right). 
		
		So let $a_l$ precede one of $a_i$ or $a_j$. Since $a_i$ and $a_j$ do not commute, any two letters $a_k$ and $a_l$ commuting with both of them must commute with one another. Then if $a_l$ precedes $a_j$, we achieve a shortlex-earlier word by shifting this copy of $a_l$ into $w_1$, while if $a_l$ precedes $a_i$ but not $a_j$, then we reach an earlier word by moving this $a_l$ into $w_2$. 		
	\end{proof}
	
	Following the above, we will say a \textit{horocyclically shortlex word} is a word of the form $w_1a_jw_2(a_ia_j)^ka_iw_3a_jw_4$ or $w_1a_jw_2(a_ia_j)^kiw_5a_iw_6$ where the $w_i$ are as before.  A \textit{horocyclic suffix} is the geodesic suffix $w=w_1w_2w_3w_4$ or $w=w_1w_2w_5w_6$. If we wish to be more precise, given a shortlex suffix $w_{suff}$, we will say it has an \textit{odd horocyclic suffix} and an \textit{even horocyclic suffix} form, depending on the parity of $|w_{suff}|$. If $|w_{suff}|$ is even, then its even horocyclic suffix form is $w_1w_2w_5w_6$, while its odd horocyclic suffix form is $w_1w_2w_3w_4$, and vice versa if $|w_{suff}|$ is odd. As a result, if we wish to consider a horosphere where $b_{\gamma}$ is even, we will consider even horocyclic suffixes, and if we wish to consider a horosphere where $b_{\gamma}$ is odd, we will consider odd horocyclic suffixes.
	
	Two remarks are in order. Firstly and importantly we highlight when the predecessor of a word may have the same suffix and different prefix.
	
	\begin{remark}\label{DescribingPredecessors}
		
		Since $P(w)=\lim_{n\to\infty} (a_ia_j)^nP_0\bigl( (a_ia_j)^{-n}w\bigr)$, $P(w)$ deletes the last letter of $w_4$ (resp. $w_6$) if it is nonempty. If it is empty, then $P(w)$ is a word with the same suffix and a prefix that is one letter more negative (lengthening a nonpositive prefix or shortening a positive one), after which $P(w)_6=w_3$ (resp. $P(w)_4=w_5$). The words $w_1$ and $w_2$ never change upon application of $P$.
		
		In the other direction, $P^{-1}(w)$ consists of all words $wa_k$ where $a_k$ is not in $S(w)$. Similarly $P^{-m}(w)$ contains $ww'$ where $w'$ is a shortlex word that can be written starting at the state $S(w)$. 
		
		$P^{-1}(w)$ contains a word of different prefix length when both $S(w)$ permits a letter $a_i$ or $a_j$ to be written, and this letter can be rearranged to reach the prefix. This happens when $w_3$ (resp. $w_5$) is empty and $w_4$ (resp. $w_6$) commutes with and precedes $a_i$ (resp. $a_j$) or is empty. In such a case, the first letter of $w_4$ (resp. $w_6$), if existent, must not commute with $a_j$ (resp. $a_i$) because if it did it would necessarily shift left to join $w_2$ (resp. $w_1$). After changing the prefix, the old $w_4$ (resp. $w_6$) becomes the new $w_5$ (resp. $w_3$).
		
		Equivalently, $P^{-1}(w)$ contains a word of different prefix exactly when $a_i$ (resp. $a_j$) is in $F_{a_i, a_j}(w_{suff})$ (so that each letter in the suffix commutes with $a_i$), and that $a_i$ is in neither $M_{lex}(w_3)$ nor $M_{lex}(w_4)$ (so that both of them consist of letters preceding $a_i$.) Since we wish to work with horocyclic suffixes, we can instead check that $a_i$ is in $F_{a_i}(w_1w_2w_3w_4)$, since the machines $F_\mathscr{B}$ make no assumption that their inputs are shortlex.
		
		As a result of this, the only words whose prefix can change twice in a row must have trivial $w_3$ and $w_4$ (resp. $w_5$ and $w_6$), so that they consist entirely of letters commuting with both $a_i$ and $a_j$, and preceding at least one of them.
		
	\end{remark}
	
	A second remark concerns the fact that $(a_ia_j)^{-n}w$ only accounts for half of the words $\gamma(n)^{-1}w$.
	
	\begin{remark}\label{FullConvergence}
		
		If we take instead the words $a_i (a_ia_j)^{-n}w = \gamma(n)^{-1}w$, where $n$ is again sufficient to cancel any positive prefix, then the form of the resulting word may change. In particular, letters that commute with both $a_i$ and $a_j$ which precede $a_i$ but follow $a_j$ (i.e. those in $w_2$) will now rearrange to the beginning of the word, along with those letters of $w_1$ which commute with and precede both letters. The remaining letters of $w_1$, which precede $a_j$ but follow $a_i$, will remain where they were. Of course, either $a_i$ precedes $a_j$ or vice versa, so only one of these two actually happens.
		
		Regardless, it is a straightforward exercise with the definition of the shortlex machine to show that $M_{lex}(a_i (a_ia_j)^{-n}w) = M_{lex}((a_ia_j)^{-n}w)$.
		
	\end{remark}
	
	For our purposes, we will consider words labeled by their horocyclic suffixes. It is an inconvenient fact that converting between shortlex and horocyclic suffixes is a slow operation (it is $O(|w|^2)$, while all other operations so far are $O(|w|)$). To start, we show that we can generate horocyclic suffixes efficiently.

	\begin{lemma}
		
		There are finite-state machines $M_{1,2,3,4}$ and $M_{1,2,5,6}$ which accept respectively the set of horocyclic suffixes $w=w_1w_2w_3w_4$ and $w=w_1w_2w_5w_6$.
		
	\end{lemma}
	
	Of course, these are not the set of odd and even horocyclic suffixes. We will describe machines for those words later.

	\begin{proof}
		
		We will prove this for $M_{1,2,3,4}$, as the proof for $M_{1,2,5,6,}$ is very similar. We would like to say that the language $\mathscr{L}(M_{1,2,3,4})$ is a concatenation of languages $L_1L_2L_3L_4$, and show that each one is regular. This almost works.
		
		A priori, $w_1$ is a word in $Star_<(a_j)\cap Star(a_i)$ and $w_2$ is a word in $Star_<(a_i)\cap Star(a_j)$ that cannot be rearranged to begin with a letter in $Star_<(a_j)\cap Star(a_i)$. But by Lemma \ref{RipsGraphTraversability}, $Star(a_i)\cap Star(a_j)$ is a clique. Therefore, the letters of $w_1$ and $w_2$ commute, and so $w_2$ does not contain any letters in $Star_<(a_j)\cap Star(a_i)$. Therefore, $L_1$ and $L_2$ are $ShortLex|_{Star_<(a_i)\cap Star_<(a_j)}$ and $ShortLex|_{Star_<(a_i)\cap Star_>(a_j)}$ and both are regular.
		
		We would like to say that $L_3$ is $ShortLex|_{Star_<(a_j)}\cap F_{Star(a_i)}$ and $L_4$ is $ShortLex\cap F_{Star_\le(a_j)}$, but this is not quite right. First of all, $w_4$ might rearrange to begin with a letter in $Star_<(a_i)\cap Star_>(a_j)$ that then commutes all the way to the beginning of $w_3$. Such a letter should appear in $w_1$ or $w_2$. We prevent this from happening by intersecting with the first-letter excluder language $F_{Star_<(a_i)\cap Star(a_j)}$. Secondly, if $w_3$ is empty and a copy of $a_i$ commutes to the front of $w_4$, then this letter should be in the prefix. Since the first letter of $w_3$ cannot commute with $a_i$, we can prevent this by intersecting with the first-letter exclude $F_{a_i}$. We then obtain a single language 
		$$L_{3,4} = \Bigl((ShortLex|_{Star_<(a_j)}\cap F_{Star(a_i)})(ShortLex\cap F_{Star_\le(a_j)\cup \{a_i\}})\Bigr)\cap F_{\{a_i\}\cup(Star_<(a_i)\cap Star(a_j))}.$$
		Then the language $\mathscr{L}(M_{1,2,3,4})$ is the concatenation $L_1L_2L_{3,4}$, which is regular.
	\end{proof}
	
	\begin{corollary} \label{DivergenceGraphVertexFSM}
		
		There are machines $M_{odd}$ and $M_{even}$ that accept odd and even horocyclic suffixes.
		
	\end{corollary}
	
	\begin{proof}
		
		Since $L_{1,2,3,4}$ and $L_{1,2,5,6}$ are regular, as are the set of words of even and odd length, $\mathscr{L}(M_{odd})=(L_{1,2,3,4}\cap Even(V))\cup (L_{1,2,5,6}\cap Odd(V))$ is regular. Similarly, \newline $\mathscr{L}(M_{even})=(L_{1,2,3,4}\cap Odd(V))\cup (L_{1,2,5,6}\cap Even(V))$ is regular. 			
	\end{proof}
	
	In practice, we will not use these machines because a faster implementation is possible to find the list of horocyclic suffixes. Nevertheless, it is helpful to know that there is a regular language of horocyclic suffixes on each horosphere.
	
	Finally, let us show that, if $w$ and $v$ have close horocyclic successors, then there are rays $\eta_w$ and $\eta_v$ of horocyclic successors to $w$ and $v$ that stay close. We will use the following standard lemma about quadrilaterals in $\delta$-hyperbolic spaces, whose proof we include for completeness.
	
	\begin{lemma}[e.g. \cite{CGSR} Lemma 3.9.] \label{ThinAlmostBigons}
		
		Suppose $w$ and $v$ are any two points in a $\delta$-hyperbolic metric space at distance at most $M$ from one another, and $\sigma_w$ and $\sigma_v$ are geodesic segments of length $L$ whose endpoints $\sigma_v(L)$ and $\sigma_w(L)$ are again within $M$ of one another. Then for all $n$, $\sigma_w(n)$ is within $3M+4\delta$ of $\sigma_v(n)$. 
		
	\end{lemma}
	
	\begin{proof}
		
		Take geodesic segments $\overline{vw}$ between $w$ and $v$ and $\overline{\sigma_v(L)\sigma_w(L)}$ between $\sigma_w(L)$ and $\sigma_v(L)$. We start by showing that the quadrilateral $v w \sigma_w(L)\sigma_v(L)$ is $2\delta$-thin, i.e. that each segment is within the $2\delta$ neighborhood of the union of the other three. To see this, e.g. for $\sigma_w$, draw a geodesic between $v$ and $\sigma_w(L)$. Then the triangle $v w \sigma_w(L)$ is $\delta$-thin so that $\sigma_w$ is within $\delta$ of the union of $\overline{vw}$ and $\overline{v\sigma_w(L)}$. But the triangle $\sigma_v(L) v \sigma_w(L)$ is also $\delta$-thin, so that $\overline{v\sigma_w(L)}$ is within $\delta$ of the union of $\sigma_v$ and $\overline{\sigma_v(L)\sigma_w(L)}$. Therefore, this quadrilateral is $2\delta$-thin by the triangle inequality.
		
		Let $n$ be strictly between $M+2\delta$ and $L-M-2\delta$.	This $d(\sigma_v(n), v)>2+2\delta$, so that $d(\sigma_v(n), \overline{vw})>2\delta$ by the triangle inequality. Similarly, $d(\sigma_v(n),\overline{\sigma_v(L)\sigma_w(L)})>2\delta$. Therefore, $\sigma_v(n)$ is within $2\delta$ of some $\sigma_w(m)$. We see that $m\ge n-M-2\delta$, because otherwise the path from $v$ to $\sigma_v(n)$ that concatenates the segments $\overline{vw}$, $\sigma_w|_{[0,m]}$, $\overline{\sigma_w(m)\sigma_v(n)}$ has length less than $n$. The same argument in revers shows that $m\le n+M+2\delta$. So since $\sigma_w$ is geodesic, $d(\sigma_v(n),\sigma_w(n))\le d(\sigma_v(n),\sigma_w(m))+d(\sigma_w(m),\sigma_w(n))\le 2\delta+(M+2\delta)=M+4\delta$. This establishes the desired bound for $n$ not near the endpoints of the interval $[0,L]$.
		
		If $n$ is within $M+2\delta$ of $0$ or $L$, then \begin{align*}
		d(\sigma_w(n),\sigma_v(n))&\le d(\sigma_w(n),w)+d(2,v)+d(v,\sigma_v(n))\\
		&\le M+2\delta+M+M+2\delta\\
		&=3M+4\delta
		\end{align*}
		
		as required.
	\end{proof}
	
	The existence of close rays now follows from a discrete version of the Arzela-Ascoli Theorem, which may also be thought of as a diagonalization argument.
	
	\begin{proposition} \label{CloseSuccessorRays}
		
		Let $w$ and $v$ have close successors, bounded by a constant $M$. Then there are geodesic rays $\eta_w$ and $\eta_v$ that stay at distance at most $3(3M+4\delta)+4\delta$ apart for all time, such that $\eta_w$ can be read starting from $S(w)$ and $\eta_v$ can be read starting from $S(v)$. 
		
	\end{proposition}
	
	\begin{proof}
		
		Since $P^0(w)=w$ and $P^0(v)=v$, the assumption shows that $w$ and $v$ are distance $M$ apart. Take a sequence of words $w=w_0, w_1, w_2,...$ and $v=v_0, v_1, v_2,...$ so that $w_i$ (resp. $v_i$) is in $P^{-i}(w)$ (resp. $P^{-i}(v)$), and so that $d(w_i, v_i)<M$. Denote  by $\overline{w_0w_i}$ (resp. $\overline{v_0v_i}$) the horocyclically shortlex segment between $w_0$ and $w_i$ (resp. between $v_0$ and $v_i$). 
		
		By Lemma \ref{ThinAlmostBigons}, the segments $\overline{v_0v_i}$ and $\overline{w_0w_i}$ are within  $4M+3\delta$ of one another for all time.
		
		For $j\le i$, denote $a_{i,j}$ the label of the $j^{th}$ edge that $\overline{w_0w_i}$ traverses. Since there are only finitely many labels, there is some letter, which we will denote $b_1$ that appears infinitely many times as $a_{i_1}$. Pass to the subsequence of $w_i$ whose first letter is $b_1$ and repeat with $a_{i_2}$ to get a letter $b_2$, etc. Each finite substring $b_1,...b_k$ appears as the beginning of some, and indeed infinitely many, of the $w_i$, so that $b_1,...b_k$ is a horocyclic successor to $w$ for each $k$. Then the ray $\eta_w=b_1b_2...$ is a horocyclic successor to $w$ and $\eta_w(n)$ is within $3M+4\delta$ of $v_{i_n}(n)$, where the index $i_n$ depends on $n$. 
		
		Each segment $\overline{v,v_{i_n}(n)}$ is therefore within the $3(3M+4\delta)+4\delta$ neighborhood of $\eta_w$ by another application of Lemma \ref{ThinAlmostBigons}. We then repeat the prior diagonalization argument on these segments to extract a horocyclic successor ray $\eta_v$ to $v$ that stays within $9M+16\delta$ from $\eta_w$.
	\end{proof}
	
	\subsection{An equivalent condition to close successor rays}
	\label{subsec:CloseSuccessors}
	
	In this subsection, we will prove a much more down-to-earth equivalent condition for horocyclically shortlex words to have close successors. In later subsections, we will show that this condition is checkable by certain FSMs.
	
	In order for $w$ and $v$ to have close successors, it will be necessary to cancel infinitely many successor letters of $w$ with those of $v$. To keep track of cancellations, the following preliminary lemma will be useful.

	\begin{lemma}\label{ShortlexNonLinking}
		Let $w=a_{i_1}a_{i_2}...a_{i_n}$ and $v=a_{j_1}a_{j_2}...a_{j_m}$ be shortlex words, and consider $w^{-1}v=_{Geo}a_{i_n}...a_{i_1}a_{j_1}...a_{j_m}$. If a letter $a_{i_k}$ of $w$ cancels, it cancels with a unique letter $a_{j_l}$ of $v$. If so, letters after $a_{i_k}$ cannot cancel with letters before $a_{j_l}$.		
	\end{lemma}
	
	\begin{proof}
		Let $a_{i_k}$ be the first letter of $w$ that can cancel with two letters, $a_{j_l}$ and $a_{j_{o}}$ of $v$, and let $o>l$. Then $a_{i_k}=a_{j_l}=a_{j_{o}}$ in order to cancel, and there is a letter $a_{j_{l'}}$ with $l<l'<o$ failing to commute with $a_{j_l}=a_{j_{o}}$ because no cancellation can be possible between letters of $v$. So in the product $w^{-1}v$, $a_{j_{l'}}$ must cancel with some letter of $w$ prior to $a_{i_k}$ canceling with $a_{j_{o}}$. But then the $a_{j_{l'}}$ appearing in $w$ must be before $a_{i_k}$, which prevents $a_{i_k}$ from canceling with $a_{j_l}$.
		
		So every letter of $w$ cancels with a unique letter of $v$ if it cancels at all. 
		
		Now suppose that $k<p$ such that $a_{i_k}$ and $a_{i_p}$ both cancel, with no other letters canceling in between them, and they cancel with $a_{j_l}$ and $a_{j_o}$ respectively for $l>o$. Take $k, p$ to be the first such pair with linked cancellation.  $w^{-1}v=_{Geo}a_{i_n}...a_{i_1}a_{j_1}...a_{j_m}$, so that these letters appear in order $a_{i_{p}}...a_{i_k}...a_{j_o}...a_{j_{l}}$. Therefore, to cancel, we require that $a_{i_k}$ and $a_{i_p}$ commute. Since none of the letters between $a_{i_k}$ and $a_{i_{p}}$ cancel, it follows also that $a_{i_{p}}$ commutes with the intermediate letters, so that $a_{i_k}$ is earlier in the alphabetic order than $a_{i_{p}}$. 
		
		Now in order for the letters $a_{j_l}=a_{i_{k}}$ and $a_{j_{o}}=a_{i_p}$ to appear in this order in the shortlex word $v$, we require a letter $a_{j_{l'}}$ between them not commuting with $a_{j_{l}}$. Then the cancellation between $a_{i_k}$ and $a_{j_{l}}$ requires that $a_{j_{l'}}$ cancel first with a letter before $a_{i_k}$, and that $a_{j_o}=a_{i_{p}}$ commutes with $a_{j_{l'}}$. As a consequence, the three letters $a_{j_{l'}}, a_{j_{l}}=a_{i_k}, a_{j_o}=a_{i_p}$ appear in this order in $w$ and in the order $a_{j_o}, a_{j_{l'}}, a_{j_{l}}$ in the word $v$. Now if $a_{j_{l'}}$ precedes $a_{j_o}$ in the shortlex order, then new letter $a_{j_{l''}}$ failing to commute with $a_{j_{l'}}$ must be inserted between the two in order for $v$ to be shortlex, which will require a matching copy before $a_{j_{l'}}$ in $w$. There are only finitely many letters before $a_{i_k}$, so iterating this argument finitely many times if needed, we may assume that $a_{j_l''}$ instead follows $a_{j_o}$ in the shortlex order.
		
		If instead $a_{j_{l'}}>a_{j_o}=a_{i_p}$, then we need a letter $a_{i_{p'}}$ between the two in $w$ not commuting with $a_{i_{p}}=a_{j_o}$. This new letter cannot be between $a_{i_k}$ and $a_{i_{p}}$ since by assumption no cancellation happens between the two of them, so it must occur before $a_{i_k}$. But then $a_{j_{l''}}$ and $a_{i_{p'}}$ constitute an earlier alternating pair than $a_{i_k}$ and $a_{i_{p}}$ in violation of the assumption. We thus repeatedly insert  letters between $a_{j_l}$ and $a_{j_{o}}$ and matching ones before $a_{i_k}$ so that the new letter commutes with $a_{j_l}=a_{i_{k}}$. $a_{i_{k}}=a_{j_l}$ is required to commute with this letter, so another letter is required before to preserve the shortlexness of $v$. But since there are finitely many letters between $a_{j_l}$ and $a_{j_{o}}$, this cannot continue indefinitely.
	\end{proof}
	
	Of course, we need to consider cancellation between horocyclically shortlex words, and it will be inconvenient to keep track of the prefixes. The following lemma shows that the same result extends to horocyclic suffixes. 
	
	\begin{lemma} \label{HorocyclicNonLinking}
		
		Let $p_1$, $p_2$, $w$, and $v$ be shortlex words so that $p_1w$ and $p_2v$ are (not necessarily shortlex) geodesics. Denote by $w'$ and $v'$ the order the letters of $w$ and $v$ appear in after shortlex rearrangement of $p_1w$ and $p_2v$. 
		
		Label the letters of $w'$ by $a_{i_1}...a_{i_n}$ and the letters of $v'$ by $a_{j_1}...a_{j_m}$ as before. Then in the difference $(p_1w)^{-1}p_2v$, the letters of $w'$ which cancel with letters of $v'$ do so uniquely, and if the letters $a_{i_k}$ and $a_{i_{k'}}$ with $k'>k$ cancel with the letters $a_{j_l}$ and $a_{j_{l'}}$, with $l'>'$, then $a_{i_k}$ cancels with $a_{j_l}$ and $a_{i_{k'}}$ cancels with $a_{j_{l'}}$
		
	\end{lemma}
	
	Note that this lemma covers cases where the words $w$ and $v$ are suffixes of potentially-different lengths.
	
	\begin{proof}
		
		To see that $a_{i_k}$ cancels with a unique $a_{j_l}$, apply the previous argument. If it cancels with $a_{j_l}$ and $a_{j_{l'}}$, then there is a letter $a_{j_o}$ between the two which fails to commute with $a_{j_l}$ in order for $pv$ to be geodesic. But such a letter prevents $a_{i_k}$ from commuting to reach $a_{j_{l'}}$. 
		
		Since the order of the $a_{i_k}$ and $a_{j_l}$ is induced by the order in a shortlex word, the letters $a_{i_{k'}}$ for $k'>k$ appear later than $a_{i_k}$ in the shortlex word equivalent to $p_1w$, and the same holds respectively for letters $a_{j_{l'}}$ with $l'<l$ appear earlier in the shortlex word for $p_2v$. Thus if $a{i_k}$ and $a_{j_l}$ cancel, then a cancellation between $a{i_{k'}}$ and $a_{j_{l'}}$ with $k'>k$ and $l'<l$ violates Lemma \ref{ShortlexNonLinking} applied to the shortlex words equivalent to $p_1w$ and $p_2v$.	
	\end{proof}
	
	Colloquially, these lemmas say that cancellation between shortlex words or between horocyclic suffixes must be unlinked. We will often refer to this fact in these terms without a formal citation. As a result, we obtain a preliminary necessary condition for $w$ and $v$ to have close successors.
	
	\begin{lemma}\label{DivergenceGraphEdgeNecessaryCondition}
		
		Let $w_{horosuff}$ and $v_{horosuff}$ be horocyclic suffixes, and multiply them by sufficient prefixes so that the resulting words $w$ and $v$ are horocyclically shortlex and on the same horosphere. Write $w^{-1}v=_{Geo}w_{sub}^{-1}v_{sub}$ where $w_{sub}$ and $v_{sub}$ are the subwords of $w$ and $v$ that do not cancel. Suppose $w_{sub}^{-1}v_{sub}$ contains a pair of letters $a_k$ and $a_l$ that do not commute, and that, for any horocyclic successors $ww'$ and $vv'$ to $w$ and $v$, this pair $a_k$ and $a_l$ still cannot be canceled. Then $w_{horosuff}$ and $v_{horosuff}$ do not have close horocyclic successors.
		
	\end{lemma}
	
	\begin{proof}
		
		Suppose that $w$ and $v$ have horocyclic successors $ww'$ and $vv'$. Recall that $w'$ and $v'$ are necessarily shortlex.
		
		Consider letters that cancel between $w'$ and $v'$. Since $a_k$ and $a_l$ cannot be canceled, letters $a_m$, $a_n$ commuting across $w_{sub}^{-1}v_{sub}$ must commute with $a_k$ and $a_l$. By the first part of Lemma \ref{RipsGraphTraversability}, the two commute.
		
		Therefore, the letters of $w'$ that can cancel with those in $v'$ must all be from a single clique in $\Gamma$. No such letter $a_k$ appears more than once in $w'$ and commutes across to $v'$, because these two letters $a_k$ then would cancel with one another, so that $w'$ is not shortlex. So there is a bound on the number of cancellations possible between $w'$ and $v'$.
		
		However, since $w_{sub}^{-1}v_{sub}$ is boundedly long, only boundedly many cancellations may occur with letters of $w_{sub}^{-1}v_{sub}$. Therefore, only boundedly many letters cancel out of $w'$ and $v'$. This means that the reduced length of $(ww')^{-1}vv'$ grows without bound as the lengths of $w'$ and $v'$ grow, so that the two do not have close horocyclic successors.	
	\end{proof}
	
	\begin{proposition} \label{NonCommutingLettersInDifferentWords}
		
		Retain the notation of Lemma \ref{DivergenceGraphEdgeNecessaryCondition}. If the non-commuting letters $a_k$ and $a_l$ in $w_{sub}^{-1}v_{sub}$ come from different words, and $|w_{suff}|=|v_{suff}|$, then these letters will never cancel in any horocyclic successor.
		
	\end{proposition}
	
	\begin{proof}
		
		Suppose that $a_k$ is a letter of $w_{sub}^{-1}$ and $a_l$ is a letter of $v_{sub}$. At most one of the two is a prefix letter. If one is, say $a_k$, then by Remark \ref{DescribingPredecessors}, the only way to change $v$'s prefix in a successor is if $v_{suff}$ commutes entirely with the prefix letter to be written. So $a_k$ can only be canceled if $v_{suff}$ commutes entirely with $a_k$, which it does not by assumption.
		
		Instead let both be suffix letters. Since the two do not commute, cancellation of either one would first require the other to cancel, which would create linked cancellation.	
	\end{proof}
	
	\begin{corollary} \label{NonCommutingLettersCor}
		
		Retain the notation of Proposition \ref{NonCommutingLettersInDifferentWords}. If $(w,v)$ is an edge in the divergence graph, and $w_{sub}^{-1}v_{sub}$ contains a pair of non-commuting letters $a_k$ and $a_l$, then both of these letters appear in either the subword $w_{sub}^{-1}$ or $v_{sub}$, and the letters of the other subword all commute with one another and with $a_k$ and $a_l$.
		
	\end{corollary}
	
	\begin{proof}
		The first part of the conclusion is the contrapositive of Proposition \ref{NonCommutingLettersInDifferentWords}. Then the second part follows from Lemma \ref{RipsGraphTraversability}.
	\end{proof}
	
	Besides the uncancelable clique, every other letter can be canceled by some successor. In fact, there is a single successor in which every potentially-cancelable letter does cancel.
	
	\begin{lemma} \label{MaximalCancellation}
		
		Suppose $w$ and $v$ are horocyclically shortlex words for which the clique $K(v,w)$ of letters in $w_{sub}^{-1}v_{sub}$ is uncancelable. Then $K(v,w)$ includes either all of $w_{sub}^{-1}$ or all of $v_{sub}$ (potentially both). Moreover, if $K(v,w)$ contains all of $w_{sub}^{-1}$ but not all of $v_{sub}$ (resp. all of $v_{sub}$ but not all of $w_{sub}^{-1}$), then there is a horocyclic successor $w'$ to $w$ (resp. a horocyclic successor $v'$ to $v$) such that $(ww')^{-1}v$ (resp. $w^{-1}vv'$) contains only the uncancelable letters in $K(v,w)$.
		
	\end{lemma}
	
	\begin{proof}
		
		Let $a_{k_0}$ and $a_{l_0}$ be the first cancelable letters of $w$ and $v$ respectively, if they exist. If $S(v)$ does not permit $a_{k_0}$, then  $a_{k_0}$ can only be written into a successor of $v$ after another letter $a_{k_0}'$ not commuting with it. But this would mean $a_{k_0}$ cannot be canceled because no letter before $a_k$ can, so that there is no letter to cancel with $a_{k_0}'$. Analogously, $S(w)$ must permit $a_{l_0}$. But then $wa_{l_0}$ and $va_{k_0}$ have linked cancellation, since writing a letter in $w$ after $a_{k_0}$ cannot prevent $a_{k_0}$ from canceling. This proves the first assertion.
		
		WLOG, then let $w$ contain cancelable letters. Denote them $a_{k_0}, ... a_{k_m}$. The first such, $a_k$ is permitted by $S(v)$. The result will then hold by repeated application of this same argument, provided that $S(va_{k_0} ... a_{k_n})$ always permits $a_{k_{n+1}}$. If this does not happen then either $a_{k_{n+1}}$ is forbidden by $S(v)$ and each $a_{k_0}, ... a_{k_n}$ commutes with $a_{k_{n+1}}$, or $a_{k_{n+1}}$ commutes with and precedes some earlier letter $a_{k_o}$ and commutes with $a_{k_{o+1}}, ... a_{k_n}$. In the first case, $a_{k_{n+1}}$ can only be written after $v$ if it is first preceded by an $a_{k_{n+1}}'$ not commuting with it. But this letter is therefore not $a_{k_0}, ... a_{k_n}$ and therefore does not cancel with any letter in $w$, so that $a_{k_{n+1}}$ is not cancelable. In the second case, the letters $a_{k_0}, ... a_{k_{n+1}}$ appear in order in $w$, with additional letters interspersed. Between $a_{k_{o}}$ and $a_{k_{n+1}}$ there appear only uncancelable letters of $w$, and the letters $a_{k_{o+1}}, ... a_{k_n}$ (the other option would be to have some letters in between that cancel with letters of $v$. But by unlinked cancellation this would make $a_{k_o}$ uncancelable). By assumption, $a_{k_{n+1}}$ commutes with each such letter, so that $w$ is not shortlex.
	\end{proof}
	
	As a consequence of this lemma, we can always find a single successor that achieves the maximal amount of cancellation. We next show that determining whether two words have close successors is easy in this setting. When two words $w$ and $v$ with close successors differ by an uncancelable clique, we can find successor rays $\eta_w$ and $\eta_v$ that are close for all time in the simplest way possible.
	
	\begin{lemma}\label{CliqueDifferenceAdjacency}
		
		Suppose $w^{-1}v=_{Geo}w_{sub}^{-1}v_{sub}$ where $w_{sub}^{-1}v_{sub}$ consists of a collection of commuting letters, with none canceling in any successor. Call this clique $K(v,w)$. Then $w$ and $v$ have close successors if and only if there are two nonadjacent letters $a_k$ and $a_l$ commuting with each letter of $K(v,w)$, so that $w(a_ka_l)^n$ and $v(a_ka_l)^n$ are horocyclic successors of $w$ and $v$ respectively.
		
	\end{lemma}
	
	We remark that since $a_k$ and $a_l$ commute with $K(v,w)$ and not with one another (or else $(a_ka_l)$ geodesic), then we only need to check whether $a_k$ is in $S(w)$  (resp. $S(v)$) in order to determine if $w(a_ka_l)^n$ succeeds $w$, (resp. $v(a_ka_l)^n$ succeeds $v$).
	
	\begin{proof}
		
		Sufficiency is clear.
		
		To see necessity, suppose to the contrary that there is no such pair of letters for $K(v,w)$, but that $P^{-*}(w)$ and $P^{-*}(v)$ have close elements for all time. Choose $w$ and $v$ so that $K(v,w)$ is maximal among counterexamples.
		
		There is a letter $a_k$ so that $P(wa_k)=w$ and $P^{-*}(wa_k)$ is bounded infimum distance from $P^{-*}(v)$ (we need infinitely many horocyclic successors to $w$, and there are only finitely many words in $P^{-1}(w)$, so there must be at least one word $wa_k$ so that infinitely many of the chosen horocyclic successors go through $wa_k$). By assumption, at least one of the following holds of $a_k$: 
		
		\begin{enumerate}
			
			\item $a_k$ is not adjacent to $K(v,w)$.
			
			\item $a_k\in S(v)$.
			
			\item For each $a_l$ adjacent to $K(v,w)$, $a_k$ is adjacent to $a_l$.
			
		\end{enumerate}
		
		In case (1), writing $a_k$ following $w$ creates an uncancelable pair. Since $a_k$ and some letter $a_m$ in $K(v,w)$ do not commute, $a_k$ can only be canceled by a letter following $v$, after $a_m$ is canceled, and $a_m$ is assumed uncancelable. Hence by Lemma \ref{DivergenceGraphEdgeNecessaryCondition}, $wa_k$ and $v$ do not have close successors.
		
		In case (2) $a_k\cup K(v,w)$ is a clique. $a_k$ can be written after $v$ only after writing a letter $a_k'$ not commuting with $a_k$. Such a letter cannot cancel with any letters following $wa_k$ because this would violate Lemma \ref{HorocyclicNonLinking}. As a result, $a_k$ cannot cancel. But now $a_k\cup K(v,w)$ is a strictly larger uncancelable clique providing a counterexample, violating the assumed maximality.
		
		In case (3), note that there are finitely many such letters and they all commute with one another by assumption. So we can write only words of length at most $clique(\Gamma)$ in these letters to follow $w$ before we are forced to write a letter that falls into cases (1) or (2).
	\end{proof}
	
	In fact, it suffices to apply Lemma \ref{CliqueDifferenceAdjacency} to the special successors described in Lemma \ref{MaximalCancellation}.
	
	\begin{proposition}\label{DivergenceEdgeCriterion}
		
		Let $w$ and $v$ be two horocyclically shortlex words and suppose there is a horocyclic successor $v'$ to $v$ so that $w^{-1}v'$ consists of the uncancelable clique $K'(v,w)$. Then $w$ and $v$ have close successors if and only if $w$ and $v'$ do.
		
	\end{proposition}
	
	\begin{proof}
		
		Since $v'$ is a successor to $v$, $P^{-*}(v')$ is contained in $P^{-*}(v)$.  Therefore, one implication is immediate. We must show that if $P^{-*}(v)$ is within finite distance of $P^{-*}(w)$, then so is $P^{-*}(v')$.
		
		Suppose that the horocyclically shortlex rays $\eta_w$ and $\eta_v$ start at $w$ and $v$ and stay close forever, but that $\eta_v$ does not pass through $v'$. Write the edges that $\eta_v$ traverses $a_{i_1},a_{i_2},...$. Suppose $a_{j_1}...a_{j_k}$ is the horocyclically shortlex segment from $v$ to $v'$, and let $a_{i_1}, ... a_{i_l}=a_{j_1}, ... a_{j_l}$ but $a_{i_{l+1}}\ne a_{j_{l+1}}$, for some $l<k$. This makes either $a_{i_{l+1}}$ or the copy of $a_{j_{l+1}}$ in $w$ not cancel, since both canceling would create linked cancellation.
		
		Since $\eta_w$ and $\eta_v$ stay close for all time, there is a last uncanceled letter. Let $w''$ and $v''$ be words on $\eta_w$ and $\eta_v$ so that $w''^{-1}v''$ is exactly the clique of letters that will not cancel along the rays $\eta_w$ and $\eta_v$. Call this clique $K'$. $K$ is a proper subclique of $K'$ because it contains either $a_{i_{l+1}}$ or $a_{j_{l+1}}$. By Lemma \ref{CliqueDifferenceAdjacency}, there are letters $a_{m_1}$ and $a_{n_1}$ that do not commute with one another, commute with $K'$ (hence with $K(v,w)$) and that one of which (WLOG $a_{m_1}$) is permitted by $S(v'')$ and $S(w'')$. If $a_{m_1}$ is permitted by $S(v')$ and $S(w)$, then the proof is complete.
		
		For $a_{m_1}$ to be written after $w''$, it must either be permitted by $S(w)$ or be preceded by some $a_{m_2}$ that it does not commute with and which cancels a corresponding copy in $\eta_v$. In such a case, set $a_{n_2}=a_{m_1}$ and repeat with the pair $a_{m_2}$ and $a_{n_2}$. Note that $a_{m_2}$, while it does not necessarily commute with $K(v'',w'')$, must still commute with $K(v,w)$ in order for the pair to cancel. So we repeat the process with $a_{m_2}$ and $a_{n_2}$. After finitely many steps we get a new pair of successor rays $\eta_w'$ and $\eta_v'$ which end with an alternating non-commuting pair $a_{m'}$ and $a_{n'}$, such that both letters commute with $K(v,w)$ and such that $a_{m'}$ is permitted by $S(w)$ and commutes with the letters between itself and $w$. If $a_{m'}$ is permitted by $S(v')$, then we will be done.
		
		The end of the word $w$ consists only of uncancelable letters plus $a_{j_1}, ... a_{j_k}$, but no letters canceling with those in $v$, and $a_{m'}$ commutes with each uncancelable letter by assumption. Therefore, $a_{m'}$ is permitted by $S(a_{j_1}...a_{j_k})$. Since $v'=va_{j_1}...a_{j_k}$, it follows that $a_{m'}$ is forbidden in $S(v')$ only if it is forbidden in $S(v)$ and $a_{j_1}, ... a_{j_k}$ consists entirely of letters commuting with $a_{m'}$. Then $a_{m'}$ being permitted by $S(a_{j_1}, ... a_{j_k})$ shows that $a_{m'}$ must follow each of $a_{j_1}, ... a_{j_k}$. But $a_{m'}$ appears in $\eta_v'$, so this would require $a_{m'}$ to be preceded after $v$ by a letter $a_o$ that it does not commute with. Such a letter must cancel with a letter in $\eta_w'$. However, we cannot put $a_o$ between $w$ and $a_{m'}$, since by assumption $a_{m'}$ commutes with the letters between itself and $w$. Therefore, $a_o$ is a letter of $w$. Moreover, $a_o$ cannot be one of the $a_{j_1} ... a_{j_k}$ as all of these letters commute with $a_{m'}$. We therefore have found a cancelable letter outside of $a_{j_1} ... a_{j_k}$ in violation of Lemma \ref{MaximalCancellation}.
        
	\end{proof}
	
	As a result, given horocyclic words $w$ and $v$ on the same horosphere, such that $w$ has cancelable letters, we can determine whether the two have close successors based only on $S(w)$, $S(v')$, and $K(v,w)$. $S$-state can be computed for any horocyclic suffix by simply pre-pending a small positive prefix to $w_{horosuff}$ and $v'_{horosuff}$ and running the word through the shortlex machine. We also have FSMs that generate the language of horocyclic suffixes. So if $K(v,w)$ and $v'$ can be computed efficiently from horocyclic suffixes, we will have the necessary ingredients to draw the divergence graph.
	
	\subsection{Edges between words of the same suffix length}
	\label{subsec:DivergenceEdgesSameLength}
	
	In this subsection, we lay out an FSM that will calculate $K(v,w)$ and the cancelable word under the assumption that $|w_{suff}|=|v_{suff}|$. A variant in the case where $|w_{suff}|>|v_{suff}|$ will be presented in the next subsection.

	\begin{lemma} \label{MatchedCancellation}
		
		Let $w$ and $v$ be words so that $|w_{suff}|$ differs from $|v_{suff}|$ by at most $1$, and so that $|w_{suff}|\le |v_{suff}|$. Let $w_{HoroSuff}$ and $v_{HoroSuff}$ be the horocyclic suffixes, and let \newline $w_{HoroSuff}=w_1w_2w_3w_4$ (resp. $w_1w_2w_5w_6$).
		
		Then letters of $w_1$ only cancel with letters of $v_1$, and letters of $w_2$ only cancel with letters of $v_2$. If $|w_{suff}|=|v_{suff}|$, then the same holds for the pairs $(w_3, v_3)$ and $(w_4, v_4)$ (resp. for $(w_5, v_5)$ and $(w_6, v_6)$). If $|w_{suff}|=|v_{suff}|-1$, then letters in $w_4$ only cancel with letters in $v_5a_iv_6$ (resp. letters in $w_6$ only cancel with letters in $v_3a_jv_4$).
		
	\end{lemma}
	
	\begin{proof}
		
		This is a direct application of Lemma \ref{ShortlexNonLinking} to the horocyclically shortlex forms of $w$ and $v$. The prefix copy of $a_j$ following $w_1$ cancels its counterpart following $v_1$, and the same for the copies of $a_i$ following $w_2$ and $v_2$. 
		
		If $|w_{HoroSuff}|=|v_{HoroSuff}|$, then the same will be true of the final copies of $a_j$ separating $w_3$ from $w_4$ and $v_3$ from $v_4$ (resp. copies of $a_i$ separating $w_5$ from $w_6$ and $v_5$ from $v_6$). If instead $|w_{HoroSuff}|=|v_{HoroSuff}|-1$, then when written in horocyclically shortlex forms, there is one more prefix letter in $v$ than in $w$. This means the second-last prefix letter in $v$ cancels $a_j$ following $w_3$ (resp. the $a_i$ following $w_5$).
	\end{proof}
	
	We will describe the algorithm for how to compute $K(v,w)$ and the canceling word in terms of states and transitions. In the following proposition, there may a priori be infinitely many states.
	
	\begin{proposition}\label{DivergenceGraphEdgeAlgorithm}
		
		Suppose $w_{horosuff}$ and $v_{horosuff}$ be horocyclic suffixes of equal length. There is an algorithm whose runtime is linear in $|w_{horosuff}|$ which either certifies that $w$ and $v$ do not have close successors, or whose final state computes $K(v,w)$ as well as the word of cancelable letters and whether those cancelable letters appear in $w$ or $v$.
		
	\end{proposition}
	
	\begin{proof}
		Note that the form of a horocyclic suffix is determined by its length and the value of the Busemann function, so $|w_{suff}|=|v_{suff}|$ and $b_\gamma(w)=b_\gamma(v)$ implies that they are of the same form.
		
		The states of this algorithm will consist of quadruples of words $u_1, u_2, u_3, u_4$ or $u_1, u_2, u_5, u_6$, which contain the potentially-cancelable letters, a set $K$ of uncancelable letters, a binary bit $b$ valued in $\{w, v\}$ indicating which word the cancelable letters belong to, and two integers $n_w$ and $n_v$ indicating which subword the words $w$ and $v$ have progressed to. $b'$ will denote the complementary value to $b$. For readability, we will use function notation to describe the values of these variables after each step. So $u_i(j)$ denotes the value of $u_i$ after the $j^{th}$ step, and similarly for the other values. By a slight abuse of notation, on step $j+1$, $u_i(j)$ denotes the value at the end of step $j$, while $u_i(j+1)$ denotes whatever the current value of $u_i$ is, although it is possible for it to change again before the end of step $j+1$. The same goes for the other variables. We will explicitly say when this can happen each time as a reminder.
		
		The initial state will consist of 4 empty strings, the empty set, $n_w=n_v=1$, and an arbitrary value of $b$. On step $j+1$, we will assume that the input pair is $(a_k, a_l)$, and that these two letters commute with $K(j)$. We must show how to update the clique $K(j)$, the words $u_i(j)$, the integers $n_{w}(j)$ and $n_v(j)$, and the bit $b(j)$. We will refer to the input from word $b(j)$ as the $j^{th}$\textit{ adding letter}, and the input from word $b(j)'$ as the $j^{th}$ \textit{canceling letter}. A step in which we compute these updates is described below in substeps.
		
		\textbf{Substep 1: Update $n_w(j)$ and $n_v(j)$.}
		
		The update rule for the $n_v(j)$ and $n_w(j)$ is straightforward and does not depend on any of the other pieces. Whenever a pair $(a_k, a_l)$ is read, $n_w(j+1)=n_w(j)$ (resp. $n_v(j+1)=n_v(j)$) if $a_k$ (resp. $a_l$) was an allowed letter in subword $n_w(j)$ (resp $n_v(j)$). Otherwise, $n_w(j+1)$ (resp. $n_v(j+1)$) is the smallest integer so that $n$ $a_k$ (resp. $a_l$) is a legal first letter of subword $n$. Note that this means this algorithm will be different depending on whether $w$ (and therefore $v$) is of form $w_1w_2w_3w_4$ or $w_1w_2w_5w_6$.
		
		If $n_{b(j)'}(j+1)>n_{b(j)'}(j)$, then the canceling word has moved on to a new subword, and any remaining letters in subword $u_{n_{b(j)'}(j)}(j)$, as well as any other subwords prior to $u_{n_{b(j')}(j+1)}(j)$, are no longer cancelable by Lemma \ref{MatchedCancellation}. Therefore, add each such letter to $K(j)$ and call the result $K(j+1)$ for now. It may change again in Substeps 2, 4, or 5. By assumption, the letters of $K(j)$ all commute with the letters of each subword $u_{*}(j)$. So check that the new letters of $K(j+1)$ pairwise commute, commute with the subwords $u_{n_{b(j')}(j+1)}(j)$ and later, and commute with the adding and canceling letters. If any of these fail, then the algorithm certifies the existence of an uncancelable pair. If all of these hold, then keep going.
		
		\textbf{Substep 2: Check for a bit flip.}
		
		If $n_{b(j)'}(j+1)>n_{b(j)}(j+1)$, this means that the word $b(j)'$ has moved onto a later subword than $b(j)$ is on. We will therefore have already added each nonempty $u_i(j)$ into the clique at the end of the previous step. By Lemma \ref{MatchedCancellation} the new letter of $b(j)'$ may cancel with first letter of subword ${b(j)}_{n_{b(j)'}(j+1)}$ once that subword starts. Therefore, set $b(j+1)=b(j)'$. Add the $j+1^{st}$ ``canceling letter", (which has nothing to cancel with because it is in too early a subword) into the clique to obtain $K(j+1)$. Set $u_{n_{b(j+1)}(j+1)}(j+1)$ to the $j+1^{st}$ adding letter. Check whether $K(j+1)$ is a clique that commutes with the $j+1^{st}$ canceling letter. If not, then $w$ and $v$ differ by a pair of uncancelable letters. If so, then we have computed the $j+1^{st}$ state and may continue. Skip all the remaining steps.
		
		If $n_{b(j)'}(j+1)\le n_{b(j)}(j+1)$, then set $b(j+1)=b(j)$ for now. It may change later in Substep 5.
		
		\textbf{Substep 3: Add the adding letter.}
		
		Append the $j+1^{st}$ adding letter to word $u_{n_{b(j+1)}(j)}(j)$. Call the resulting words (only one of which has changed) $u_{n}(j+1)$ for now. One of them may change again in Substep 4, or they may all be updated in Substep 5.
		
		\textbf{Substep 4: Process the canceling letter.}
		
		Determine if the canceling letter is a geodesic first letter of the word $u_{n_{b(j+1)'}(j+1)}(j+1)$. If so, cancel this letter, and move all prior letters in $u_{n_{b(j+1)'}(j+1)}(j+1)$ into $K(j+1)$. They have become uncancelable by Lemma \ref{HorocyclicNonLinking}. Check that $K(j+1)$ remains a clique that commutes with each letter in the $u_{*}(j+1)$. If so, then there is non-uncancelable pair. We have computed the $j+1^{st}$ state and may continue. Skip the remaining substep.
		
		If the canceling letter does not cancel, we must again check whether the bit needs to flip. For clarity, we describe this check in a subsequent step, although in practice it must be checked here. For now, assume that the bit does not need to flip. The canceling letter immediately becomes uncancelable and is added to $K(j+1)$. Additionally, every letter $a_m$ that the canceling letter forbids to be written next (i.e. those that commute with and precede it) can next be written after writing a letter $a_m'$ not commuting with it. So if any such letter is a geodesic first letter of $u_{n_{b(j+1)'}(j+1)}(j+1)$, it can no longer be canceled. Delete each such letter from $u_{n_{b(j+1)'}(j+1)}(j+1)$ and add them instead into the uncancelable clique $K(j+1)$. Check as usual that $K(j+1)$ is a clique commuting with each remaining letter of a $u_{*}(j+1)$. If not, then we have created an uncancelable pair. If so, then we have computed the $j+1^{st}$ state and may continue.
		
		\textbf{Step 5: Check for a bit flip.}
		
		If $n_{b(j+1)}(j+1) = n_{b(j+1)'}(j+1)$, i.e. the adding and canceling words are both on the same subword, and if additionally the canceling letter commutes with and follows each letter of $u_{n_{b(j+1)}(j+1)}(j+1)$, then no letter of this word can be written until first writing a letter not commuting with it, so the whole word becomes uncancelable. Moreover, a copy of this canceling letter, if next written onto the $j+1^{st}$ adding word, would cancel, so that the ``canceling letter" is in fact cancelable.
		
		Therefore, as in Step 2, we add each letter of $u_{n_{b(j+1)}(j+1)}(j+1)$ into $K(j+1)$, and set $u_{n_{b(j+1)}(j+1)}(j+1)$ to be equal to the $j+1^{st}$ ``canceling letter". We flip the value of the bit. We then check as usual that $K(j+1)$ is a clique. By assumption at this point it commutes with the $j+1^{st}$ ``canceling letter". If it is not a clique, then we have an uncancelable pair. If it is, then we have arrived at the $j+1^{st}$ state and can continue. 
		
		\textbf{Concluding the algorithm.}
		
		After reading in each input pair $(a_k, a_l)$ from $(w_{horosuff}, v_{horosuff})$, denote the outputs $K_{fin}$, $u_{*, fin}$, etc. If $n_{b_{fin}', fin}$ is not either $4$ or $6$ (depending on the form of the words), then the cancellation of the prefix copy of either $a_j$ or $a_i$, depending on the form of the horocyclic suffixes, will render all remaining letters of subwords $1, 2$ and either $3$ or $5$ uncancelable. We therefore delete each of these $3$ subwords, moving any remaining letters to $K_{fin}$. We check whether $K_{fin}$ is a clique commuting with each letter in the final value $u_{4, fin}$ or $u_{6, fin}$. If not, then we have an uncancelable pair. If so, then $w$ and $v$ do not differ by a uncancelable pair, the uncancelable clique is $K_{fin}$, and the word $u_{4, fin}$ or $u_{6, fin}$ is the word of cancelable letters, which must be written onto the word $b_{fin}'$.
	\end{proof}
	
	\begin{corollary} \label{DivergenceGraphEqualLengthEdgeFSM}
		If $|w_{suff}|=|v_{suff}|$, then the above algorithm is implemented by an FSM $M_{K}$.
	\end{corollary}
	
	\begin{proof}
		To see that an FSM suffices, note that if, after taking two simultaneous inputs, the total length of the words $u_{*, i+1}$ is larger than the corresponding sum of the lengths of the $u_{*, i}$ only if the second letter did not cancel. Any time this happens, the clique $K_{n+1}$ increases by at least one letter. Therefore, no state is required with longer total word length than the largest clique in the defining graph of $W_\Gamma$ (and in fact 1 letter shorter even suffices). There are therefore finitely many states. The concluding step of the algorithm is accomplished by padding each input with a blank character $*$. When the pair $(*, *)$ is read, the state updates according to that rule. Only such states are allowed to be final states of the resulting FSM.
	\end{proof}
	
	\begin{remark}
		
		A few notes are merited here about the implementation of this algorithm in the code. The reader who is uninterested in how the code works can skip this remark.
		
		First of all, for performance reasons, it is faster to include some additional data in each of the states. Specifically, to avoid having to re-compute at each step which letters do or do not cancel, we add to each state the (geodesic) first letters of of $u_{j,i}$ and of each truncated word consisting of the last $k$ letters of $u_{j,i}$ for each $k$. Since this information can be computed from $u_{j,i}$, this does not change the number of states.
		
		Second of all, the algorithm described does not guarantee that its input consists of a pair of horocyclic suffixes, or even a pair of geodesics. We could do this without trouble using Proposition \ref{CombiningFSMs}, at the cost of additional memory usage. To keep the states (slightly) understandable, we do not do this, and instead are careful only to pass arguments to the FSM that are of the desired form. This means that there are some nonsense transitions in the FSM, but does not effect the correctness of the FSM's output if passed an input of the correct form.
		
	\end{remark}
	
	Using the output of this algorithm, we can cancel every cancelable letter. That is, a priori each letter $a_k$ might cancel in a different pair of horocyclic successors $w'_k$ and $v'_k$ to $w$ and $v$. Here we show that a single successor suffices.
	
	As a result, we obtain the following algorithm.
	
	\begin{theorem}\label{DivergenceEdgeSameLengthEdgeChecker}
		
		Let $w$ and $v$ be words so that $b_{\gamma}(w)=b_{\gamma}(v)$ and $|w_{suff}|=|v_{suff}|$. Then we can determine whether $w$ and $v$ span an edge in the divergence graph in linear time with respect to $|w|$.
		
	\end{theorem} 
	
	\begin{proof}
		
            First we input $(w,v)$ into $M_k$ as described in Corollary \ref{DivergenceGraphEqualLengthEdgeFSM}. If the pair is accepted, then we use the output state to find a successor $v'$ or $w'$ to $v$ or $w$ as in Corollary \ref{MaximalCancellation}, which takes at most as many steps as the clique size of $\Gamma$. We compute $S(v)$ and $S(w')$ or $S(w)$ and $S(v')$, which takes linear time in the lengths of the inputs, and check whether there are any pair of vertices adjacent to $K(v,w)$ that do not commute and that at least one of which is in $S(v)^c\cap S(w')^c$ or $S(v')^c\cap S(w)^c$. The vertices $w$ and $v$ span a divergence graph edge if and only if there is such a pair, by Propositions \ref{DivergenceEdgeCriterion} and \ref{CliqueDifferenceAdjacency}.
	\end{proof}
	
	\begin{remark}
		
		We could in fact perform all of these checks as part of the final step of the algorithm, and obtain a regular language of edges between horocyclic suffixes of the same length. We do not do this because it would further complicate the algorithm and require tracking even more data in each state. It is also not clear that this would speed up the computations, because the language of adjacent horocyclic suffixes of equal length is very far from prefix-closed. It is somewhat hard to bound the inefficiency that arises from computing may suffix pairs that do not have close successors (in the $P^{-*}$ sense) but which have successors (in the FSM sense) who do. Instead, we take a backtracking approach similar to what we used for the Rips graph.
	\end{remark}
	
	When we have many vertices, checking each pair against one another is impractically slow, but an improvement is possible. 
	
	\begin{lemma}\label{MaximumDistanceForDivergenceEdge}
		
		If $w$ and $v$ are on the same horosphere and have close successors, then they are at distance at most $2n-2$ from one another, where $n$ is the clique dimension of $\Gamma$.
		
	\end{lemma}
	
	We note that a similar upper bound appears in Lemma 7.3 of \cite{CGSR} with no assumptions other than that the group in question is hyperbolic. However, their bound requires computing a value of $\delta$ for the group, which may be impractical.
	
	\begin{proof}
		
		Suppose $d(v,w)=2m$ (recalling that the distances between points on the same horosphere is always even). Then we have $m$ uncanceled letters in $v$ and $m$ uncanceled letters in $w$. As a result, if the pair $(v,w)$ is accepted my the FSM $M_K$, it outputs an uncancelable clique of size at least $m$. In order for there to be any letters at all that commute with the entirety of $K(v,w)$, $m$ must be less than the maximal clique size in $\Gamma$.
	\end{proof}
	
	We can do even better than check the word $w$ against $ww'$ where $w'$ ranges across words of length $2n-2$, because most of these $ww'$ are not on the correct horosphere. As before, we will delete $n-1$ letters, followed by adding $n-1$ letters in. First we need a lemma to show that a version of the backtracking process in Lemma \ref{DeletionPreservesShortlexSuffixes} will work for horocyclic suffixes.
	
	\begin{lemma}\label{DeletionPreservesHorocyclicSuffixes}
		
		If $w$ is a horocyclic suffix that can be rearranged to end with $a_l$, then deleting the last copy of $a_l$ from $w$ results in a word that is again a horocyclic suffix.
		
	\end{lemma}
	
	\begin{proof}
		
		If this last copy of $a_l$ is in $w_1$ and $w_2$, this is a direct application of Lemma \ref{DeletionPreservesShortlexSuffixes}.
		
		If the deleted copy of $a_l$ is in $w_3w_4$, then by Lemma \ref{DeletionPreservesShortlexSuffixes} the resulting word is still a concatenation of a word in $ShortLex|_{Star_<(a_j)}$ with one in $ShortLex$. So the only remaining case is that the letter $a_l$, once deleted, allows a letter in $Star(a_i)$ to reach the start of $w_3$, or a letter in $Star_\le (a_j)\cup\{a_i\}$ to reach the beginning of $w_4$, or a letter of $Star_<(a_i)\cap Star(a_j)$ to reach the beginning of $w_3w_4$. Call this other letter $a_k$. In each case $a_k$ must follow $a_l$ and not commute with it, so that $a_l$ is not a last letter of $w$.
	\end{proof}
	
	As in the Rips graph case, we again need a geodesic horocyclic suffix machine which will describe the letters that we can insert into a word.
	
	\begin{lemma}\label{GeoHoroSuffMachine}
		
		There are finite-state machine $M_{Geo1234}$ and $M_{Geo1256}$ whose accepted language consists of word $w=w_1w_2w_3w_4$ and $w_1w_2w_5w_6$, where each $w_i$ is a geodesic word and which satisfies the same restricted alphabet and rearrangement rules as the languages of $M_{1234}$ and $M_{1256}$, (i.e. $w_2$ consists of letters commuting with $a_i$ and $a_j$, and preceding $a_i$ but not $a_j$, etc.)
		
	\end{lemma}
	
	\begin{proof}
		
		Construct the machines $M_{1234}$ and $M_{1256}$, but replace each use of the shortlex machine $M_{lex}$ with the geodesic machine $M_{Geo}$ instead.
	\end{proof}
	
	Again as in the Rips graph case, insertion is an efficient operation.
	
	\begin{lemma}\label{HorocyclicInsertion}
		
		Let $w=w_1w_2w_3w_4$ (resp. $w=w_1w_2w_5w_6)$ be a horocyclic suffix and $a_l$ be permitted by $M_{Geo1234}(w)$ (resp. $M_{Geo1256}(w)$. The horocyclic suffix equivalent to $wa_l$ can be computed in $O(|w|)$ steps.
		
	\end{lemma}
	
	\begin{proof}
		
		We read $w$ from right to left, remembering the last time we read a letter commuting with and following $a_l$. The first time we read a letter failing to commute with $a_l$, or when we reach the beginning of the first word that allows $a_l$ in its language, we stop. We then insert $a_l$ into the last memorized location.	
	\end{proof}
	
	\begin{proposition} \label{DivergenceGraphSameSuffixLengthEdgeGenerator}
		
		Let $w_{horosuff}$ be a horocyclic suffix of length $m$, and let $n$ be the largest size of a clique in $\Gamma$ as before. Fix a value $k$ for the Busemann function compatible with the parity of $|w_{horosuff}|$ (i.e. if $k$ is even, $w_{horosuff}$ should be an even horocyclic suffix). The set of horocyclic suffixes $v_{horosuff}$ of words $v$ such that $b_\gamma(v)=b_\gamma(w)=k$ and $w$ and $v$ have close horocyclic successors can be generated by an algorithm whose runtime is linear in $m$ and exponential in $n$.
	\end{proposition}
	
	\begin{proof}
		
		We will mimic the proof of Lemma \ref{RipsEdgesSameSuffixLength} to generate a list of candidate words $v$ of size exponential in $n$, and then check them all using the algorithm described in Theorem \ref{DivergenceEdgeSameLengthEdgeChecker}.
		
		Therefore, we will use an iterated version of the process described in subsection \ref{subsec:RipsEdgesSameLength}. We first backtrack from $n-1$ steps to get a set of horocyclic suffixes of length $m-n+1$. If $m<n-1$, then we backtrack to the empty suffix. We then use $M_{Geo1234}$ or $M_{Geo1256}$ to choose a permitted letter to insert into the backtracked word, which takes linear time by Lemma \ref{HorocyclicInsertion}. We do this $n-1$ or $m$ times, whichever is less, to obtain the list of all horocyclic suffixes of the same length and the same form, at distance at most $2n-2$. We therefore obtain a set of candidate words $v$ whose size is exponential in $n$, and we check them with Theorem \ref{DivergenceEdgeSameLengthEdgeChecker}, which takes linear time in their length.
	\end{proof}

	\subsection{Edges between words of different suffix length}
	\label{subsec:DivergenceEdgesDifferenLength}
	
	In this section we will describe edges between words $w$ and $v$  where $|w_{suff}|>|v_{suff}|$. This will turn out to break down into two cases: one where $|w_{suff}|=|v_{suff}|+1$, and one where $|w_{suff}|>|v_{suff}|+1$. The reason for this is that if the prefix lengths of $w$ and $v$ differ by at least $2$, then there is automatically a pair of non-commuting prefix letters early in the difference $w^{-1}v$, which will make it very hard for the two to have close successors. We will address the case where the difference is at least $2$ first. 
	
	To put a word in horocyclically shortlex form, we first multiply it by a sufficiently high power of $a_ja_i$ so that its prefix becomes positive. If we multiply both $w$ and $v$ by $(a_ja_i)^k$ so that the become horocyclically shortlex, and $b_\gamma(w)=b_\gamma(v)$, then the shorter suffix on $v$ corresponds to a longer prefix. That is, once put in horocyclically shortlex form, the two are the same length.
	
	We first give a lemma to check whether the word $w$ has close successors with any word of shorter suffix.
	
	\begin{lemma}
		
		Let $w_{horosuff} = w_1w_2w_3w_4$ (resp. $w_1w_2w_5w_6$). If $w_3$ (resp. $w_5$) is nonempty, or if $w_4$ (resp. $w_6$) contains a letter not commuting with $a_i$ (resp. $a_j)$, then $w$ does not have close successors with any word $v$ with shorter suffix.
		
	\end{lemma}
	
	\begin{proof}
		
		If $w$ is of form $w_1w_2w_3w_4$ (resp. $w_1w_2w_5w_6$), then such a word $v$ has a longer prefix containing at least one uncanceled copy of $a_i$ (resp. $a_j$). The first letter of $w_3$ (resp. $w_5$) is assumed not to commute with $a_i$ (resp. $a_j$), so if it is present, there is an uncancelable pair. Similarly, such a letter in $w_4$ (resp. $w_6$) again creates an uncancelable pair. 	
	\end{proof}
	
	If the difference is at least two letters, then we have a further restriction.
	
	\begin{lemma}
		
		Suppose $w$ and $v$ are words on the same horosphere, and \newline $|w_{suff}|\ge |v_{suff}|-2$. Then $w$ and $v$ have close successors only if $w_{HoroSuff}=w_1w_2$.
		
	\end{lemma}
	
	\begin{proof}
		
		We have at least one uncanceled copy of both $a_i$ and $a_j$, so we need to be able to write the first one after $w$. As in Remark \ref{DescribingPredecessors}, the only way for this to happen is for $w_3$ to be empty and $w_4$ to consist of letters commuting with and preceding $a_i$.
		
		The first letter $a_k$ of $w_4$ (if existent) therefore cannot commute with $a_j$, because if so that letter must be in $w_1$ or $w_2$ instead. But then this copy of $a_k$ and the remaining uncanceled copy of $a_j$ in $(a_ja_i)^mv$ would be a non-commuting pair. So $w_4$ is empty.
	\end{proof}
	
	As a result, we obtain a linear-time check for each horocyclic suffix that tells us whether it is adjacent to words with shorter suffixes, and a constant-time check that tells us whether it is adjacent to words with suffixes more than $1$ character shorter. The machinery of Lemmas \ref{DeletionPreservesHorocyclicSuffixes} and \ref{HorocyclicInsertion} will allow us to find the list of shorter nearby horocyclic suffixes efficiently. So we need a version of the machine $M_K$ that works for horocyclic suffixes of different lengths, and a corollary that bounds how far apart two such words can be and still have close successors.
	
	\begin{lemma} \label{DivergenceGraphDifferentLengthEdgeFSM}
		Let $|w_{suff}|>|v_{suff}|$, and let $b_\gamma(w)=b_\gamma(v)$. Suppose $v_{HoroSuff}=v_1v_2v_3v_4$ (resp. $v_{HoroSuff}=v_1v_2v_5v_6$). If so, we can either certify that the two do not have close successors or compute $K(v,w)$ and the word of cancelable letters using an FSM.
	\end{lemma}
	
	\begin{proof}
		
		We use a modification of the FSM described in Corollary \ref{DivergenceGraphEqualLengthEdgeFSM}. Specifically, we will include the uncanceled prefix letters in the input. Therefore, the steps of calculating a new state will be precisely the same. The only difference is that the machine will recognize slightly different subwords. The ``third subword" will consist of all but the first uncanceled prefix letter, while the ``fourth subword" will be the concatenation $v_5a_iv_6$ or $v_3a_jv_4$, since these are the letters that $w_4$ or $w_6$ (if present) can cancel with. That is, the only change is the update rule for the integers $n_{w,i}$ and $n_{v, i}$. 
	\end{proof}
	
	As in the equal-length case, each time the total cancelable letters increase, there is a new uncanceled letter, so we obtain Corollary \ref{MaximumDistanceForDivergenceEdge} in this case as well. Then, similarly to Proposition \ref{DivergenceGraphSameSuffixLengthEdgeGenerator}, we can enumerate all such edges exiting a word $w$.
	
	\begin{proposition}\label{DivergenceGraphDifferentSuffixLengthEdgeGenerator}
		
		Let $w_{HoroSuff}$ be a horocyclic suffix of a word $w$ on the horosphere $b_{\gamma}^{-1}(k)$ such that $w_3$ or (resp. $w_5$) is empty, and $w_4$ or (resp. $w_6$) consists entirely of words commuting with $a_i$ (resp. $a_j$). We can enumerate the horocyclic suffixes $v_{HoroSuff}$ of words $v$ on the horosphere $b_{\gamma}^{-1}(k)$ such that $w$ and $v$ have close successors by an algorithm whose time in linear in $|w_{HoroSuff}|$ and exponential in $Clique(\Gamma)$.
		
	\end{proposition}
	
	\begin{proof}
		
		It takes constant time to check whether $w_4$ (resp. $w_6$) is empty. If so, then $w_{HoroSuff}=w_1w_2$ is of length shorter than the largest clique in $\Gamma$. As such, we take the set of horocyclic suffixes of length at most $n-2$, where $n=clique(\Gamma)$, and compare them to $w$ using the FSM in Lemma \ref{DivergenceGraphDifferentLengthEdgeFSM}. We then check the output using Proposition \ref{DivergenceEdgeCriterion}. All of these operations take linear time in $|w|$, and the list of candidates is exponential in $n$.
		
		If $w_4$ (resp. $w_6$) is not empty, then we only need to check suffixes 1 letter shorter. We apply Lemma \ref{DeletionPreservesHorocyclicSuffixes} to delete $n-1$ or $|w_{suff}|$ letters, whichever is lesser, and then use $M_{Geo1234}$ or $M_{Geo1256}$ and Lemma \ref{HorocyclicInsertion} to insert $n$ more characters. This creates a list of candidates whose length is exponential in $n$. We check these candidates as before, which takes linear time in their length.
	\end{proof}
	
	Combining everything we conclude the following theorem.
	
	\begin{theorem} \label{DivergenceGraphGenerationSpeed}
		
		Let $\Gamma$ be a graph satisfying the standing assumptions together with the condition of Proposition \ref{EveryStateLarge}. Let $a_i$ and $a_j$ two non-adjacent letters of $\Gamma$, and let \newline $\gamma=(a_ia_j)^\infty$. Take $k$ to be an integer. Then $n$ vertices of the divergence graph on horosphere $b_\gamma^{-1}(k)$ and the edges between them can be generated by an algorithm whose runtime is $O(n\log(n))$.
	\end{theorem}
	
	\begin{proof}
		
		First we use Proposition \ref{DivergenceGraphVertexFSM} to generate the list of horocyclic suffixes of length at most $m$. The number $n$ of these grows exponentially in $m$, so that this requires $n$ operations, each with approximately $\log(n)$ steps. We take these to be the vertices of the graph.
		
		For each such vertex, the edges to horocyclic suffixes of the same length can be enumerated using the algorithm in Proposition \ref {DivergenceGraphSameSuffixLengthEdgeGenerator} in $O(\log(n))$ steps, and the edges to horocyclic suffixes that are shorter can be generated by Proposition \ref{DivergenceGraphDifferentSuffixLengthEdgeGenerator} in $O(\log(n))$ steps. 
	\end{proof}
	
	\subsection{Geometry of divergence graphs}
	\label{subsec:DivergenceGraphGeometry}
	
	In this subsection we describe the geometric corollaries of the last two sections. First of all, as mentioned in the introduction, divergence graphs inherit some distortion properties from the related Rips graphs.
	
	\begin{proposition}
		
		Let $w$ and $v$ be on the horosphere $b_{\gamma}^{-1}(k)$ and let $d_H$ denote the distance in the divergence graph on the horosphere. Then $d_H(w, v)\ge \frac{1}{O_\Gamma(1)}2^{\frac{d(w, v)}{2\delta}}$.
		
	\end{proposition}
	
	Note that the proof of this proposition makes no use of the assumption that the divergence graph is defined on the whole horosphere other than for $w$ and $v$ to be vertices of the divergence graph. So if the assumption in Proposition \ref{EveryStateLarge} does not hold, but $w$ and $v$ are large or pre-large states, the conclusion will still be true.
	
	\begin{proof}
		
		By Lemma \ref{MaximumDistanceForDivergenceEdge}, any two vertices that span an edge of the divergence graph are at distance at most $2n-2$, where $n$ is the clique dimension of the graph $\Gamma$. Therefore, the divergence graph is a subgraph of the $2n-2$-Rips graph. Hence any path in the divergence graph is at least the minimum length of a path in the $2n-2$-Rips graph. The result then follows from Proposition \ref{HorosphericalDistortionLowerBound} together with Proposition \ref{RipsGraphBiLipschitzEquivalence}.
	\end{proof}
	
	Similarly, comparison to the $2n-2$-Rips graph gives a polynomial growth statement without any need for the assumption in Proposition \ref{EveryStateLarge}.
	
	\begin{corollary}\label{DivergenceGrowth}
		
		Let $w_0$ be the word of empty suffix on the horosphere $H_k$ and $P$ the polynomial such that $|B_{\text{Rips}}(w_0, r)|\le P(r)$ where the ball is taken with respect to the metric on the $2n-2$-Rips graph. Then the same holds for $B_{\text{div}}(w_0, r)$, the ball with respect to the divergence graph's path metric.
		
	\end{corollary}
	
	\begin{proof}
		
		The vertex set of the divergence graph is a subset of that of any Rips graph, and the metric on the divergence graph dominates the metric on the $2n-2$-Rips graph. Therefore, $B_{\text{div}}(w_0, r)$ is a subset of $B_{\text{Rips}}(w_0, r)$.
	\end{proof}
	
	Finally, we address an upper bound for the distortion of the divergence graph. While it is not clear how to obtain such a bound in general, we provide a sufficient condition for the distance in the divergence graph to be no more than a fixed multiple of the distance in the $2$-Rips graph.
	
	\begin{proposition}\label{DivergenceDistortionUpperBound}
		
		Let $\Gamma$ be a graph that is not a clique, without induced squares or separating cliques, such that every vertex has another vertex at distance at least $3$ away, and such that for any pair of adjacent vertices in $a_{k}$ and $a_l$ in $\Gamma$, and any $a_m$ commuting with both of them, there is a maximal clique $K$ containing $a_k$, and $a_l$ but not $a_m$. Then there is a constant $L$ such that for any two words $w$ and $v$ on the same horosphere such that $d(w,v)=2$, there is a path of length at most $L$ between $w$ and $v$. 
		
	\end{proposition}

	Note that this somewhat artificial condition on adjacent vertices will be satisfied, e.g., for the 1-skeleton of any manifold. 
	
	The proof will rely on the following lemma.
	
	\begin{lemma}
		
		Let $\Gamma$ be a graph satisfying these conditions, and let $w$ and $v$ be two words on the same horosphere differing by a pair of commuting letters. Then $w$ and $v$ are adjacent in the divergence graph.
		
	\end{lemma}
	
	\begin{remark}
		
		One can prove more generally that if any sub-maximal clique $K$ in $\Gamma$ and vertex $a_m$ adjacent to each vertex of $K$, there is a maximal clique $K'$ containing $K$ and not $a_m$, then two words on a horosphere are adjacent if and only if their difference is a clique. We will not present a proof here because we do not use it for sub-maximal cliques of size greater than 2. However, the proof follows along the same lines.
		
	\end{remark}
	
	\begin{proof}
		
		We will divide into cases based on the lengths of the suffixes. \newline WLOG $|w_{suff}|\ge |v_{suff}|$.
		
		\textbf{Case I}: $w$ and $v$ have equal-length suffixes.
		
		Let $w^{-1}v=_{Geo}w_{horosuff}^{-1}v_{horosuff}=_{Geo}a_ka_l$, where $a_k$ comes from $w$ and $a_l$ from $v$, and $a_k$ and $a_l$ commute.
		
		\textbf{Case I(a)}: Neither $a_k$ nor $a_l$ is the last letter of $w$ or $v$.
		
		In this case, both letters are uncancelable. In order to avoid linked cancellation, the last letters of $w_{horosuff}$ and $v_{horosuff}$ must be the same letter $a_m$ and commute with both $a_k$ and $a_l$.
		
		By assumption, the set $S$ of vertices of $\Gamma$ adjacent to $\{a_k, a_l\}$ contains no vertex adjacent to each other vertex of $S$. In particular, $m$ is not adjacent to all of $S$. Now, the state of the words $w$ and $v$ forbids a subset of the letters adjacent to $a_m$, in particular therefore not all of $S$. So if $a_o$ is some letter of $S$ not forbidden in either state, then $wa_ma_oa_ma_o...$ and $va_ma_oa_ma_o...$ stay close for all time.
		
		\textbf{Case I(b)}: One of $a_k$ or $a_l$, but not both, is the last letter of $w_{horosuff}$ or $v_{horosuff}$.
		
		WLOG suppose that $a_k$ is at the end of $w_{horosuff}$. The letter before $a_k$ must cancel with some letter in $v$, and by non-linking, it must cancel with the last letter. The same holds for the each other letter preceding $a_k$. So there are words $w'$ and $w''$ so that $w''$ commutes with $a_l$ and such that $w_{horosuff} = w'w''a_k$ while $v_{horosuff}=w'a_lw''$. Since $a_k$ is written after $w'w''$, it follows that the state of $w'a_lw''$ forbids writing $a_k$ only if $a_l$ forbids $a_k$ (i.e. $a_l$ commutes with and follows $a_k$) and if $a_k$ commutes with (and necessarily follows) each letter of $w''$. But since $w''$ commutes past $a_l$ to cancel, the first letter of $w''$ follows $a_l$ in the shortlex order. As a result, the letter $a_k$ is permitted after $w'a_lw''$, achieving full cancellation.
		
		From here, one takes any letter $a_m$ adjacent to $a_l$ but not $a_k$. One exists by the assumption that there are no separating cliques, so that the set of neighbors of any point cannot be a clique. Then $w'w''a_ka_ma_ka_m...$ and $w'a_lw''a_ka_ma_ka_m...$ stay close for all time.
		
		\textbf{Case I(c)}: Both $a_k$ and $a_l$ are the last letters of the $w_{horosuff}$ and $v_{horosuff}$. 
		
		In this case, WLOG suppose that $a_k$ follows $a_l$ in the shortlex order. Then we apply the same argument as in case I(b).		
		
		\textbf{Case II}: $|w_{suff}|=|v_{suff}|+1$
		
		In this case, in order to be on the same horosphere, $v_{pref}$ is 1 more positive than $w_{pref}$. WLOG suppose $va_ia_k=_{Geo}$, where $a_i$ commutes to the prefix. Then $w_{suff}$ consists of letters commuting with $a_i$. We can therefore apply the same arguments as in cases I(a) and I(b) as follows.
		
		If neither $a_i$ nor $a_k$ is the final letter of $v$, then as in I(a), the final letter $a_l$ of $v$ cancels with the final letter of $w$, so that $a_l$ commutes with $a_i$ and $a_k$, rendering $\{a_i, a_k\}$ a sub-maximal clique. If instead either $a_i$ or $a_k$ is the final letter of $v$, one argues as in case I(b) that that letter must be permitted by the state of $w$. The argument in this case is slightly easier, because, since $a_k$ and $a_i$ commute by assumption, any favorable rearrangement of $wa_k$ or $wa_i$ would have to be available in $v$. 
	\end{proof}
	
	With this Lemma, we can now prove Proposition \ref{DivergenceDistortionUpperBound}.
	
	\begin{proof}
		
		Let $wa_ka_l=_{Geo}v$, where $a_k$ and $a_l$ are not assumed to commute, and where $a_k$ cancels with a letter in $w$. We must bound the distance between $w$ and $v$ in the divergence graph.
		
		Assume that, after reduction, $w_{horosuff}a_k$ contains a letter which fails to commute either with $a_i$ or $a_j$. This covers all but finitely-many pairs $(w,v)$, so if a bound on the distance between $w$ and $v$ in the divergence graph is found in this case, it will imply a bound overall\footnote{Although it is not necessary, we can give an explicit upper bound in the other case as well. If $|w_{horosuff}|=|v_{horosuff}|$, then there is a path from $a_k$ and $a_l$ to a letter adjacent to either $a_i$ or $a_j$ due to Lemma \ref{RipsGraphTraversability}. Choosing whichever letter will allow us to shorten the suffix then provides a path of length at most $2D+2$. If instead the two have different length suffixes (and WLOG $|w_{suff}|=|v_{suff}|+1$), then $a_l$ is either $a_i$ or $a_j$. In this case there is a path of length at most $D+1$ that first moves $a_k$ to a letter $a_m$ allowing the prefix to be lengthened by $a_l$, and then performs that lengthening by deleting $a_m$}. In such a case, the set of letters $a_m$ such that $wa_ka_m$ is not on the correct horosphere is a clique $K$ in $\Gamma$. By assumption $\Gamma\setminus K$ is connected with diameter at most $D$. So we choose a path of length at most $D$ in this graph between $a_k$ and $a_l$,  and, using the previous Lemma, use an edge in the divergence graph for each step along this path.		
	\end{proof}
	
	\begin{corollary}
		
		If $\Gamma$ is as in Proposition \ref{DivergenceDistortionUpperBound}, $d_H$ denotes the distance in the divergence graph on a horosphere, and $d$ denotes the distance in the Cayley graph, then $d_H(w,v)\le O_\Gamma(1)(D+1)^{d(w,v)}$.
		
	\end{corollary}
	
	\section{Examples}
	\label{sec:Pictures}
	
	Here we present various outputs of the above algorithms. These images come from converting these graphs to graphml structure, followed by rendering in Mathematica. All of these graphs but the last one can be computed in reasonable time on a personal laptop.
	
	\begin{figure} [H]
		
		\centering
		
		\includegraphics[scale=0.3]{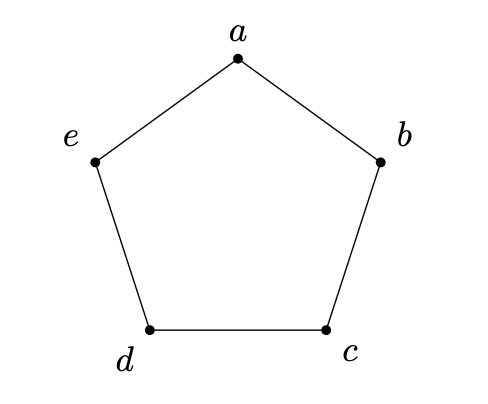}
		
		\includegraphics[scale=0.7]{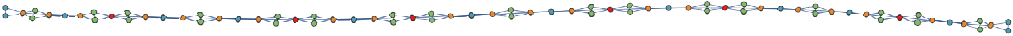}
		
		\caption{Above: the defining graph for a RACG that is virtually a surface group. Below: the 2-Rips graph on words of suffix length at most 4 in this group. The defining ray is $(ac)^\infty$.}
        \label{fig:SurfaceRipsGraph}
		
	\end{figure}
	
	In Figure \ref{fig:SurfaceRipsGraph} we see a portion of a horosphere in a virtual surface group. As expected, it looks like a fattened line. Every point is coarsely a cut point, which matches ones intuition for Fuchsian groups. The growth rate of this graph appears linear. One can show by hand that the graph will remain a fattened line as we generate more of it.
	
	\begin{figure} [H]
		
		\centering
		
		\includegraphics[scale=0.7]{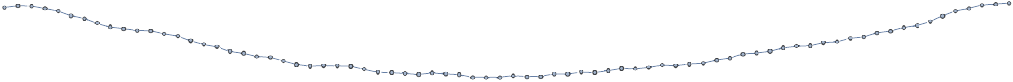}
		
		\caption{The divergence graph on the same vertex set.}
            \label{fig:SurfaceDivergenceGraph}
	\end{figure}
	
	In Figure \ref{fig:SurfaceDivergenceGraph} we see an example where the divergence graph provides a cleaner picture than the related Rips graph. Rather than a fattened line, we get a line on the nose. Again one can prove by hand that this will continue forever.
	
	\begin{figure} [H]
		\centering
		
		\includegraphics[scale=0.2]{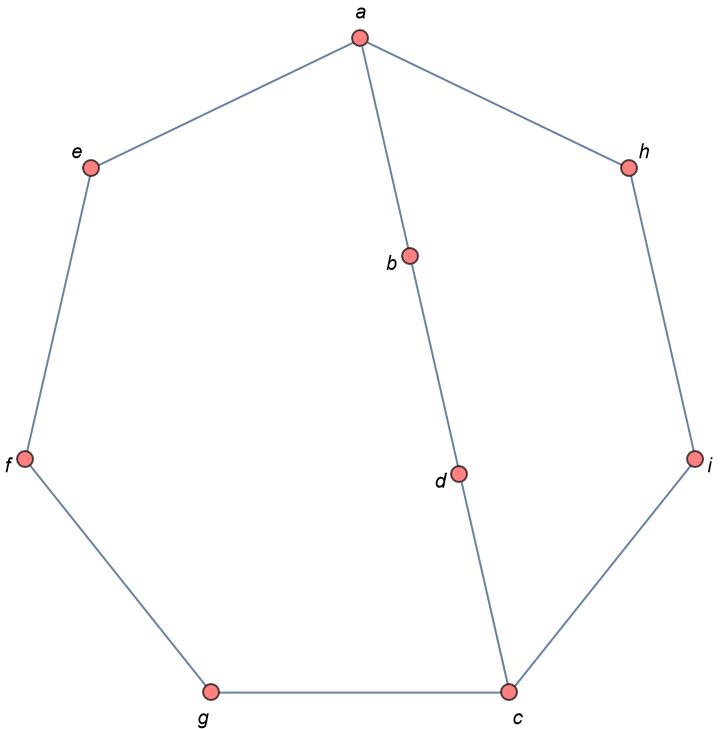}		\includegraphics[scale=0.6]{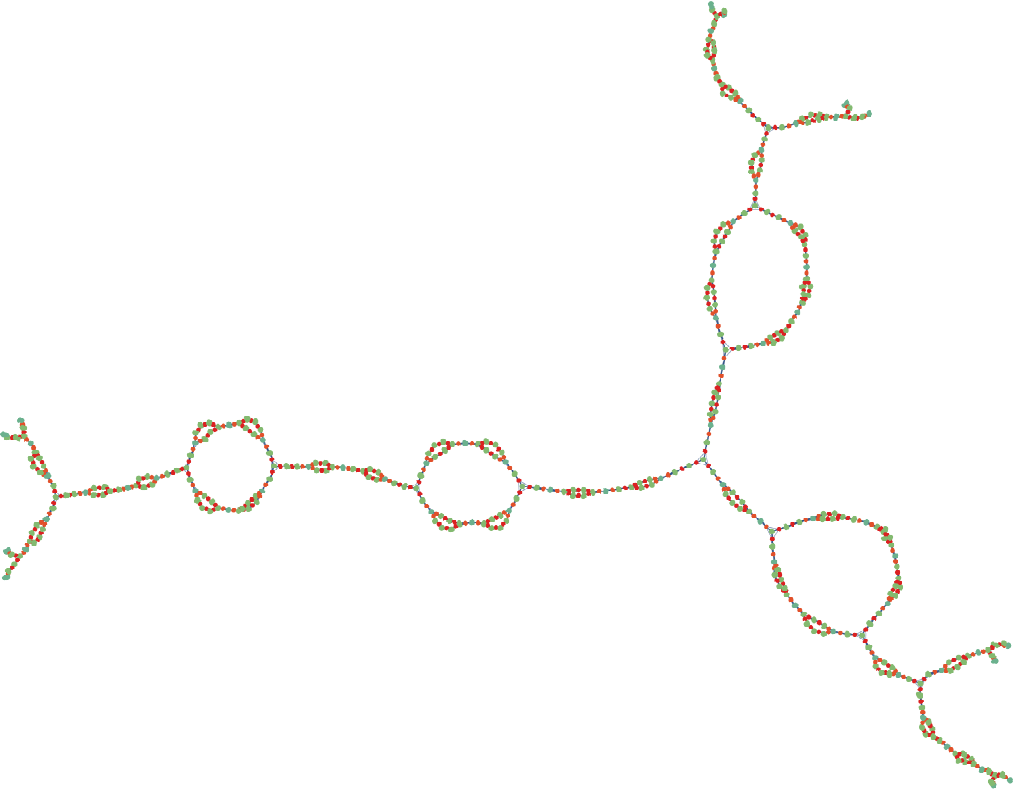}
		
		\caption{Above: the defining graph for a RACG that is virtually a surface amalgam. Below: the 2-Rips graph on words of suffix length at most 4 in this group. The defining ray is $(ac)^\infty$.}
        \label{fig:BranchedSurfaceRipsGraph}
		
	\end{figure}
	
	In Figure \ref{fig:BranchedSurfaceRipsGraph}, the group $W_\Gamma$ is virtually the fundamental group of a surface amalgam in the sense of \cite{Stark}. This class of groups was studied further in \cite{ThetaGraphs}. One sees again the indication of linear growth, which can the be proven by hand. Notice the nested branching structure. Every point is again coarsely a cut point, but how large a neighborhood needs to be taken in order for the complement to be disconnected depends on how deep into a branch the point is. The center point of the threefold branch is the word with empty suffix. As discussed in Proposition \ref{HorosphericalDistortionUpperBound}, this matches the boundary, in which the endpoints of the line $...acaca...$ disconnect $\partial W_\Gamma$ into 3 pieces. 
	
	\begin{figure} [H]
		
		\centering
		
		\scalebox{-1}[1]{\includegraphics[scale=0.7]{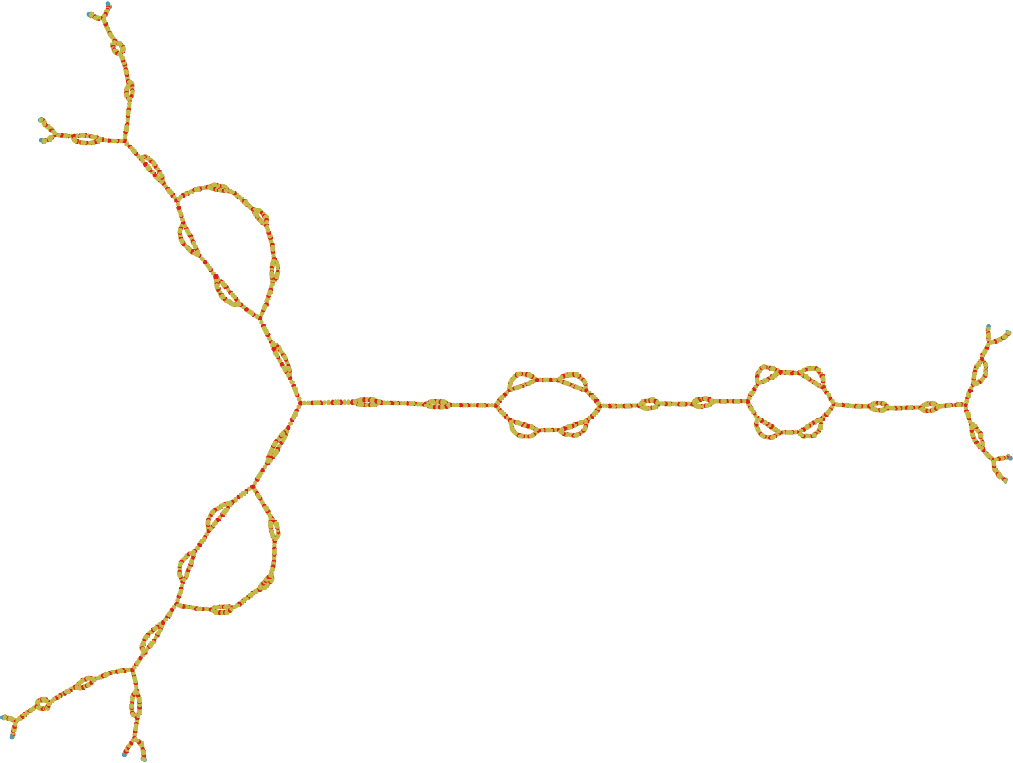}}
		
		\caption{The divergence graph on the same vertex set.}
            \label{fig:BranchedSurfaceDivergenceGraph}
		
	\end{figure}
	
	In Figure \ref{fig:BranchedSurfaceDivergenceGraph} we see again a cleaner divergence graph. What were complete graphs have again been replaced with paths.

    \begin{figure} [H]
        \centering
        \includegraphics[scale=0.3]{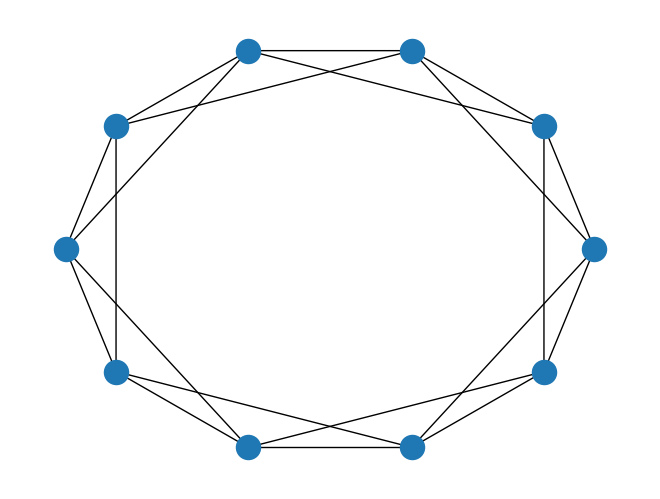}
        
        \includegraphics[scale=0.7]{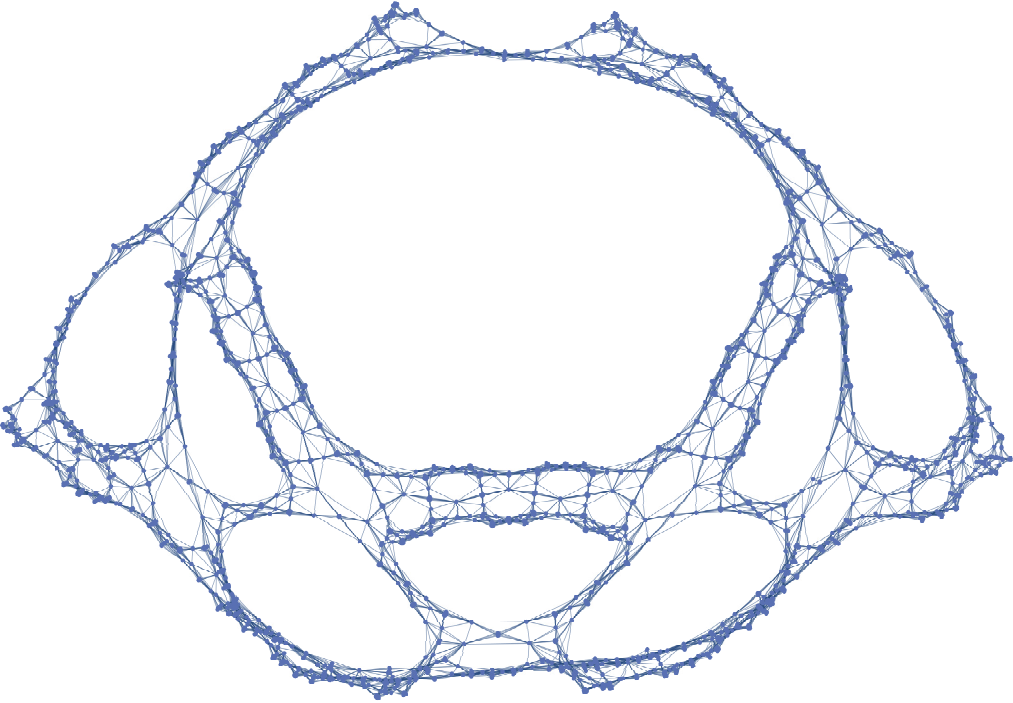}
        \caption{Above: the defining graph for a RACG that is virtually the fundamental group of a hyperbolic 3-manifold with boundary. Below: the Rips graph on words of suffix length at most 4 in this group.}
        \label{fig:SierpinskiCarpetRipsGraph}
    \end{figure}

    Figures \ref{fig:SierpinskiCarpetRipsGraph} and \ref{fig:SierpinskiCarpetDivergenceGraph} concern a group that is (virtually) the fundamental group of a hyperbolic 3-manifold with totally-geodesic boundary \cite{KapovichKleiner}. Its boundary is the Sierpinski Carpet \cite{SwiatkowskiSierpinskiCarpet}. We therefore expect to see fractally-nested holes. In Figure 5 we observe holes of different sizes, but there is some visual clutter. In Figure 6, we see a cleaner divergence graph in which more smaller holes are visible.

    \begin{figure}
        \centering 
        \scalebox{1}[-1]{\includegraphics[scale=0.7]{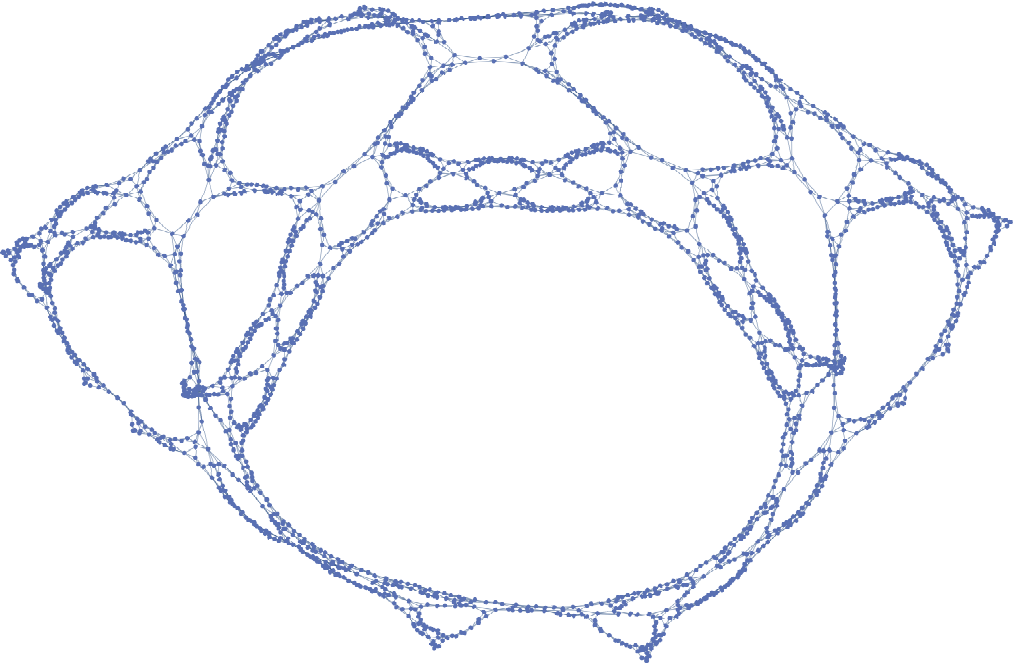}}
        \caption{The divergence graph on the same vertex set. Note the holes of 4 different scales that are visible.}
        \label{fig:SierpinskiCarpetDivergenceGraph}
    \end{figure}

    The horosphere shown in Figure \ref{fig:PontryaginSphereRipsGraph} should approximate the Pontryagin Sphere, which is a tree of manifolds in the sense of \cite{Swiatkowski}. This space arises by starting with a torus, taking a countably infinite connect sum with other tori over a dense family of discs, and then repeating the process infinitely on the newly-added tori. This graph is visually very hard to parse, but one sees the indication, as expected, of no coarse cut points. It remains to be investigated what the numerical properties of this graph may tell us about the geometry of the boundary, but with the algorithms described in this paper, many such numerical properties should be straightforward to calculate.

	\begin{figure}[H]
		
		\centering
		
		\includegraphics[scale=0.25]{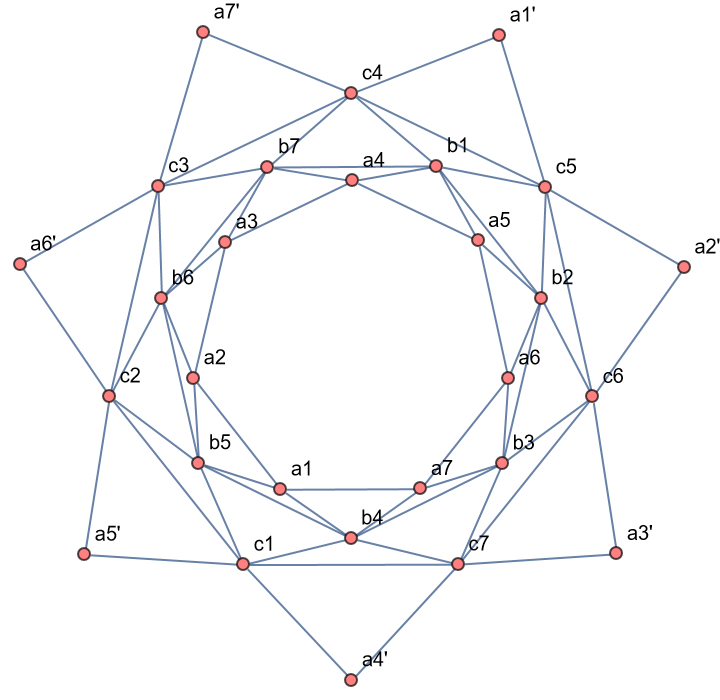} \includegraphics[scale=0.65]{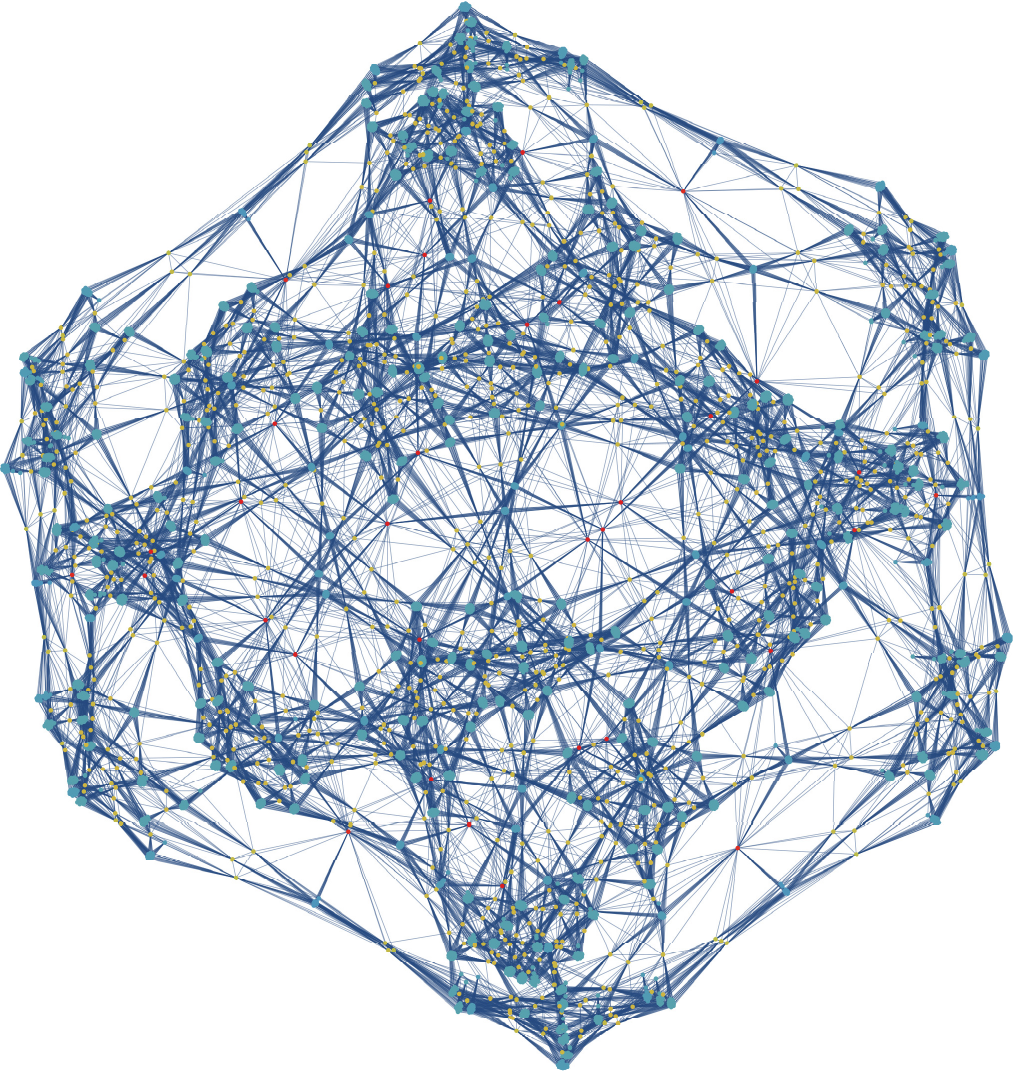}
		
		\caption{Above: the defining graph for a hyperbolic pseudomanifold group. The inner and outer letters should be glued together, so that the graph is a flag triangulation of a torus. Below: the 2-Rips graph on words of suffix length at most 3 in this group.}

        \label{fig:PontryaginSphereRipsGraph}
		
	\end{figure}

	\begin{figure}
		
		\centering
		
		\includegraphics[scale=0.7]{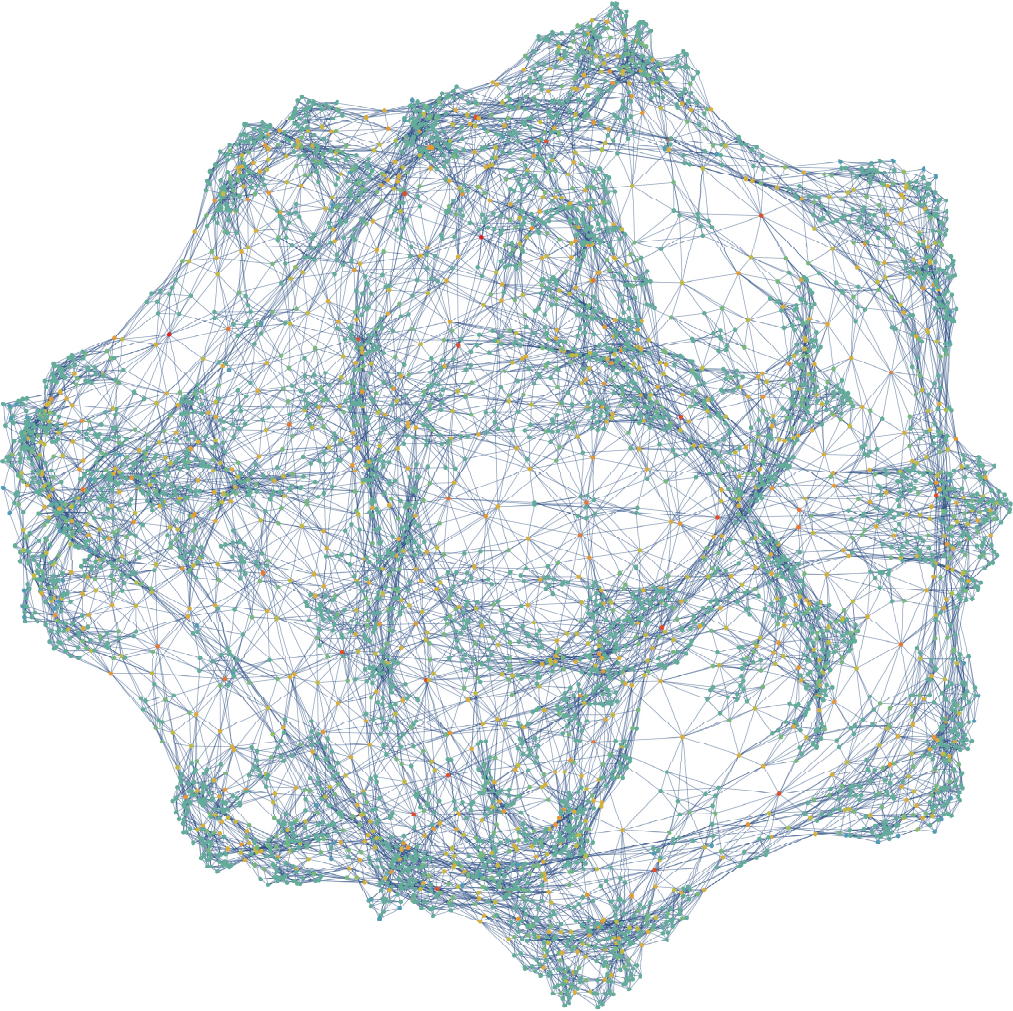}
		
		\caption{The divergence graph on the same vertex set.}

            \label{fig:PontryaginSphereDivergenceGraph}
		
	\end{figure}
	
	In Figure \ref{fig:PontryaginSphereDivergenceGraph}, we see again a cleaner version of Figure \ref{fig:PontryaginSphereRipsGraph}. However, the graph is still too complicated to extract any obvious information, other than to remark that there is some indication of surfaces branching off of one another in a tree-like manner.
	
	\textbf{acknowledgements}
		Daniel Levitin would like to thank Mark Pengitore for first introducing them the study of horospheres in hyperbolic groups, and Corey Bregman for suggesting the study of RACGs with Pontryagin Surface boundaries. They would also like to thank Enrico Le Donne for an interesting conversations about distortions of horospheres. 

            The authors would like to thank the anonymous referee for several comments that improved the exposition and for pointing us to a paper discussing the geodesic machine.
		
		The authors would also like to thank the Madison Experimental Mathematics lab, and its director \c{C}a\u{g}lar Uyanik and associate director Grace Work, for facilitating this project. Daniel Levitin would like to thank them especially for matching them with inquisitive and motivated undergraduates to work with. The whole group would also like to thank this project's faculty mentor, Tullia Dymarz, for her guidance.
	
	\bibliographystyle{plain}
	\bibliography{Horospheres_in_RACGs_Paper.bib}

\end{document}